\documentclass[reqno, xcolor=pdftex,letterpaper]{amsart}

\usepackage{longtable} 
\usepackage{hyperref}
\hypersetup{
    colorlinks=true,    % Colors the text, removes the boxes/squares
    citecolor=green,      % Citations [1] are Red
    linkcolor=red,     % Internal links (Eq. 1.1) are Blue (change to 'red' if you want them red too)   % URLs are Blue   % Ensures no border is drawn
}
\usepackage[T1]{fontenc}
\usepackage[utf8]{inputenc}
\usepackage[english]{babel}
\usepackage{textcomp}
\usepackage{dsfont}
\usepackage{latexsym}
\usepackage{amssymb}
\usepackage{amsthm}
\usepackage{amsmath}
\DeclareMathAlphabet{\mathpzc}{OT1}{pzc}{m}{en}
\usepackage{yfonts}
\usepackage{xfrac}
\usepackage{newlfont}
\usepackage{graphicx}
\usepackage{mathtools}
\usepackage{comment}
\usepackage{indentfirst}
\usepackage{braket}
\usepackage{mathrsfs}
\usepackage{xcolor} 
\usepackage{csquotes} 
\usepackage{tikz}
\usetikzlibrary{decorations.pathreplacing}

\usepackage{wela} %\wela
\usepackage{calligra}

\usepackage{enumerate}

\usepackage{etoolbox}

\usepackage{scalerel}[2014/03/10]
\usepackage[usestackEOL]{stackengine}
\newcommand{\dashint}{\,\ThisStyle{\ensurestackMath{%
			\stackinset{c}{.2\LMpt}{c}{.5\LMpt}{\SavedStyle-}{\SavedStyle\phantom{\int}}}%
		\setbox0=\hbox{$\SavedStyle\int\,$}\kern-\wd0}\int}

\newcommand{\Car}{\mathrm{C}}
\newcommand{\RC}{\mathrm{RC}}
\newcommand{\Samp}{\mathrm{S}}

\DeclareMathOperator{\supp}{supp}

\DeclareMathOperator{\tr}{tr}
\DeclareMathOperator{\ad}{ad}

\renewcommand{\Re}{\mathrm{Re}\,}
\renewcommand{\Im}{\mathrm{Im}\,}

\newcommand{\Supp}[1]{\supp\left( #1\right) }

\newcommand{\ee}{\mathrm{e}}

\newcommand{\dL}{\mathfrak{D}}

\newcommand{\loc}{\mathrm{loc}}
\newcommand{\vect}[1]{\mathbf{{#1}}}
\newcommand{\dd}{\mathrm{d}}

\DeclarePairedDelimiter{\abs}{\lvert}{\rvert}

\DeclarePairedDelimiter{\norm}{\lVert}{\rVert}

\let\originalleft\left
\let\originalright\right
\renewcommand{\left}{\mathopen{}\mathclose\bgroup\originalleft}
\renewcommand{\right}{\aftergroup\egroup\originalright}

\newcommand{\N}{\mathbb{N}}
\newcommand{\Z}{\mathbb{Z}}
\newcommand{\Q}{\mathbb{Q}}

\newcommand{\C}{\mathbb{C}}

\newcommand{\R}{\mathbb{R}}

\newcommand{\gf}{\mathfrak{g}}

\newcommand{\Mf}{\mathfrak{M}}

\newcommand{\Bs}{\mathscr{B}}

\newcommand{\Fc}{\mathcal{F}}

\newcommand{\Jc}{\mathcal{J}}

\newcommand{\Lc}{\mathcal{L}}
\renewcommand{\Mc}{\mathcal{M}}

\newcommand{\Pc}{\mathcal{P}}

\newcommand{\Rc}{\mathcal{R}}
\newcommand{\Sc}{\mathcal{S}}
\newcommand{\Tc}{\mathcal{T}}

\newcommand{\Fs}{\mathscr{F}}

\newcommand{\Mr}{\mathscr{M}}

\newcommand{\meg}{\leqslant}
\newcommand{\Meg}{\geqslant}
\renewcommand{\geq}{\geqslant}
\renewcommand{\leq}{\leqslant}
\newcommand{\eps}{\varepsilon}
\renewcommand{\phi}{\varphi}
\newcommand{\mi}{\mu}

\DeclareFontFamily{U}{mathx}{\hyphenchar\font45}
\DeclareFontShape{U}{mathx}{m}{n}{
      <5> <6> <7> <8> <9> <10>
      <10.95> <12> <14.4> <17.28> <20.74> <24.88>
      mathx10
      }{}
\DeclareSymbolFont{mathx}{U}{mathx}{m}{n}
\DeclareFontSubstitution{U}{mathx}{m}{n}
\DeclareMathAccent{\widecheck}{0}{mathx}{"71}
\DeclareMathAccent{\wideparen}{0}{mathx}{"75}
\thanks{{\em Math Subject Classification 2020}: 46E36, 22E30.}
\thanks{The authors are members of the 	Gruppo Nazionale per l'Analisi
	Matematica, la Probabilit\`a e le	loro Applicazioni (GNAMPA) of
	the Istituto Nazionale di Alta Matematica (INdAM). The authors were partially funded by the INdAM-GNAMPA Project CUP\_E5324001950001.
}

\begin{document}
	
		\title[Function spaces for weighted subcoercive operators]{Function spaces for weighted subcoercive operators}

	\author[T.\ Bruno]{Tommaso Bruno}
	\address{Dipartimento di Matematica, Universit\`a degli Studi di Genova\\ Via Dodecaneso 35, 16146 Genova, Italy}
	\email{tommaso.bruno@unige.it}
	
	\author[M.\ Calzi]{Mattia Calzi}
	\author[M.\ M.\ Peloso]{Marco M.\ Peloso}
	\address{Dipartimento di Matematica, Universit\`a degli Studi di
	Milano, Via C. Saldini 50, 20133 Milano, Italy}
\email{{\tt mattia.calzi@unimi.it}}
\email{{\tt marco.peloso@unimi.it}}

	\theoremstyle{definition}
	\newtheorem{deff}{Definition}[section]

	\newtheorem{oss}[deff]{Remark}
	
	\newtheorem{ass}[deff]{Assumptions}
	
	\newtheorem{nott}[deff]{Notation}

	\theoremstyle{plain}
	\newtheorem{teo}[deff]{Theorem}
	
	\newtheorem{lem}[deff]{Lemma}
	
	\newtheorem{prop}[deff]{Proposition}
	
	\newtheorem{cor}[deff]{Corollary}

	\begin{abstract}
	We develop a theory of Sobolev, Besov and Triebel--Lizorkin spaces associated with weighted subcoercive operators on any real connected Lie group.
	\end{abstract}
	
	\maketitle
\addtocontents{toc}{\protect\setcounter{tocdepth}{1}}
	\tableofcontents

	\section{Introduction}
        The theory of function spaces is a fundamental, classical
        subject that is intimately connected with  regularity of solutions of
differential equations, theory of singular integrals, geometric and
functional inequalities, just to name a few.  Sobolev spaces, their generalizations of Triebel--Lizorkin spaces,  as well as Besov
spaces, play
a prominent role.  It turns out that many of the questions arising in
this vast area are better explained in terms of a particular space,
i.e.\ a particular topology, adapted to the geometry of the problem
under consideration, just as the classical Laplacian is used to define
Sobolev, Besov and Triebel--Lizorkin spaces in the
Euclidean setting. 

In this paper we introduce and study the basic
properties of a class of function spaces that are modelled on a general class of operators on any connected Lie group.  The operators we
consider are the so-called \emph{weighted sub-coercive operators},
introduced by A.\ ter Elst and D.\ Robinson~\cite{ElstRobinson} and
subsequently studied, among others, by A.\ Martini~\cite{MartiniTesi,Martini}.  These
operators are hypoelliptic and include as particular cases the positive
elliptic operators in Euclidean spaces, the sub-Laplacians on general
Lie groups, and the Rockland operators on homogeneous Lie
groups. Furthermore, the heat kernels associated to a weighted
sub-coercive operator satisfy Gaussian estimates. One of
the main results in \cite{ElstRobinson} says that the converse implication
holds too.  We refer to Section \ref{sec3-rev} for the precise
statement.  Thus, the weighted sub-coercive operators constitute a
very general class of differential operators for which ``good''
properties hold and, starting from them, it is possible to develop a
satisfactory
regularity theory.

Thus, let $G$ denote a connected Lie group, $\mathfrak{g}$ its Lie
algebra, and $U(\mathfrak{g})$  the algebra of left-invariant
differential operators on $G$ with complex coefficients.

Unlike in the case of a homogeneous group, where
 every vector field is assigned its degree of homogeneity, 
 the weighted theory begins by choosing 
an algebraic basis $X_1,\dots,X_d$ of $\mathfrak g$ (that is, a system
of generators of the Lie algebra), together with a
 system of positive weights
$w_1,\dots,w_d\in [1,\infty)$,
attached to each basis element.  
We call $X_1,\dots,X_d$ a
\emph{weighted algebraic basis}.  
Such a basis
allows one to 
determine an increasing right-continuous filtration
$(F_\lambda)_{\lambda>0}$  of 
${\mathfrak g}$.
\begin{comment}
It is then possible to select a
suitable 
subset of the 
basis, that we call a \emph{reduced weighted basis} and still denote
by $X_1,\dots,X_d$, and  defines the
same filtration $(F_\lambda)_{\lambda>0}$.
\end{comment}  
Then, the filtration
$(F_\lambda)_{\lambda>0}$ defines a finite set of weights
$1\meg\lambda_1<\cdots<\lambda_k$, with the following properties.
Setting
\[
  F_\lambda = \bigcap_{\mu > \lambda} F_\mu, \qquad F_\lambda^-
  = \bigcup_{\mu < \lambda} F_\mu, \qquad {\mathfrak g}_{*,\lambda}=  F_\lambda/F_\lambda^-
  \]
  then $\mathfrak{g}_{*,\lambda}\neq\{ 0 \} $  exactly when
  $\lambda\in\{\lambda_1,\dots,\lambda_k\}$ and  
  \[
{\mathfrak g}_* :=  \mathfrak{g}_{*,\lambda_1} \oplus \dots \oplus
\mathfrak{g}_{*,
  \lambda_k}
  \]
 may be endowed with a natural homogeneous Lie algebra structure with degrees  $\lambda_1<\cdots<\lambda_k$.

  We call \emph{weighted Lie algebra} a Lie algebra ${\mathfrak g}$  with a chosen
  reduced weighted basis, and we call ${\mathfrak g}_*$ the \emph{associated
    homogeneous algebra}.   The Lie group $G$ is called
  \emph{weighted} if ${\mathfrak g}$ is weighted and the simply connected Lie group
  $G_*$ whose Lie algebra is ${\mathfrak g}_*$ is called the \emph{contraction}
  of $G$.   We refer the reader to Subsections \ref{weightedLie} and
  \ref{subcoercive-ops} for all the details and relevant references. 

The associated graded enveloping algebra
$U({\mathfrak g}_*)$
plays the role of the algebra of principal symbols in the weighted calculus.
It is important to emphasize that the weights do not arise from a
grading of the Lie algebra in general, as no homogeneous
structure is assumed on $\mathfrak g$. The weighted filtration is defined solely
by assigning positive  weights to an algebraic basis. When the
Lie algebra happens to be homogeneous
\[
\mathfrak g=\bigoplus_{j=1}^rV_{\lambda_j},
\]
one naturally assigns weight $j$ to every basis element of $V_{\lambda_j}$, and
the weighted filtration coincides with the homogeneous filtration
determined by the grading. Thus, the weighted theory strictly
generalizes the homogeneous setting.

Having assigned the weights $w_1,\dots,w_d$, given a multiindex
$\alpha=(\alpha_1,\dots,\alpha_k )$, with $\alpha_j\in \{1,\dots,d\}$
for every $j$, we define its
weighted length $\| \alpha\|=\sum_{j=1}^k w_{\alpha_j}$.

Consider a left-invariant differential operator, for $m\in\N$,
\[
D=\sum_{\|\alpha\|\le m}c_\alpha \mathbf X_\alpha,
\]
where $\mathbf X_\alpha = X_{\alpha_1}\cdots X_{\alpha_k}$.  We define
the degree of $D$ as $\deg(D)=\max \{
\|\alpha\|\colon c_\alpha\neq0\}$; for simplicity, we shall assume that $\deg(D)=m$. 
The component
\[
P=\sum_{\|\alpha\|=m}c_\alpha \bar{\mathbf{X}}_\alpha,
\]
where $\bar X_j$ is a natural homogeneous operator on $G_*$ induced by $X_j$ (see
Subsection~\ref{weightedLie}), plays the r\^ole of the principal part. Then, the operator $D$ will be called a
\emph{weighted subcoercive operator} if $m/w_j$ is even for
$j\in\{1,\dots,d\}$ and
$P+P^*$ is a positive Rockland operator on $G_*$.   We recall that a
left-invariant operator defined on a homogeneous group is a Rockland
operator if it is homogeneous of some positive degree and its
infinitesimal representation is injective on smooth vectors for all
non-trivial unitary irreducible representations.  Again, we refer the
reader to Subsections \ref{weightedLie} and
  \ref{subcoercive-ops} for the precise definitions and proofs. 

  Here are some simple
  examples:
  \begin{itemize}
	\item If $X_1,\dots,X_d$ linearly generate $\mathfrak{g}$ and
          each of them has weight $1$, then the contraction
          $\mathfrak{g}_*$ is Euclidean (abelian and isotropic), and
          it is not difficult to see that the positive left-invariant
          elliptic operators on $G$ are precisely the weighted
          subcoercive operators with respect to this structure.
	\item If $G = \R^d$ is endowed with isotropic dilations, a
          differential operator is Rockland if and only if it is a
          constant-coefficient homogeneous elliptic operator on
          $\R^d$. In the general case a homogeneous element of $
          U(\mathfrak{g}_*)$ is Rockland if and only if it is
          hypoelliptic, see \cite{HelfferNourrigat}.
	\item On $\R^{2}$ endowed with the basis
          $\{X_{1}=\partial_{x_{1}}, X_{2}=\partial_{x_{2}} \}$ and
          weights $w_{j} = j$, the operator $\partial_{x_{1}}^{4} -
          \partial_{x_{2}}^{2} + \partial_{x_{1}} $ is weighted
          subcoercive but not Rockland.
	\end{itemize}

  \medskip

The notion of weighted subcoercivity may be formulated in terms of
a weighted Gårding
inequality.
In fact, in the seminal work  \cite{ElstRobinson} ter Elst and Robinson
show that a left-invariant operator $D$ is weighted subcoercive if and
only if it satisfies the weighted Gårding
inequality described below.
An operator $D$ is a weighted subcoercive
operator if
$m/w_j$ is even for every $j\in \{1,\dots, d\}$ and if there exist constants $C>0$ and
$\omega\in\R$ such that 
\[
\Re \langle D\phi ,\phi\rangle
\Meg
C\sum_{ \|\alpha\|\le m/2} \| \mathbf X_\alpha \varphi \|_{L^2(G)}^2
-\omega\|\varphi\|_{L^2(G)}^2
\]
for every $\varphi \in C_c^\infty(G)$.

Another fundamental characterization of weighted subcoercive
operators is still given in \cite[Theorem 1.1]{ElstRobinson}: $D$ is 
weighted subcoercive  if and only if its closure  generates a
holomorphic semigroup  with kernels satisfying Gaussian
bounds.

For simplicity, in the sequel we will actually start with a filtration
$(\gf_{\lambda})_{\lambda>0}$ to which one can associate several
weighted algebraic bases. In this way we wish to emphasize 
the fact that our construction depends on the filtration and not on
the chosen basis.
\medskip

The above mentioned facts  about weighted subcoercive
operators show that they generalize many of the fundamental properties
of elliptic constant coefficients operators in $\R^d$ and of the
sub-Laplacians on graded Lie groups.  Thus, it is 
natural to study function spaces with regularity expressed in terms of
such operators, namely the Besov and Triebel--Lizorkin spaces. The
latter ones contain the Sobolev spaces as a special case.

This is the goal of the present paper.  Given a connected Lie group
$G$  endowed with a relatively invariant measure $\beta$, 
we fix an admissible filtration $(\gf_{\lambda})$ and a weighted
subcoercive operator $\Lc$ of degree $\dL$.
Then $\Lc$ generates a semigroup of operators $(\ee^{-t\Lc})_{t>0}$,
which satisfies gaussian-like estimates.

  Suppose $\alpha \in \R$, and
  $p,q\in [1,\infty]$.
 We define the Besov space $B^{p,q}_\alpha(\beta)$ and, if $p\in
 (1,\infty)$, the Triebel--Lizorkin space $F^{p,q}_\alpha(\beta)$,
as the space of $f\in \Sc'(G)$ such that, respectively,
\[
\norm{f}_{B^{p,q}_\alpha(\beta)}\coloneqq \norm{\ee^{-\Lc} f}_{L^p(\beta)} + \bigg( \int_{0}^{1} \big(t^{-\alpha/\dL}\| (t\Lc)^{m_{\alpha}} \ee^{-t\Lc} f\|_{L^p(\beta)}\big)^{q} \, \frac{\dd t}{t}\bigg)^{1/q}
\]
and
\[
\norm{f}_{F^{p,q}_\alpha(\beta)}\coloneqq \norm{\ee^{-\Lc} f}_{L^p(\beta)} + \bigg\|\bigg( \int_{0}^{1} t^{-\alpha/\dL} | (t \Lc)^{m_{\alpha}} \ee^{-t\Lc} f|^{q} \, \frac{\dd t}{t}\bigg)^{1/q}\bigg\|_{L^p(\beta)}
\]
are finite, with the obvious modifications when $q=\infty$, where
$m_{\alpha}= ([\alpha/\dL]+1)_{+}$.  Here and in what follows, we denote by $\Sc=\Sc(G)$ the (Schwartz) space of
smooth super-exponentially decaying functions on $G$ and by $\Sc'$ its dual, the space of
tempered distribution; see Definition~\ref{def:1}.

For these spaces we  prove the following basic facts:
\begin{enumerate}[(I)] 
\item they do not depend on the particular weighted subcoercive
  operator $\Lc$, whence they are intrinsically defined by the initial
  choice of the filtration $(\gf_\lambda)_{\lambda>0} $ on ${\mathfrak g}$;
\item they are Banach spaces;			
\item they embed in $ L^p(\beta)$ if $\alpha>0$;
\item the space of test functions $C_c(G)$ is dense in
  $B^{p,q}_\alpha(\beta)$ and $ F^{p,q}_\alpha(\beta)$, whenever $p,q<\infty$.
\end{enumerate}

Moreover, we provide several different characterizations for the norms
in $B^{p,q}_\alpha(\beta)$ and $F^{p,q}_\alpha(\beta)$, respectively.  Among these, a discretized version of the integral over $(0,1)$
with respect to the measure $\dd t/t$, providing a sort of Littlewood--Paley-type
characterization of the Besov and Triebel--Lizorkin spaces.

\medskip

Other properties of the  Besov and Triebel--Lizorkin spaces we
prove are: 
\begin{enumerate}[(I)] \setcounter{enumi}{4}
  \item $B^{p,q}_\alpha(\chi \beta)=\chi^{-1/p} B^{p,q}_\alpha(\beta)$
  and $F^{p,q}_\alpha(\chi \beta)=\chi^{-1/p} F^{p,q}_\alpha(\beta)$
  for every (continuous) positive character $\chi$ of $G$;
	\item there is $\omega_0$ such that $(\omega
          I+\Lc)^{\alpha'/d}$ induces   isomorphisms within the scales
          of Besov and of Triebel--Lizorkin spaces, namely
\[
\begin{aligned}
B^{p,q}_\alpha(\beta) &\to B^{p,q}_{\alpha-\alpha'}(\beta)\\
	F^{p,q}_\alpha(\beta) &\to F^{p,q}_{\alpha-\alpha'}(\beta)
	\end{aligned}
 \]
                      for every $\omega\Meg \omega_0$;
                    \item (recursive property) if $(X_j)$ are   elements of $\mathfrak g$
 which induce a homogeneous basis of $\gf_*/[\gf_*,\gf_*]$, and
 $d_j=\deg(X_j)$,  $\dd= \mathrm{l.c.m.}(d_j)$, then,
	\[
	\begin{aligned}
				B^{p,q}_\alpha(\beta) &=\{f\in \Sc'\colon \ee^{-\Lc}f\in L^p(\beta),\; \forall j\;\; X_j^{\dd/d_j} f\in B^{p,q}_{\alpha-\dd}(\beta)  \},\\
	F^{p,q}_\alpha(\beta) &=\{f\in \Sc'\colon \ee^{-\Lc}f\in L^p(\beta),\; \forall j\;\; X_j^{\dd/d_j} f\in F^{p,q}_{\alpha-\dd}(\beta)  \},
			\end{aligned} 
	\]
\end{enumerate}
in the sense that the spaces on the right hand sides coincide with the
Besov and Triebel--Lizorkin spaces, respectively, with equivalence of
norms.

Other remarkable properties of these space that we prove in this work
are the following. 
\begin{enumerate}[(I)] \setcounter{enumi}{7}
\item
(Duality) Let $p,q\in [1,\infty]$ and $\alpha\in \R$. Then, the map
\[
B\colon (f,g)\mapsto \lim_{t\to 0^+} \langle \ee^{-t \Lc}f \vert \ee^{-t \Lc^*}g\rangle
\]
induces a continuous sesquilinear map on
$B^{p,q}_\alpha(\beta)\times B^{p',q'}_{-\alpha}(\beta)$ and on
$F^{p,q}_\alpha(\beta)\times F^{p',q'}_{-\alpha}(\beta)$ if $p \in
(1,\infty)$, which identifies $ B^{p',q'}_{-\alpha}(\beta)$ and $F^{p',q}_{-\alpha}(\beta)$ with the dual of $B^{p,q}_\alpha(\beta)$ and $F^{p,q}_\alpha(\beta)$, respectively.\footnote{We point out that, throughout the whole paper,
  we denote by $\langle\cdot\,\vert\, \cdot\rangle$ a  sesquilinear
pairing, and by $\langle\cdot\,,\, \cdot\rangle$ a bilinear one.}
\item (Interpolation) For simplicity, we only state the result for the
  complex interpolation.  For the complete statement, see Section~\ref{Interpolation:sec}.

As customary, suppose  $\theta\in (0,1)$, $\alpha,\alpha_0,\alpha_1\in \R$, $p,p_0,p_1,q,q_0,q_1\in [1,\infty]$, and
\[
\alpha_\theta=(1-\theta)\alpha_0+\theta \alpha_1, \qquad
\frac{1}{p_\theta}=\frac{1-\theta}{p_0}+\frac{\theta}{p_1},
\qquad \frac{1}{q_\theta}=\frac{1-\theta}{q_0}+\frac{\theta}{q_1}.
\]
Then,
\begin{itemize}
  \item[{\tiny $\bullet$}]
$(B_{\alpha_0}^{p_0, q_0}(\beta),B_{\alpha_1}^{p_1,
  q_1}(\beta))_{[\theta]}= B_{\alpha_\theta}^{p_\theta,
  q_\theta}(\beta)$ ($q_\theta\neq \infty$ or $\alpha_0=\alpha_1$),\smallskip 
\item[{\tiny $\bullet$}]
  $(F_{\alpha_0}^{p_0, q_0}(\beta),F_{\alpha_1}^{p_1,
    q_1}(\beta))_{[\theta]}= F_{\alpha_\theta}^{p_\theta,
    q_\theta}(\beta)$ ($p_0,p_1\in (1,\infty)$ and either
  $q_\theta\neq \infty$ or $\alpha_0=\alpha_1$). 
  \end{itemize}
\end{enumerate}
\medskip

Our paper lies at the frontier of a longstanding line of research of classical potential theory in the Euclidean setting and on Lie groups. Its
 history is too vast to be recalled here in
detail, but let us mention a few key steps and works.

The first systematic developments were carried out on stratified
(Carnot) Lie groups by Folland and Stein, who introduced Sobolev and
Hardy spaces associated with sub-Laplacians and established the basic
tools of subelliptic harmonic analysis, including singular integral
theory, Littlewood–Paley decompositions, and Sobolev embedding
theorems. In their works, they showed that
many classical results of Euclidean analysis admit natural analogues
once the Euclidean Laplacian is replaced by a sub-Laplacian defined in terms of a basis satisfying
H\"ormander's condition. In later contributions, in the setting
of unimodular Lie groups, Coulhon, Russ, and
Tardivel-Nachef~\cite{CRTN}, and Bruno, Peloso, Tabacco and Vallarino~\cite{BPTV} 
on general Lie groups, studied Sobolev spaces defined by a
sub-Laplacian.

The theory of Sobolev spaces was subsequently extended to general homogeneous (graded)
Lie groups by Fischer and Ruzhansky~\cite{FischerRuzhansky} replacing
a sub-Laplacian on a stratified Lie group by a
positive Rockland operators, relying on a 
fundamental result of Helffer and Nourrigat~\cite{HelfferNourrigat}
that
asserts that positive
Rockland operators are hypoelliptic.
Fischer and Ruzhansky proved that these spaces are independent of the
particular choice of the
Rockland operator and established Sobolev embeddings, interpolation,
and other geometric inequalities, see also~\cite{FR2}.
\medskip

The theory of Besov and Triebel--Lizorkin spaces defined in terms of a
sub-Lalplacian on general Lie groups was developed in a
systematic way by Bruno, Peloso and Vallarino~\cite{BPV} and~\cite{BPV2}.  The authors proved
heat-semigroup characterizations, interpolation theorems, Sobolev
embeddings, and furthermore
algebra properties.
We also
mention the recent work on homogeneous Lie
groups by Hu, Rottensteiner, Ruzhansky, and Van
Velthoven~\cite{HuRottensteinerRuzhanskyVanVelthoven2025} where
the authors prove several further properties of  \emph{homogeneous} Besov and Triebel--Lizorkin spaces  
including molecular decompositions, maximal-function
characterizations, duality, and embedding theorems.

It should be mentioned that Triebel studied  Besov and
Triebel--Lizorkin spaces  on complete Riemannian manifolds with
bounded geometry and also on Lie groups endowed with a left-invariant Riemannian structure, see~\cite[Chap.\ 7]{TriebelFS2}.  In this setting the spaces were defined only in terms of the Laplace--Beltrami
operator.  For analogous spaces on metric measure spaces, we refer the reader, without pretense of exhaustiveness, to~\cite{B, GKKP1, WHHY, BBD, LYY}.

The paper is organized as follows.
Section 2 is devoted to   basic definitions and some preliminary facts.
In Section~\ref{weightedsubcoerciveops:sec} we
recall the theory of weighted Lie algebras and weighted subcoercive
operators as established by ter Elst and Robinson, and
Martini~\cite{ElstRobinson,MartiniTesi,Martini}.  
In Section~\ref{sec3-rev}, starting from a filtration of the Lie algebra rather than from a
grading, we reformulate the notion of weighted subcoercivity in
terms of filtrations and prove that the definition is intrinsic, i.e.,
independent of the choice of weighted basis. We proceed by establishing
several structural results on weighted subcoercive operators,
including their relationship with positive Rockland operators on the
associated graded group.

In Section~\ref{heatker:sec} we provide a detailed study of the heat
semigroup generated by a weighted subcoercive operator. Building on
the Gaussian kernel estimates of ter Elst--Robinson, we  obtain
additional heat-kernel estimates. In Section~\ref{Carmeas:sec} we introduce suitable Carleson
measures
and prove several estimates on integral operators on $L^p$-spaces
with respect to  these measures.
The measures and estimates we allow us to treat in a fairly unified
way a number of arguments on Besov and Triebel--Lizorkin spaces,
respectively,
that otherwise would require two separate sets of proofs.

In Section~\ref{BTL:sec}
we finally 
define Besov and Triebel--Lizorkin spaces using heat-semigroup norms,
generalizing the classical Euclidean, homogeneous-group
constructions.  In particular we prove
\begin{itemize}
\item completeness;
\item independence of the choice of $\Lc$;
equivalent norms;
\item
  discrete Littlewood--Paley-type  characterizations;
\item 
recursive characterizations via invariant derivatives, and 
invariance under multiplication by positive characters.
\end{itemize}
In Section~\ref{emb:sec}
we prove inclusions  between Besov and 
 between Triebel–-Lizorkin spaces, respectively, as well as 
comparisons with $L^p$-spaces. In Section~\ref{dual:sec} we prove duality between Besov and 
between Triebel–-Lizorkin spaces, respectively. Along the way, we also prove the
density of test functions in Besov and 
in Triebel–-Lizorkin spaces.
Section~\ref{Interpolation:sec} is devoted to interpolation theorems
with respect to both real and complex methods.  We conclude with Sections~\ref{A1:sec}
and \ref{A2:sec}, which contain proofs of two results stated in
previous sections.

\medskip

To conclude, we mention that several works are in preparation (cf.~\cite{Calzi2,Calzi3,Calzi4,Calzi5,CalziRizzo}), where further properties of these spaces will be studied, including:
\begin{itemize}
	\item characterizations in terms of differences;
	
	\item algebra properties;
	
	\item localization of norms;
	
	\item pointwise multipliers;
	
	\item the spaces $F^{\infty,q}_\alpha(\beta)$;
	
	\item the full scale of Besov and Triebel--Lizorkin spaces ($p,q\in(0,\infty]$) when $G$ has polynomial growth, including discretization theorems and comparison with local Hardy and $\mathrm{bmo}$ spaces.
\end{itemize}

	\subsection*{Source disclosure}  This paper was written without any assistance of large language models. The content is exclusively the authors' original work.

\section{Preliminaries}

Throughout the paper, we shall denote with $G$ a finite-dimensional,
real, connected Lie group with Lie algebra $\gf$.

	\subsection{Relatively Invariant Measures and Convolution}

	We shall endow $G$ with a non-trivial positive relatively invariant (Radon) measure $\beta$. Thus, there are two positive characters $\Delta_L$ and $\Delta_R$ of $G$ such that 
	\[
	\Delta_L(y)\int_G f(y x)\,\dd \beta(x) =  \int_G f(x)\,\dd \beta(x) = \Delta_R(y) \int_G f(x y )\,\dd \beta(x)
	\]
	for every $f\in L^1(\beta)$ and $y\in G$. In particular,
	\[
	\beta(x A)=\Delta_L(x) \beta(A) \qquad\text{and}\qquad \beta(A x)=\Delta_R(x) \beta(A)
	\]
	for every $\beta$-measurable subset $A$ of $G$ and $x\in G$. Notice that $\Delta_L\equiv 1$ if $\beta$ is a left Haar measure and $\Delta_R\equiv 1$ if $\beta$ is a right Haar measure. In addition,
\begin{equation}\label{reflection}
	\int_G f(x^{-1}) \,\dd \beta(x)= \int_G f(x) \Delta_L(x^{-1})\Delta_R(x^{-1})\,\dd \beta(x)
\end{equation}
	for every positive $\beta$-measurable function $f$ on $G$. We shall define 
	\[
	\beta_L\coloneqq \Delta_L^{-1}\cdot \beta, \qquad \beta_R\coloneqq \Delta_R^{-1}\cdot \beta,
	\]
	so that $\beta_L$ is a left Haar measure and $\beta_R$ is a right Haar measure. 
	
	\begin{oss}\label{comparison0}
Notice that if $\beta$ is as above, then it is absolutely continuous with respect to any right Haar measure and its density is a positive character $\chi$ of $G$. In other words, $\beta$ can be thought of as a measure $\mu_{\chi}$ as those considered in~\cite{BPTV,BPV}, namely as a measure of the form $\dd\beta = \chi\,  \dd \rho$ where $\rho$ is a fixed Haar measure of $G$ and $\chi = \Delta_{R}$. Viceversa, the measures considered in~\cite{BPTV,BPV} are relatively invariant in the sense above.
	\end{oss}
	
Given two convolvable distributions $f,g$ on $G$, their convolution is defined via
	\[
	\langle f*g, \phi\rangle= \langle f\otimes g, (x,y)\mapsto \phi(xy)\rangle
	\]
	for every $\phi\in C^\infty_c(G)$. If  $f\in L^1_\loc(\beta)$,
        we shall identify it with the distribution $f\cdot \beta$,
        that is, with the measure with density $f$ with respect to
        $\beta$. Analogously, if $f,g$ are two convolvable functions
        such that $f*g$ is absolutely continuous with respect to
        $\beta$, we shall identify $f*g$ with its density with respect
        to $\beta$. Thus, under very mild conditions which in this work will be always verified when needed,
	\[
	(f* g)(x)=\int_G f(x y^{-1}) g(y) \Delta_R(y^{-1})\,\dd \beta(y)= \int_G f(y) g(y^{-1}x)\Delta_L(y^{-1})\,\dd \beta(y).
\]
	The above in particular holds if $f$ and $g$ are positive measurable functions.	Similar formulae apply when either $f$ or $g$ is a distribution and the convolution is still a function.
	Observe that if $f,g$ are convolvable distributions, $X$ is a left-invariant differential operator, and $Y$ is a right-invariant differential operator, then $Yf $ and $X g$ are convolvable and
	\[
	YX(f*g)=(Yf)*(Xg).
	\]
	
	If $f,g,h$ are distributions then, under some reasonable conditions that will always be satisfied when needed,
\begin{equation}\label{fgh}
	\langle f * g, h \rangle =\langle f, h* (\Delta_R \Rc g)\rangle =\langle g, (\Delta_L \Rc f)* h\rangle,
\end{equation}
	where $\Rc$ is the reflection operator 
	\[
	\langle \Rc f,\phi\rangle=\langle f,\widecheck \phi \rangle, \qquad \phi\in C^\infty_c(G),
	\] 
	and $\widecheck \phi=\phi(\,\cdot\,^{-1})$. Note that since by~\eqref{reflection}
	\[
	\Rc \beta= \Delta_L^{-1} \Delta_R^{-1} \beta,
	\]
	one has 
	\[
	\Rc(f\cdot \beta) =(\Delta_L^{-1} \Delta_R^{-1}  \widecheck f)\cdot \beta.
	\]
	In other words, one must be careful not to confuse $\Rc f$ with $\widecheck f$ when $f\in L^1_\loc(\beta)$.
		
	We now recall Young's inequality in this context. Suppose $p_1,p_2,p_3\in [1,\infty]$ are such that $\frac{1}{p_1'}+\frac{1}{p_2'}=\frac{1}{p_3'}$. Then
	\[
	\norm{(\Delta_L^{1/p_2'} f)*(\Delta_R^{1/p_1'}g)}_{L^{p_3}(\beta)}\meg \norm{f}_{L^{p_1}(\beta)}\norm{g}_{L^{p_2}(\beta)}
	\]
	or, equivalently,
	\[
	\norm{f*g}_{L^{p_3}(\beta)}\meg \norm{\Delta_L^{-1/p_2'}f}_{L^{p_1}(\beta)}\norm{\Delta_R^{-1/p_1'}g}_{L^{p_2}(\beta)}
	\]
	for every two positive $\beta$-measurable functions $f,g$. We refer the reader to~\cite[Lemma 2.1]{KR78} in the case when $\Delta_L=1$, and the the general case follows easily.
	
	\subsection{Differential Operators} We identify $\gf$ with the algebra of left-invariant vector fields on $G$. Correspondingly,  the (complexification of the) enveloping algebra $U(\mathfrak{g})$ of $\gf$ will be identified with the algebra of left-invariant differential operators. We fix a scalar product on $\gf$, endow $U(\mathfrak{g})$ with the corresponding scalar product, namely the quotient of the natural scalar product on the (complexification of the) tensor algebra over $\gf$, and denote by $|X|$ the norm of an element $X\in
U(\mathfrak{g})$. 
	 
	Given a left-invariant differential operator $X$, we denote with $X^R$ the right-invariant differential operator which induces the same point distribution as $X$ at $e$, the identity element of $G$.  In other words, 
	\[
	(X f)(e)=(X^R f)(e), \qquad f\in C^\infty (G).
	\]
Recall that left-invariant vector fields commute with right-invariant ones.

	We denote with $X^+$ the transpose of $X$ in the enveloping algebra $U(\mathfrak{g})$. In other words, the mapping $X\mapsto X^+$ is the unique anti-automorphism of $U(\mathfrak{g})$ (that is, such that $(XY)^+=Y^+ X^+$) which extends the automorphism $X\mapsto -X$ of $\gf$. 
		
		We denote with $X^\dag$ the formal transpose of $X$ with respect to $\beta$, that is, the unique left-invariant differential operator such that
		\[
		\int_G (X f) g\,\dd \beta=\int_G f X^\dag g\,\dd \beta
		\]
		for every $f,g\in C^\infty_c(G)$. We denote with $X^*$ the formal adjoint of $X$, that is, $\overline X^\dag$. Notice that if $X\in \mathfrak{g}$, then $X^{*} = X^{\dagger}$. We define the formal transpose and the formal adjoint of right-invariant differential operators in a similar way. If $u$ is a distribution, we then define $X u$ so that
		\[
		\langle X u,\phi\rangle =\langle u, X^\dag \phi\rangle
		\]
		for every $\phi\in C^\infty_c(G)$. In this way, if $f\in C^\infty(G)$, then $(X f)\cdot \beta=X(f\cdot \beta)$. Analogously, we define
		\[
		\langle X^{\dag} u,\phi\rangle =\langle u, X \phi\rangle.
		\]
		To avoid cumbersome notation, we write $X^{R\dag}$ instead of $(X^R)^\dag$, etc.
	
	Notice that $X f$ \emph{depends} on $\beta$ if $f$ is a distribution, whereas $X^\dag f$ does not. On the contrary, $X^\dag f$ \emph{depends} on $\beta$ if $f\in C^\infty(G)$, whereas $X f$ does not.  This unfavorable dichotomy disappears if $\beta$ is right-invariant, as the next proposition shows.  We denote by $\delta_{e}$ the Dirac delta at $e\in G$.

	\begin{prop}\label{prop:9}
Let $X \in U(\mathfrak{g})$ be a left-invariant differential operator. Then the following hold.
		\begin{enumerate}
			\item[\textnormal{(1)}] $X^\dag f= \Delta_R^{-1} X^+(\Delta_R f)$ and $X^{R\dag} f=\Delta_L^{-1} X^{+R}(\Delta_L f)$ for every $f\in C^\infty(G)$.
			
			\item[\textnormal{(2)}] $X\delta_e= \Delta_R \Rc(X^\dag \delta_e)$  and $X^R\delta_e=  \Delta_L \Rc (X^\dag \delta_e)$.
			
			\item[\textnormal{(3)}] $X^+ f=(X^R \widecheck f) \widecheck{\,}$ for every $f\in C^\infty(G)$.

			\item[\textnormal{(4)}] $X\delta_e=X^{\dag R\dag}\delta_e$.
			
			\item[\textnormal{(5)}] $X\chi= (X\chi)(e)\chi$ for every character $\chi$  on $G$.
		\end{enumerate}
	\end{prop}

	\begin{proof}
		We start by the first equality in (1), and assume first that $\beta$ is right-invariant, that is, $\Delta_R=1$, and prove that $X^\dag=X^+$. Since  $(YZ)^\dag=Z^\dag Y^\dag$ for every $Y,Z\in U(\mathfrak{g})$, it suffices to observe that $Y^\dag=-Y$ for every $Y\in \gf$. Then, consider the general case, and observe that, for every $f \in C^\infty(G)$ and $g\in C^\infty_c(G)$,
		\[
		\int_G (Xg) f\,\dd \beta=\int_G (X g) (\Delta_R f) \,\dd \beta_R= \int_G g X^+(\Delta_R f)\,\dd \beta_R= \int_G g \Delta_R^{-1} X^+(\Delta_R f)\,\dd \beta,
		\]
		so that the first assertion follows by the arbitrariness of $g$.  The proof of the second assertion in (1) is slightly postponed.
		
	As for (2), we shall use repeatedly~\eqref{fgh}. For all $f\in C^\infty(G)$ and $g\in C^\infty_c(G)$,
		\[
		\langle f*(X\delta_e),g \rangle= \langle X f, g \rangle= \langle f, X^\dag g \rangle=\langle f, g*(X^\dag \delta_e)\rangle= \langle f*(\Delta_R \Rc(X^\dag \delta_e)),g \rangle,
		\]
whence the first assertion. Then, observe that for $f\in C^\infty(G)$
		\[
		\begin{split}
		(Xf)(e)&=(X^R f)(e)\\
			&=\langle(X^R \delta_e)*f,\delta_e\rangle\\
			&=\langle f, (\Delta_L \Rc(X^R \delta_e))*\delta_e\rangle\\
			&=\langle \Rc(X^R \delta_e), \Delta_L f\rangle= \langle X^R \delta_e, \Delta_L^{-1}\widecheck f \rangle= (X^{R\dag}(\Delta_L^{-1}\widecheck f))(e).
		\end{split}
		\]
		Applying the previous identity to $\Delta_L^{-1}\widecheck f$ instead of $f$, we then find
		\[
		X(\Delta_L^{-1}\widecheck f)(e)=(X^{R\dag}f)(e).
		\]
		Then,
		\[
		\begin{split}
			\langle X^R \delta_e, f \rangle&=\langle \delta_e, X^{R\dag} f\rangle\\
			& =\langle \delta_e, X(\Delta_L^{-1}\widecheck f)\rangle =\langle X^\dag \delta_e, \Delta_L^{-1} \widecheck f \rangle= \langle  \Rc (X^\dag \delta_e), \Delta_L f\rangle=\langle \Delta_L\Rc (X^\dag \delta_e), f\rangle,
		\end{split}
		\]
		whence the second  assertion.

		We now prove (3). Observe that if $X\in \gf$ and $f\in C^\infty(G)$, then
		\[
		\begin{split}
			(X^R \widecheck f)(x^{-1})&= \frac{\dd}{\dd t}\Big \vert_{t=0} \widecheck f(\exp_G(t X) x^{-1})\\
				&=\frac{\dd}{\dd t}\Big \vert_{t=0} f(x\exp_G(-t X))=-(Xf)(x)=(X^+ f)(x)
		\end{split}
		\]
		for every $x\in G$, so that $X^+ f=(X^R \widecheck f) \widecheck{\,}$. In addition, for every $X,Y\in U(\mathfrak{g})$,
		\[
		\begin{split}
			(XY)^R \delta_e&=\Delta_L \Rc((XY)^\dag \delta_e)\\
				&=\Delta_L \Rc(Y^\dag X^\dag \delta_e)\\
				&=\Delta_L \Rc( X^\dag \delta_e * Y^\dag \delta_e )\\
				&=[\Delta_L\Rc (Y^\dag \delta_e)]*[\Delta_L\Rc (X^\dag \delta_e)]=(Y^R\delta_e)*(X^R\delta_e)=(Y^R X^R)\delta_e
		\end{split}
		\]
		by (2), so that $(XY)^R=Y^R X^R$. The assertion follows.
		
We are now in a position to prove the second assertion in (1). The equality 
\[
X^{R\dag} f=\Delta_L^{-1}X^{+R}(\Delta_L f)
\]
may be proved as above when $X\in \gf$ (in which case it reduces to $X^{R\dag} f=-\Delta_L^{-1}X^{R}(\Delta_L f)$). Then, observe that by  (3),  
\[
(YZ)^{R\dag}=(Z^R Y^R)^\dag=Y^{R\dag}Z^{R\dag}, \qquad (YZ)^{+R}=(Z^+ Y^+)^R=Y^{+R}Z^{+R}
\]
for every $Y,Z\in U(\mathfrak{g})$, so that the assertion follows by the universal property of the enveloping algebra.
	
	To show (4), observe that for $f\in C^\infty(G)$
		\[
		\begin{split}
			\langle X \delta_e, f\rangle &=\langle \delta_e, X^\dag f\rangle= \langle \delta_e, X^{\dag R} f\rangle= \langle X^{\dag R \dag}\delta_e, f\rangle,
		\end{split}
		\]
		whence the assertion.
		
		We are left with proving (5). Observe first that 
		\[
		X\chi=\chi*X\delta_e=\chi(1*\chi^{-1}X\delta_e)
		\]
		and that moreover, for $x\in G$,
		\[
		\begin{split}
			(1*\chi^{-1}X\delta_e)(x)&=\langle \chi^{-1}X\delta_e, \Delta_R^{-1}\rangle\\
				&= \langle X\delta_e, (\chi \Delta_R)^{-1}\rangle =X^\dag((\chi \Delta_R)^{-1})(e)=(X^+ \chi^{-1})(e)=(X\chi)(e)
		\end{split}
		\]
		thanks to (1) and (3). The assertion follows.
	\end{proof}
	Let us notice that, by Proposition~\ref{prop:9}~(4), 
	\[
	\begin{split}
		(Xf)*g=(f*X\delta_e)*g=f*(X\delta_e*g)=f*(X^{\dag R \dag} \delta_e*g)= f*(X^{\dag R \dag} g)
	\end{split}
	\]
	whenever $f$ and $g$ are sufficiently regular to grant associativity of the convolution. Notice that $X^{\dag R \dag}-X^R$ is a differential operator whose order is strictly smaller than that of $X$, and whose degree (we shall see below its definition) is strictly less than that of $X$.
	
For future need, let us also state the following remark.
\begin{oss}\label{comparison1}
If we write $\beta$ in the form $\mu_{\chi}$ where $\chi$ is a positive character of $G$, recall Remark~\ref{comparison0}, then if $X\in \mathfrak{g}$
\[
X^{*} f = X^{\dag} f = -Xf  - (X\chi)(e)f
\]
 for all $f\in C^{\infty}(G)$.
\end{oss}

\subsection{Homogeneous groups}
A Lie group $G$ is called \emph{homogeneous} if so is $\mathfrak{g}$. Recall that $\mathfrak{g}$ is said to be homogeneous if it is endowed with a family of automorphic dilations $(\delta_t = e^{A \log t})_{t>0}$, where $A\in \mathrm{GL}(\mathfrak{g})$ is diagonalizable with strictly positive eigenvalues. In this case, the eigenspaces $\mathfrak{g}_\lambda$ of $A$ determine a direct-sum decomposition
\[
\mathfrak{g} = \bigoplus_{\lambda>0} \mathfrak{g}_\lambda = \mathfrak{g}_{\lambda_1} \oplus \dots \oplus \mathfrak{g}_{\lambda_k},
\]
where $0<\lambda_{1}< \dots <\lambda_k$ are the eigenvalues of $A$,
and $[\mathfrak{g}_\lambda,\mathfrak{g}_{\lambda'}] \subseteq
\mathfrak{g}_{\lambda+\lambda'}$  for  $\lambda,\lambda' >0$. Up to rescaling, one can always suppose that $\lambda_1 \geq 1$, and we
shall do so in what follows whenever a homogeneous group is involved.

A homogeneous Lie group $G$ is nilpotent, and it can be
identified with its Lie algebra $\mathfrak{g}$, and the automorphisms
$\delta_{t}$ on $\mathfrak{g}$ induce automorphisms of $G$ which we
still denote by $\delta_{t}$. Since a homogeneous group $G$ is
nilpotent, it is unimodular, i.e.\ $\Delta_L=\Delta_R = 1$ which in turn implies $\beta_{L}=\beta_{R}$. If $G$ is homogeneous, the quantity $Q:=\tr(A)$ is called the homogeneous dimension of $G$. Recall that $ \beta_{L}(\delta_t(U)) = t^{\tr(A)} \beta_{L}(U)$ for any measurable subset $U$ of $G$. 

\subsection{Notation}

Given two measure spaces $(X,\nu)$ and $(Y,\mu)$ and $q,p\in
[1,\infty]$, we denote by $L^{q,p}(\mu,\beta)$ the space of
$(\nu\otimes \mu)$-measurable functions $f\colon X\times Y \to \C$
such that 
	\[
	\| f\|_{L^{q,p}(\mu, \nu)}:= \bigg(\int_{X} \bigg( \int_{Y}|f(x,y)|^{q}\, d\mu(y)\bigg)^{p/q}\, d\nu(x)\bigg)^{1/p}<+\infty,
	\]
with the usual modifications when $\max(p,q)=\infty$.

		In order to simplify the notation, given a measure space $(X,\mi)$, $p\in (0,\infty]$, and a $\mi$-measurable function $f$ on $X$, we shall sometimes write $\norm{f(x)}_{L^p_x(\mi)}$ instead of $\norm{f}_{L^p(\mi)}$ (we allow the possibility $\norm{f}_{L^p(\mi)}=+\infty$). This will be particularly useful in the presence of nested $L^p$-norms.
For instance, we will also write
\[
	\| f\|_{L^{q,p}(\mu, \beta)}=  \big\| \| f(x,y) \|_{L^{q}_{y}(\mi)} \big\|_{L^{p}_{x}(\beta)}.
  \]

All throughout, given two quantities  $A$ and $B$ and a constant $C$ independent of them, we write $A\asymp_{C} B$ if $C^{-1} B \leq A\leq C B$. If the role of $C$ is not relevant, we just write  $A\asymp B$.

	\section{Weighted subcoercive operators}\label{weightedsubcoerciveops:sec}
In this section we summarize some results from~\cite{ElstRobinson, MartiniTesi, Martini} to which we refer for all the details. Note that in~\cite{ElstRobinson, MartiniTesi, Martini} the group is endowed with a right Haar measure, but the situation remains substantially the same in our setting.
\subsection{Rockland operators}
Let $G$ be a homogeneous group. The automorphic dilations $\delta_t$ of $\mathfrak{g}$ extend to automorphisms $\delta_t$ of $U(\mathfrak{g})$. An element $D \in U(\mathfrak{g})$ is said to be homogeneous if there exists $\lambda\in\R$ such that
\[\delta_t(D) = t^\lambda D \qquad \forall \, t > 0,\]
and such a $\lambda$ is called its degree of homogeneity.

\begin{deff}
A left-invariant differential operator $D \in U(\mathfrak{g})$ is said to be \emph{Rockland} if it is homogeneous and, for every non-trivial irreducible unitary representation $\pi$ of $G$ on a Hilbert space $\mathcal{H}$, $d\pi(D)$ is injective on the space $\mathcal{H}^\infty$ of the smooth vectors of the representation. 
\end{deff}

We recall that a homogeneous left-invariant differential operator
is Rockland if and only if it is hypoelliptic, see~\cite{HelfferNourrigat}.

\subsection{Weighted Lie groups and contractions} \label{weightedLie}
Let us start by introducing the notion of weighted algebraic basis of the Lie algebra $\mathfrak{g}$.
\begin{deff}
A weighted algebraic basis of $\mathfrak{g}$ consists of a collection of linearly independent left-invariant vector fields $X_1,\dots,X_d$ which generate $\mathfrak{g}$ as a Lie algebra, together with weights $w_1,\dots, w_{d} \in [1,\infty)$.
\end{deff}

Suppose $X_1,\dots,X_d$ is a weighted algebraic basis of $\mathfrak{g}$. We denote by $J(d)$ the set of finite sequences of elements of $\{1,\dots,d\}$, and $J_+(d)$ the subset of non-empty sequences. For every $\alpha = (\alpha_1,\dots,\alpha_k) \in J(d)$, we write $|\alpha|$ for its length $k$, and set $\|\alpha\| = \sum_{j=1}^k w_{\alpha_j}$. The element $ X_{\alpha_1} X_{\alpha_2} \cdots X_{\alpha_k}$ of $U(\mathfrak{g})$ is denoted by $\mathbf{X}_\alpha$. The fixed weighted basis defines an increasing right-continuous filtration on $\mathfrak{g}$, since if we set, for $\lambda \geq 0$,
\[F_\lambda = \mathrm{span} \{[[\dots[X_{\alpha_1},X_{\alpha_2}],\dots],X_{\alpha_k}] \colon \alpha \in J_+(d), \, \|\alpha\| \leq \lambda\}
\]
then
\[[F_{\lambda},F_{\mu}] \subseteq F_{\lambda+\mu}, \qquad F_\lambda = \bigcap_{\mu > \lambda\geq 0} F_\mu, \qquad \bigcup_{\lambda \geq 0 } F_\lambda = \mathfrak{g}.\]
Set $F_\lambda^- = \bigcup_{\mu < {\lambda}} F_\mu$; the weighted basis is said to be \emph{reduced} if for all $\lambda\geq 0$
\[
\mathrm{span} \{ X_j \colon w_j = \lambda\} \cap F_\lambda^- = \{0\}.
\]
For any weighted algebraic basis, it is always possible to remove some elements in order to obtain a reduced basis of $\mathfrak{g}$ which defines the same filtration.

A  weighted Lie algebra is a Lie algebra with a choice of reduced weighted basis. Notice that if $\mathfrak{g}$ is homogeneous, then every system of linearly independent generators $X_1,\dots,X_d$ of $\mathfrak{g}$ (as a Lie algebra) consisting of homogeneous elements, with the weights equal to the respective degrees of homogeneity, is a reduced basis of $\mathfrak{g}$; such a basis is said to be adapted to the homogeneous structure of $\mathfrak{g}$.

Let $\mathfrak{g}$ be a weighted Lie algebra, and let the filtration $(F_\lambda)_{\lambda>0}$ be defined as above. The filtration determines a finite set of weights $\lambda_1,\dots,\lambda_k$, with $1 \leq \lambda_1 < \dots < \lambda_k,$ defined by the condition $F_{\lambda_j} \neq F_{\lambda_j}^{-}$ for $j=1,\dots,k$. Then
\[
\mathfrak{g}_* := \bigoplus_{\lambda >0} \mathfrak{g}_{*\lambda} = \mathfrak{g}_{*\lambda_1} \oplus \dots \oplus \mathfrak{g}_{*\lambda_k}, \qquad \mathfrak{g}_{*\lambda} := F_\lambda / F_\lambda^-, 
\]
is a homogeneous Lie algebra, with weights $\lambda_1, \dots, \lambda_k$. We call this the homogeneous algebra associated to  the weighted algebra $\mathfrak{g}$.

Since the weighted basis $X_1,\dots,X_d$ is reduced, the $\bar X_j = X_j + F_{w_j}^{-} \in \mathfrak{g}_{*w_j}$ are linearly independent homogeneous elements of $\mathfrak{g}_*$.

\begin{deff}
The group $G$ is said to be \emph{weighted} if $\mathfrak{g}$ is weighted. If $G$ is weighted, the \emph{contraction} $G_*$ of $G$ is the homogeneous Lie group whose Lie algebra is $\mathfrak{g}_*$.
\end{deff}

\subsection{Weighted subcoercive forms and operators}\label{subcoercive-ops}

Let $G$ be a weighted Lie group, and let $X_1,\dots,X_d$ be the corresponding  reduced basis with weights $w_1,\dots,w_d$. We call a \emph{form}  an element of the free (non-commutative associative unital) algebra over $\C$ on $d$ indeterminates $a_1,\dots,a_d$; equivalently, a form is a function $C \colon J(d) \to \C$ null off a finite subset of $J(d)$, given by the (non-commutative) polynomial
\begin{equation}\label{genericform}
\sum_{\alpha \in J(d)} C(\alpha) a_{\alpha_1} a_{\alpha_2} \cdots a_{\alpha_k}.
\end{equation}
The \emph{degree} of the form $C$ is 
\[
\max \{\|\alpha\| \colon \alpha \in J(d),\,C(\alpha) \neq 0\}.
\]
If $C$ is a form of degree $m$, then its \emph{principal part} is the form $P \colon J(d) \to \C$ which is given by the sum of the terms of $C$ of degree $m$:
\[P(\alpha) = \begin{cases}
C(\alpha) &\text{if $\|\alpha\| = m$,}\\
0 &\text{otherwise.}
\end{cases}\]
A form is said to be \emph{homogeneous} if it equals its principal part. The \emph{adjoint} of a form $C$ is the form $C^*$ defined by
\[C^*(\alpha) = (-1)^{|\alpha|} \overline{C(\alpha_*)} ,\]
where $\alpha_* = (\alpha_k,\dots,\alpha_1)$ if $\alpha = (\alpha_1,\dots,\alpha_k)$.

If $C$ is as in~\eqref{genericform}, we may associate a left-invariant differential operator $d\mathrm{R}_G(C) \in U(\mathfrak{g})$ by setting
\[d\mathrm{R}_G(C) = \sum_{\alpha \in J(d)} C(\alpha) \vect{X}_{\alpha}.
\]
Let $G_*$ be the contraction of $G$, with Lie algebra $\mathfrak{g}_*$. Since the elements $\bar X_1,\dots,\bar X_d$ induced on $\mathfrak{g}_*$ by $X_1,\dots,X_d$ are a reduced basis of $\mathfrak{g}_*$, we can associate to a form $C$ both a differential operator $d\mathrm{R}_G(C)$ on $G$ and a differential operator $d\mathrm{R}_{G_*}(C)$ on $G_*$. Following~\cite[Theorem 2.3]{Martini}, we may give the following definition (which avoids introducing local G\aa rding inequalities; see~\cite[Theorem 10.1]{ElstRobinson} or~\cite[Theorem 2.3]{Martini}).

\begin{deff}
Let  $X_1,\dots,X_d$ be a reduced weighted basis with weights $w_1,\dots,w_d$. Let $C$ be a form of degree $m$ such that $m/w_j \in 2\N$ for $j=1,\dots,d$ and whose principal part is $P$. We say that $C$ is weighted subcoercive if $d\mathrm{R}_{G_*}(P+P^{*})$ is a positive Rockland operator on $G_*$. 

A left-invariant differential operator $\mathcal{L}$ on $G$ is said to be \emph{weighted subcoercive with respect to the reduced weighted basis $X_1,\dots,X_d$} if there exists a weighted subcoercive form $C$ such that  $\mathcal{L} =d\mathrm{R}_G(C)$. 
\end{deff}

Recall that a weighted subcoercive operator (with respect to any reduced weighted basis) is hypoelliptic, and that every positive Rockland operator on a homogeneous group is weighted subcoercive with respect to any basis adapted to the homogeneous structure of $\mathfrak{g}$. Moreover, it is easy to check that, for every choice of a system of linearly independent generators $X_1,\dots,X_d$ of a Lie algebra $\mathfrak{g}$, the assignment of weights all equal to $1$ always gives a reduced basis, and that the corresponding contraction $\mathfrak{g}_*$ is stratified; in particular, the sum of squares sub-Laplacian $\Delta= -\sum_{j=1}^{d}X_j^2$ is weighted subcoercive with respect  to $X_1,\dots,X_d$  (see also Proposition~\ref{prop:4} and Remark~\ref{comparison} below).

	\section{Weighted subcoercive operators revisited}\label{sec3-rev}
We provide here an alternative route to defining weighted subcoercive operators, whose focus is the (increasing) filtration rather than the weighted algebraic basis, and the operator rather than the form.

	\subsection{Filtrations}
	
	Suppose $(\gf_\lambda)_{\lambda\Meg 0}$ is an increasing filtration of $\gf$ (that is, $[\gf_\lambda, \gf_\mi]\subseteq \gf_{\lambda+\mi}$ for $\lambda, \mi\Meg 0$) such that 
\begin{equation}\label{eq:filtr}
	\gf_\lambda=\{ 0\}, \quad \mbox{for all }\lambda<1, \qquad \bigcup_{\lambda\Meg 0} \gf_\lambda=\gf, \qquad \bigcap_{\mi>\lambda} \gf_\mi=\gf_\lambda \quad \mbox{for all } \lambda\Meg 0.
\end{equation}
For $X\in \gf$, we define  
\begin{equation}\label{degX}
\deg X\coloneqq \min\{\lambda\Meg0\colon X\in \gf_\lambda\}
\end{equation}
and we call $\deg X$ the degree of $X$.
	Define, for $\lambda>0$, 
	\[
	\gf_{\lambda^-}\coloneqq \bigcup_{\mi<\lambda} \gf_\mi, \qquad \gf_{*,\lambda}\coloneqq\gf_\lambda/\gf_{\lambda^-}, \qquad 	\gf_*\coloneqq \bigoplus_{\lambda>0} \gf_{*,\lambda}.
	\]
	Then $\gf_*$ can be given a Lie algebra structure as follows: if $X=\sum_{\lambda>0} (X_\lambda+ \gf_{\lambda^-})$ and $Y= \sum_{\lambda>0} (Y_\lambda+ \gf_{\lambda^-})$ for some $(X_\lambda),(Y_\lambda)\in \prod_{\lambda>0} \gf_\lambda$, then
	\[
	[X,Y]\coloneqq \sum_{\lambda,\mi>0}\left(  [X_{\lambda},Y_{\mi}]+\gf_{(\lambda+\mi)^-}\right) .
	\]
	It is easily seen that $(\gf_{*,\lambda})_{\lambda>0}$ is a graduation of type $((0,+\infty),+)$ of $\gf_*$, that is, 
	\[
	[\gf_{*,\lambda}, \gf_{*,\mi}]\subseteq \gf_{*,\lambda+\mi}, \qquad \lambda,\mi>0.
	\]
	One may then endow $\gf_*$ with the automorphic dilations $(\delta_r)_{r>0}$ defined so that $\delta_r(X)=r^\lambda X$ for every $X\in \gf_{*,\lambda}$ and $\lambda>0$. In this way, $\gf_*$ is the Lie algebra of some homogeneous group $G_*$, with homogeneous dimension $Q_*\coloneqq \sum_{\lambda>0} \dim \gf_{*,\lambda}$.

	We shall say that a basis $(X_j)_{j\in J}$ of $\gf$ is \emph{compatible with the filtration} if $(X_j)_{\deg X_j\meg \lambda}$ is a basis of $\gf_\lambda$ for every $\lambda>0$. Note that a basis of $\mathfrak{g}$ which is compatible with the filtration is automatically a weighted reduced basis in the sense of Subsection~\ref{weightedLie}, with their degrees as weights.

	\begin{oss}
The reason for the requirement $\gf_\lambda=\{0\}$ for
$\lambda<1$ will become apparent when we introduce the left-invariant distance, see
Remark~\ref{distance:rem}.
 There may be also good reasons to require $\gf_1\neq \{0\}$, but we prefer not to do so in order to keep a natural comparison with graded groups: if $\gf$ has a graduation $(\tilde \gf_j)_{j\in \Z_+^*}$ with integer degrees, then it is natural to set 
	\[
	\gf_\lambda=\bigoplus_{j=1}^{[\lambda]} \tilde \gf_j, \qquad \lambda \Meg 0,
	\]
	but in general $\tilde \gf_1$ needs not be non-trivial.
\end{oss}

	We now extend this filtration to the  enveloping algebra $U(\mathfrak{g})$.	
	For every $\lambda\Meg 0$, define 
	\[
	U_\lambda = \mathrm{span}\,  \{X_1\cdots X_k \colon k\in\N, \; X_1,\dots, X_k\in \gf, \; \deg X_1+\cdots+ \deg X_k\meg \lambda\}.
	\]
	Notice that the identity operator, corresponding to the case $k=0$, belongs to all $U_\lambda$'s. Then $(U_\lambda)$ is an increasing filtration of $U(\mathfrak{g})$, that is, $U_\lambda, U_\mi\subseteq U_\lambda U_\mi\subseteq U_{\lambda+\mi}$ for every $\lambda, \mi\Meg 0$.
	For every $X\in U(\mathfrak{g})$, we define 
	\[
	\deg X\coloneqq \min \{\lambda\Meg 0\colon X\in U_\lambda\}.
	\]
We will see soon in Proposition~\ref{prop:8} below that $\gf\cap U_\lambda=\gf_\lambda$ for every $\lambda\Meg 0$, so this definition is consistent with~\eqref{degX}.
	
	Define now $U_{0^-}\coloneqq\{0\}$,
	\[
	U_{\lambda^-}\coloneqq \bigcup_{\mi<\lambda} U_\mi, \quad \lambda>0, \qquad U_{*,\lambda}\coloneqq U_\lambda/U_{\lambda^-}, \quad \lambda\Meg 0, \qquad	U_*\coloneqq \bigoplus_{\lambda\Meg 0} U_{*,\lambda} .
	\] 
Then $U_*$ can be endowed with an algebra structure as follows:  if $X=\sum_{\lambda\Meg 0} (X_\lambda+ U_{\lambda^-})$ and $Y= \sum_{\lambda\Meg 0} (Y_\lambda+ U_{\lambda^-})$ for some $(X_\lambda),(Y_\lambda)\in \prod_{\lambda\Meg 0} U_\lambda$, then
	\[
	XY\coloneqq \sum_{\lambda,\mi\Meg 0}\left(  X_{\lambda}Y_{\mi}+U_{(\lambda+\mi)^-}\right) .
	\]
	It is easily seen that $U_*$ becomes a graded algebra of type $([0,+\infty),+)$ with this structure.
	
	\begin{prop}\label{prop:8}
		The canonical inclusions $\gf_\lambda \subseteq U_\lambda$, $\lambda> 0$, induce a linear mapping $\pi \colon \gf_*\to U_*$. The canonical extension $U(\pi)\colon U(\mathfrak{g}_*)\to U_*$ of $\pi$ is an  isomorphism of graded algebras.
	\end{prop}

	\begin{proof}
		Let $(X_j)_{j\in J}$ be a basis of $\gf$ compatible with the filtration, so that $(X_j)_{\deg X_j\meg \lambda}$ is a basis of $\gf_\lambda$ for every $\lambda>0$. Endow the finite set $J$ with a total ordering. Define $\dd_j\coloneqq \deg X_j$ for every $j\in J$, and $\dd_\alpha\coloneqq \sum_{j\in J} \alpha_j\dd_j$ for every $\alpha\in \N^J$.
		Then, the Poincaré--Birkhoff--Witt theorem shows that
		\[
		\vect{X}^\alpha =\prod_{j\in J} X_j^{\alpha_j}, \quad \alpha\in \N^J,
		\]
		form a basis of $U(\mathfrak{g})$; analogously, 
		\[
		 \prod_{j\in J} (X_j+\gf_{\dd_j^-})^{\alpha_j}, \quad \alpha \in \N^J,
		 \]
		 form a basis of $U(\mathfrak{g}_*)$.
		
		Now, fix $\lambda\Meg 0$ and let us prove that $(\vect{X}^\alpha)_{\dd_\alpha\meg \lambda}$ is a basis of $U_\lambda$. Observe first that $\deg \vect{X}^\alpha\meg \dd_\alpha \meg\lambda$ by the definition of $U_\mi$, $\mi\Meg 0$. Conversely, every element of $U_\lambda$ can be written as a finite linear combination of elements of the form $X_{\delta_1}\cdots X_{\delta_n}$ for some $n\in\N$ and for some $\delta\in J^n$ such that $\dd_{\delta_1}+\cdots+\dd_{\delta_n}\meg \lambda$, so that it will suffice to show that each such `monomial' belongs to the vector space generated by  $(\vect{X}^\alpha)_{\dd_\alpha\meg \lambda}$. However, observe that for every $j_1,j_2\in J$ there is a (unique) finite family $(a_{j_1,j_2,j_3})_{j_3\in J}$ such that 
		\[
		[X_{j_1},X_{j_2}]= \sum_{j_3\in J} a_{j_1,j_2,j_3} X_{j_3}
		\]
		and such that $a_{j_1,j_2,j_3}=0$ if $\dd_{j_3}>\dd_{j_1}+\dd_{j_2}$, since the basis $(X_j)$ is compatible with the filtration $(\gf_\mi)$ of $\gf$. Therefore, it is possible to reorder the factors in the product $X_{\delta_1}\cdots X_{\delta_n}$, getting
		\[
		X_{\delta_1}\cdots X_{\delta_n}= \sum_{\dd_\gamma \meg \dd_{\delta_1}+\cdots+\dd_{\delta_n}} b_{\delta,\gamma} \vect{X}^\gamma
		\]
		for some finite family $(b_{\delta,\gamma})$. This may be proved by induction on $n$, as replacing $X_{\delta_j} X_{\delta_{j+1}} $ with $X_{\delta_{j+1}} X_{\delta_j}$  adds a `monomial' in $U_{\dd_{\delta_1}+\cdots+\dd_{\delta_n}}$ of length $n-1$. Thus, $X_{\delta_1}\cdots X_{\delta_n}=\vect X^\gamma + P$ for some $\gamma$ with $\dd_\gamma=\dd_{\delta_1}+\cdots+\dd_{\delta_n}$, and $P$ is a sum of monomials in $U_{\dd_{\delta_1}+\cdots+\dd_{\delta_n}}$ of length $n-1$, and the assertion follows. Now our claim follows too.
		
		Consequently, $(\vect{X}^\alpha+U_{\lambda^-})_{\dd_\alpha=\lambda}$ is a basis of $U_{*,\lambda}$ for every $\lambda\Meg0$. Since clearly
		\[
		U(\pi)\left( \prod_{j\in J} (X_j+\gf_{\dd_j^-})^{\alpha_j} \right)=\vect{X}^\alpha+U_{\dd_\alpha^-}
		\] 
		for every $\alpha\in \N^J$, it follows that $U(\pi)$ is an isomorphism compatible with the products (by the definition of $U(\pi)$) and the graduations of $U(\mathfrak{g}_*)$ and $U_*$.
              \end{proof}

 From now on, we shall identify $U_*$ and $U(\mathfrak{g}_*)$ by means of $U(\pi)$. 	With a similar proof, the following proposition  follows. 
 \begin{prop}\label{da-aggiungere:prop}
A family $(X_j)_{j\in J}$ in $\gf$ is a basis compatible with the filtration if
and only if it induces a homogeneous basis of $\gf_*$.
\end{prop}
\begin{proof}
Take a family $(X_j)$ in $\gf$, define $d_j=\deg X_j$ for every $j$,
and set $Y_j=X_j+ \gf_{d_j^-}$. 

 If $(X_j)_{j\in J}$ is a basis compatible with the
filtration, then, by definition   $(X_j)_{d_j\meg\lambda}$ is a basis
of $\gf_\lambda$ for every $\lambda\Meg1$. Consequently, $(Y_j)_{d_j=\lambda}$ is a basis of $\gf_\lambda/\gf_{\lambda^-}$, for every
$\lambda\Meg1$. In other words, $(Y_j)$ is a homogeneous basis of
$\gf_*$.

Conversely, if $(Y_j)$ is a homogeneous basis of
$\gf_*$, then $(Y_j)_{d_j=\lambda}$ is a basis of $\gf_\lambda/\gf_{\lambda^-}$, for every
$\lambda\Meg1$. Arguing by induction as in the proof of
Proposition~\ref{prop:8}, we then see that  $(X_j)_{d_j\meg\lambda}$ is a basis
of $\gf_\lambda$ for every $\lambda\Meg1$, whence the conclusion.
\end{proof}
		\subsection{Weighted subcoercive operators revisited}
                We say that a family $(X_j)_{j\in J}$ of elements of
 $\gf$ is a \emph{minimal basis} of $\mathfrak{g}$ if, setting $Y_j\coloneqq X_j+ \gf_{(\deg X_j)^-}$, the family $(Y_j)$ induces a (homogeneous) basis of $\gf_*/[\gf_*,\gf_*]$. If $(\gf_\lambda)$ is admissible, so that the $\deg(X_j)$ are integer multiples of some element of $(0,+\infty)$, then we denote with $\dd$ the least common multiple of the degrees of the $X_j$, $j\in J$.
		
     We need the following fact; we provide a proof in
    Section~\ref{A2:sec} for the reader's convenience.
	\begin{prop}\label{min-basis:pro}
		Let $(X_j)_{j\in J}$ be a minimal basis of $\gf$. Then,  $(X_j)$ generates $\gf$ as a Lie algebra and $(\gf_\lambda)$ as a filtration.
	\end{prop}

In other words, 
\[
\gf_\lambda = \mathrm{span}\, \{ [[\dots[ X_{j_1},X_{j_{2}}], \dots],  X_{j_k}] \colon k\Meg 1, \; j_1,\dots, j_k\in J, \; \deg(X_{j_1})+\cdots +\deg(X_{j_k})\meg \lambda\},
\]
and
\[
U_\lambda =  \mathrm{span}\, \{  X_{j_1}\cdots X_{j_k} \colon k\in\N, \; j_1,\dots, j_k\in J, \; \deg(X_{j_1})+\cdots +\deg(X_{j_k})\meg \lambda\}.
\]
In particular, $(X_j)$ is a \emph{reduced weighted algebraic basis} of $\gf$.
	
	Observe that $\dd$ does \emph{not} depend on the choice of $(X_j)$, since it is the least common multiple of the degrees of the non-zero elements of $\gf_*/[\gf_*,\gf_*]$.

	\begin{deff}
We say that a left-invariant differential operator $\Lc\in U(\mathfrak{g})$ of degree $\dL$ is \emph{weighted subcoercive with respect to the filtration $(\mathfrak{g}_{\lambda})$} if \begin{equation}\label{LLstarU}
	 \Lc+\Lc^*+U_{\dL^-}
\end{equation}
	 is a positive Rockland operator on $G_*$.
	\end{deff}
	
	We are now ready to prove the aforementioned equivalent characterization of weighted subcoercivity. Before doing so, let us give the following definition, whose meaning will be clarified soon (see Remark~\ref{admissibility} below).
 
	\begin{deff}
	We say that an increasing filtration $(\gf_\lambda)_{\lambda\Meg 0}$ satisfying~\eqref{eq:filtr} is \emph{admissible} if there is $\dd_{0} \geq 1$ such that, if $\lambda$ is such that $\mathfrak{g}_{\lambda} \neq \mathfrak{g}_{\lambda^{-}}$, then $\lambda \in \dd_{0}\mathbb{Q}$.
	\end{deff}

	\begin{teo}\label{teo:1}
	Let $\Lc\in U(\mathfrak{g})$ be a left-invariant differential operator of degree $\dL$ with respect to a given admissible filtration $(\mathfrak{g}_{\lambda})$. The following conditions are equivalent:
	\begin{enumerate}
	\item $\Lc$ is weighted subcoercive with respect to the filtration $(\mathfrak{g}_{\lambda})$;
	\item $\Lc$ is weighted subcoercive with respect to some reduced weighted basis compatible with $(\mathfrak{g}_{\lambda})$.
	\end{enumerate}
In addition, if one of the above conditions holds, then $\dL/\dd$ is an even integer and  $\Lc$ is weighted subcoercive with respect to every reduced weighted basis compatible with $(\mathfrak{g}_{\lambda})$, with weights $w_{1},\dots, w_{d}$  such that $\dL/w_{j} \in 2\mathbb{N}$.
	\end{teo}
	
It may be tempting to add a third equivalent condition in Theorem~\ref{teo:1} requiring that $\Lc$ be weighted subcoercive with respect to every reduced weighted basis with weights $w_{1},\dots, w_{d}$ which is compatible with $(\mathfrak{g}_{\lambda})$ and such that $\dL/w_{j} \in 2\mathbb{N}$. However, this latter condition is weaker than (1) and (2) in the statement, since there may be no reduced weighted bases compatible with $(\mathfrak{g}_{\lambda})$  and with weights $w_{1},\dots, w_{d}$  such that $\dL/w_{j} \in 2\mathbb{N}$.
		
		\begin{proof}
Let us assume (2), and prove (1). Then $\Lc$ is weighted subcoercive with respect to some reduced weighted algebraic basis $(X_j)_{j\in J}$ compatible with the filtration $(\gf_\lambda)$. Let $C$ be a weighted subcoercive form  of degree $ \dL\coloneqq \deg \Lc$ such that $\Lc = d\mathrm{R}_G(C)$. If $P$  denotes the principal part of $C$, then the operator on $\gf_*$ corresponding to the family $(X_j+\gf_{\dd_j^-})_{j\in J}$ and the form $P+P^*$ is $  \Lc+\Lc^*+U_{\dL^-}$, which is then a positive Rockland operator on $G_*$.
		
Let us now prove that (1) implies (2). Suppose now that~\eqref{LLstarU} is a positive Rockland operator on $G_*$. Let $(X_1,\dots, X_k)$ be a minimal basis of $\gf$, and $C$ be a form of degree $\dL$ which induces $\Lc$ with respect to such basis. 
		Let $P$ be the principal part of $C$. Then~\eqref{LLstarU} is the operator on $\gf_*$ corresponding to $ Y_j=X_j+\gf_{\dd_j^-}$ and the form $\smash{P+P^*}$. In order to show that $\Lc$ is weighted subcoercive with respect to $(X_j)$ with weights $\deg(X_{j})$, it suffices to show that $\smash{\dL\in 2\deg(X_j)\N}$ for every $j=1,\dots, k$. However, if $\tilde G$ is the abelianization of $G_*$ and $\widetilde \Lc$ is the (positive, cf.~\cite[Proposition 3.12]{Calzi}) differential operator on $\tilde G$ corresponding to the positive Rockland operator~\eqref{LLstarU}, then the Fourier transform of $\widetilde \Lc\delta_0$ is a  positive homogeneous polynomial which vanishes only at $0$. Considering its restriction to the `homogeneous' lines $\exp(\R Y_j+[\gf_*,\gf_*])$, we then see that  $\dL=\dd_{\widetilde \Lc}\in 2\dd_j\N$ for every $j=1,\dots, k$. 
		
Finally, assume that (1) holds and take a reduced weighted basis $X_1,\dots, X_d$ with weights $w_{1},\dots, w_{d}$ which is compatible with $(\mathfrak{g}_{\lambda})$ and such that $\dL/w_{j} \in 2\mathbb{N}$. In order to prove that $\Lc$ is weighted subcoercive with respect to $X_1,\dots, X_d$, it suffices to argue as in the proof that (2) implies (1), since the condition $\dL/w_{1},\dots,\dL/w_{d} \in 2\mathbb{N}$ is satisfied by assumption.
\end{proof}
	
\begin{teo}\label{teo:7-1}
 Let $\Lc$ be a weighted subcoercive operator with respect to some admissible filtration. If $\Lc=\Lc^*$, then $\Lc$ is essentially self-adjoint on $C^\infty_c(G)$.
	\end{teo} 
	\begin{proof}
Consider the continuous unitary representation $\pi \colon G \to L^2(\beta)$ defined by 
	\[
	\pi(x) f =\Delta_R^{1/2}(x) f(\cdot \,  x), \qquad x\in G,
	\]
	for $f\in L^2(\beta)$, and the left-invariant differential operator $\Lc'$ such that 
	\[
	\Lc' f= \Delta_R^{1/2} \Lc(\Delta_R^{-1/2} f)
	\]
	for $f\in C^\infty(G)$.  Since $\smash{\Lc'-\Lc \in U_{\dL^-}}$, also $\Lc'$ is a weighted subcoercive operator with respect to the same admissible filtration as $\Lc$. Observe that since
	\[
	\int_G \Delta_R^{1/2}(x) \phi(y x)  \overline{\psi(y)}\,\dd \beta(y) = \int_G (\Delta_R^{1/2} \phi)(y x) \Delta_R^{-1/2}(y)  \overline{\psi(y)}\,\dd \beta(y), \qquad x\in G,
	\] 
	for all $\phi,\psi\in C^\infty_c(G)$, then
		\[
		\begin{split}
		\langle \dd \pi(\Lc') \phi\vert \psi\rangle&= \Lc'\left(x\mapsto\int_G \Delta_R^{1/2}(x) \phi(y x) \overline{\psi(y)}\,\dd \beta(y)\right)(e)\\
			%&= \Lc'\left(x\mapsto\int_G (\Delta_R^{1/2} \phi)(y x) \Delta_R^{-1/2}(y) \psi(y)\,\dd \beta(y)\right)(e)\\
			&=\int_G \Lc'(\Delta_R^{1/2} \phi)(y) \Delta_R^{-1/2}(y)  \overline{\psi(y)}\,\dd \beta(y)\\
			&=\langle \Lc \phi\vert \psi \rangle
		\end{split}
		\]
so that $\dd\pi(\Lc')=\Lc$ on $C^\infty_c(G)$. 
		By~\cite[Theorem 2.3 (c)]{Martini}, $\dd \pi(\overline{\Lc'^+})^*$ is the closure of $\dd\pi(\Lc')$, initially defined on 
		\[
		C^\infty(\pi)= \{f\in L^2(\beta)\colon X f\in L^2(\beta) \,\:\:  \forall X\in U(\mathfrak{g})\}.
		\] 
		If $\Lc=\Lc^*$, then   $\dd\pi(\Lc')=\Lc$ is Hermitian on $C^\infty_c(G)$, which is dense in $C^\infty(\pi)$. This proves that $\Lc$, with initial domain $C^\infty_c(G)$, is essentially self-adjoint on $L^2(\beta)$.
	\end{proof}

It seems natural to compare the theory developed in the present paper and the one of~\cite{BPV,BPV2,BPV3}. We do so by means of the following proposition -- see also Remark~\ref{comparison} below. Here and in the following, we say that a differential operator $T$ is positive if $\langle T f | f\rangle \Meg 0$ for every $f\in C^\infty_c(G)$.
	
	\begin{prop}\label{prop:4}
Let  $(\gf_{\lambda})$ be an admissible filtration of $\gf$ and $(X_j)_{j\in J}$ be a minimal basis of $\gf$ with $d_j=\deg X_j$ for $j\in J$. Then 
\begin{equation}\label{Lclink}
		\Lc = \sum_{j\in J} (X_j^{k\dd/d_j})^\dag X^{k\dd/d_j}_j
\end{equation}
is a real,  positive and formally self-adjoint  weighted subcoercive operator with respect to $(\gf_{\lambda})$ for every integer $k\Meg 1$.
	\end{prop}
	
	\begin{proof}
		It is clear that $\Lc$ is positive, formally self-adjoint,  and real. Let us prove that it is weighted subcoercive with respect to $(\mathfrak{g}_{\lambda})$. Denoting with $(Y_j)$ the   family of (homogeneous) elements of $\gf_*$ corresponding to $(X_j)$, it will suffice to show that 
		\[
		\Lc'\coloneqq \sum_{j\in J} (-1)^{k\dd/d_j} Y_j^{2k \dd/d_j}
		\]
		is a positive Rockland operator on $G_*$. Positivity is clear. Then, let $\pi$ be a non-trivial irreducible continuous unitary representation of $G$ in some Hilbert space $H$, and let us prove that $\dd \pi(\Lc')$ is one-to-one on $C^\infty(\pi)$. Take $v\in C^\infty(\pi)$, and assume that $\dd\pi(\Lc')v=0$. Then
		\[
		0=\langle \dd \pi(\Lc')v\vert v \rangle= \sum_{j\in J} \norm{ \dd \pi(X_j)^{k\dd/d_j} v}_H^2,
		\]
		so that $\dd\pi(X_j)^{k\dd/d_j}v=0$ for every $j\in J$. This implies that $\dd \pi(X_j) v=0$ (note that the mapping $\R \ni t \mapsto \pi(\exp(t X_j))v$ must be a polynomial of degree $<k \dd/d_j$; since it is bounded, it must be constant, whence $\dd\pi(X_j)v=0$). Thus $\dd \pi(Y) v=0$ for every $Y$ in the Lie algebra generated by $(X_j)$, that is, for every $Y\in \gf_*$. Since $\pi$ is irreducible and non-trivial, this means that $v=0$. 
	\end{proof}

	   \begin{oss}\label{comparison}
If $\gf_1$ is a vector subspace of $\gf$ which generates $\gf$ as a Lie algebra, and if $(\gf_\lambda)$ is the filtration generated by $\gf_1$ (that is, $\gf_\lambda$ is the vector space generated by the elements of $\gf_1$ and their commutators up to order $[\lambda]$ for $\lambda\Meg 1$), then any basis of $\gf_1$ is a minimal basis (and conversely), and the operator $\Lc$ in~\eqref{Lclink} is a sub-Laplacian, with drift unless $\beta$ is right-invariant. In particular, recall Remarks~\ref{comparison1} and~\ref{comparison0}, for any H\"ormander basis $(X_j)_{j\in J}$ of $\mathfrak{g}$ and any positive character $\chi$, the sub-Laplacian with drift on $L^{2}(\mu_{\chi})$, where $\dd\mu_{\chi} = \chi \dd\rho$,
\[
\Delta_{\chi} = -\sum_{j\in J}(X_{j}^{2} + (X_{j}\chi)(e) X_{j})
\]
is weighted subcoercive with respect to the above filtration (which assigns the weight $d_{j}=1$ to each $X_{j}$). Therefore, the setting developed in~\cite{BPV,BPV2,BPV3} falls under the one developed in the present paper.
	   \end{oss}

\begin{oss}\label{admissibility}
	Let us stress that if $(\mathfrak{g}_{\lambda})$ is an increasing filtration which is not admissible, then there exists no weighted subcoercive operator with respect to $(\mathfrak{g}_{\lambda})$. Indeed, as a consequence of~\cite[Proposition 1.3]{Miller}, if $(\gf_\lambda)$ is not admissible, then there cannot be any Rockland operators on $\gf_*$.
	\end{oss}

		   \subsection{Control distance}\label{subsection:controldistance} Let $X_1,\dots,X_k$ be an orthonormal basis of $\mathfrak{g}$, compatible with the filtration, and set $d_{j}= \mathrm{deg}(X_{j})$. For $s \in \{0,\infty,*\}$ and $\varepsilon > 0$, let $C_s(\varepsilon)$ be the set of absolutely continuous curves $\gamma : [0,1] \to G$ such that
\begin{equation}\label{gammacurve}
\gamma'(t) = \sum_{j=1}^k \phi_j(t) \, X_j|_{\gamma(t)} \qquad\text{for a.e.\ $t \in [0,1]$,}
\end{equation}
where for $t \in [0,1]$ and $j=1,\dots,k$
\[
|\phi_j(t)|  \leq  \begin{cases}
\varepsilon^{d_j} & \text{if $s = 0$,}\\
\varepsilon       & \text{if $s = \infty$,}\\
\min\{\varepsilon,\varepsilon^{d_j}\} & \text{if $s = *$.}
\end{cases}
\]
For $x,y \in G$, we define then
\[d_s(x,y) = \inf \{\varepsilon > 0 \colon \exists \gamma \in C_s(\varepsilon) \text{ such that } \gamma(0) = x, \, \gamma(1) = y\}.\]

Then $d_0$, $d_\infty$ and $d_*$ are left-invariant distances on $G$, compatible with the topology of $G$. In fact, $d_\infty$ is the genuine Carnot--Carath\'eodory distance associated with the H\"ormander system $X_1,\dots,X_k$, while $d_0$ is a ``weighted'' version thereof. Recall that $d_0(x,y) \leq 1$  if and only if $d_\infty(x,y) \leq 1$ for $x,y\in G$, and the same holds with strict inequalities. Finally,
\[
d_*(x,y) = \begin{cases}
d_0(x,y) &\text{for $d_0(x,y) \leq 1$,}\\
d_\infty(x,y) &\text{for $d_0(x,y) \geq 1$.}
\end{cases}
\]
The distance $d_*$ is said to be a \emph{control distance} on the weighted Lie group $G$. We denote by
\[
|x|_* = d_*(x,e)
\]
the control modulus $|\cdot|_*$ on $G$ induced by $d_{*}$. Let us recall from~\cite[Section 2.3]{Martini} that if $B(x,r)$ denotes the ball centered at $x\in G$ with radius $r>0$ with respect to $d_{*}$, then
	\[
	\beta(B(e,r))\asymp r^{Q_*}, \qquad r\to 0^+,
	\]
where $Q_*$ is the homogeneous dimension of the contraction $\mathfrak{g}_*$, while
	\[
	\beta(B(e,r))\meg \ee^{C r}
	\]
	for some constant $C>0$ and for every $r\Meg 1$; see Lemma~\ref{lem:31} below. More precisely, the rate of growth of $\beta_{R}(B_r)$ for large radii coincides with the (intrinsic) volume growth of $G$ at infinity, so that if $G$ has polynomial growth of degree $D$, then $\beta_{R}(B_r) \asymp r^{D} $ for $r \geq 1$. 

        \begin{oss}\label{distance:rem} The advantage in assuming
          that $\lambda_1\geq 1$ is that
          $d_*$ is indeed a \emph{distance} on $G$,
          rather than a quasi-distance. Indeed, take for instance $G=\mathbb R$
          so that $X$ is just the first derivative, with assigned degree
          $\lambda$. Then $d_*(x,y)=|x-y|^{1/\lambda}$, which is a
          distance if and only if $\lambda \geq 1$. \end{oss}

	\begin{oss}
	Note that the distance $d_{*}$ does depend on the choice of the basis $X_1,\dots,X_k$, even though different choices of bases provide bi-Lipschitz equivalent distances; see~\cite[Corollary 6.5]{ElstRobinson} for $d_{0}$, while the equivalence of $d_{*}$ follows from the connectedness of the distance, cf.~\cite[Proposition III.4.2]{VSCC}. However, one can define a distance which depends only on the filtration and the scalar product which is bi-Lipschitz equivalent to those above, as follows.

	Fix an orthonormal basis $X_1,\dots,X_k$ which is compatible with the filtration. Given an absolutely continuous curve $\gamma\colon [0,1]\to G$ as in~\eqref{gammacurve}, we define its content $\mathcal{C}(\gamma)$ as
	\[
	\mathcal{C}(\gamma) = \inf\bigg\{ \varepsilon>0 \colon \Big( \sum_{j\colon \!\! \deg(X_{j})=\lambda} \!\!\!\!\!\varphi_{j}(t)^{2} \Big)^{1/2}\leq \min(\epsilon, \epsilon^{\lambda}) \mbox{ for a.e.\ $t\in [0,1]$ and every $\lambda>0$} \bigg\}.
		\]
For $x\in G$, then 
\[
\abs{x}_*' := \inf \big\{ \mathcal{C}(\gamma) \colon \mbox{$\gamma\colon [0,1]\to G$ is absolutely continuous, $\gamma(0)=e$ and $\gamma(1)=x$}\big\}
\]
is a control modulus, inducing the left-invariant distance on $G$ given by	 $d'(x,y)\coloneqq \abs{y^{-1}x}'_*$.
\end{oss}

	\smallskip
	
	In analogy to~\cite[Proposition 5.7 (ii)]{HMM},  see also~\cite[Lemma 2.3]{BPTV}, one can show that any character of $G$ grows at most exponentially in terms of the control modulus $| \cdot |_{*}$. The characters we shall be mostly interested in are $\Delta_L$ and $\Delta_R$.
	
		\begin{lem}\label{lem:31}
		Let $\chi\colon G\to \C\setminus \{0\}$ be a character of $G$. Then, there is $c>0$ such that 
		\[
		\abs{\chi(x)}\meg \ee^{c(\abs{x}_*+1)} \qquad \forall \, x\in G.
		\]
		As a consequence, there exists $C>0$ such that $\beta(B(e,r))\meg \ee^{C r}$ for every $r\geq 1$.
	\end{lem}

	\begin{proof}
		Set $c\coloneqq \max_{\overline B(e,1)} \log \abs{\chi}$. Observe that for every $x\in G$ we may find an absolutely continuous curve $\gamma\colon [0,1]\to G $ such that $\gamma(0)=e$ and $\gamma(1)=x$, and with content at most $\eps=[\abs{x}_*]+1$. Then, $y_j\coloneqq \gamma((j-1)\eps)^{-1}\gamma(j/\eps)\in \overline B(e,1)$ for every $j=1,\dots, \eps$ and $x=y_1\cdots y_\eps$, so that 
		\[
		\abs{\chi(x)}=\prod_{j=1}^\eps \abs{\chi(y_j)}\meg \ee^{c \eps}\meg \ee^{c(\abs{x}_*+1)}.
		\]
	We conclude that, since $\Delta_{R}$ is a character on $G$, for some $c,C>0$
		\[
	\beta(B(e,r)) = \int_{\abs{x}_* \leq r} \Delta_{R}(x)\, d\beta_{R}(x) \leq \ee^{c(r+1)} \beta_{R}(B(e,r)) \leq \ee^{C r}
		\]
		for all $r\geq 1$, the last estimate by~\cite[(2.2)]{BPTV}.
	\end{proof}

	\subsection{The Schwartz space on $G$} Let us introduce now a space of functions on $G$ which is the analogue of the Euclidean Schwartz space (wherefrom the name). See also~\cite{BPV}.

\begin{deff}\label{def:1}
	We define $\Sc(G)$ as the space
	\[
	\Sc(G) = \{f\in C^\infty(G) \colon \norm{\ee^{c \abs{\,\cdot\,}_*} X^R f}_{L^1(\beta)}<+\infty \;\; \forall c>0, \; \forall  X\in U(\mathfrak{g})\}
	\]
endowed with the natural topology. We denote with $\Sc'(G)$ its dual, endowed with the topology of uniform convergence on the bounded subsets of $\Sc(G)$.
\end{deff}
Observe that by the arbitrariness of $c$, by Lemma~\ref{lem:31} one may replace $\beta$ with $\beta_L$ in Definition~\ref{def:1}. Consequently, the results in~\cite{Schweitzer} are applicable in this context.
	
\begin{teo}\label{teo:8}
	The following hold.
	\begin{enumerate}
		\item[\textnormal{(1)}] $\Sc(G)$ is a nuclear Fréchet space.
		
		\item[\textnormal{(2)}] $\Sc(G)$ is reflexive.
		
		\item[\textnormal{(3)}] The bounded subsets of $\Sc(G)$ are relatively compact.
		
		\item[\textnormal{(4)}] $\Sc(G)$ is a Fréchet $*$-algebra under convolution and under pointwise multiplication.
		
		\item[\textnormal{(5)}] For every $p\in [1,\infty]$, 
\begin{align*}
	\Sc(G) 
	&= \{f\in C^\infty(G) \colon \norm{\ee^{c \abs{\,\cdot\,}_*} X Y^R f}_{L^p(\beta)}<+\infty \;\; \forall c>0, \; \forall  X,Y\in U(\mathfrak{g})\}\\
	& = \{f\in C^\infty(G) \colon \norm{\ee^{c \abs{\,\cdot\,}_*} X f}_{L^p(\beta)}<+\infty \;\; \forall c>0, \; \forall  X\in U(\mathfrak{g})\}\\
	& = \{f\in C^\infty(G) \colon \norm{\ee^{c \abs{\,\cdot\,}_*} X^{R} f}_{L^p(\beta)}<+\infty \;\; \forall c>0, \; \forall  X\in U(\mathfrak{g})\},
	\end{align*}
			and all the associated topologies coincide.
	\end{enumerate} 
\end{teo}

\begin{proof}
	Statement~(1) follows from~\cite[Theorem 6.24]{Schweitzer}, and (2)--(3) follow from (1) and~\cite[Proposition 50.2, Theorem 36.4, and Corollary 1 to Proposition 33.2]{Treves}. The first assertion in~(4) follows from~\cite[Definition 1.3.9 and Theorem 1.3.13]{Schweitzer}, and the second follows from the case $p=\infty$ in (5).
	
	It remains to show~(5). Since $\Sc(G)$ is a $*$-algebra under convolution by the first assertion in~(4), $\Sc(G)$ is the space of $f\in C^\infty(G)$ such that the seminorms  $\norm{\ee^{c \abs{\,\cdot\,}_*}  Y f}_{L^1(\beta)}$, $c>0$, $ Y\in U(\mathfrak{g})$, are finite. Since the map $f\mapsto X^R f$ is an endomorphism of $\Sc(G)$ for every $X\in U(\mathfrak{g})$,   the assertion follows when $p=1$. The assertion for general $p$ follows from~\cite[Theorem 6.24]{Schweitzer} and the above remarks.
\end{proof}

\begin{cor}
	$\Sc'(G)$ is a complete, reflexive, bornological, and nuclear space.  The bounded subsets of $\Sc'(G)$ are relatively compact.
\end{cor}

\begin{proof}
	Combining (1) and (3) of Theorem~\ref{teo:8} with~\cite[Corollary to Proposition 39.10 and Proposition 36.10]{Treves}, one sees that $\Sc'(G)$ is reflexive, and that its bounded subsets are relatively compact. Combining (2) of Theorem~\ref{teo:8} with~\cite[Proposition 50.6]{Treves} we see that $\Sc'(G)$ is nuclear. Completeness follows from~\cite[Chapter III, Proposition 2 of \S 2 and Corollary 1 to Proposition 12 of \S 3, No.\ 8]{BourbakiTVS}. The fact that $\Sc'(G)$ is bornological follows from~\cite[Corollary to Proposition 4 of Chapter IV, \S 3, No.\ 4]{BourbakiTVS}.
\end{proof}

	\section{Heat kernel and gaussian-like estimates}\label{heatker:sec}
\underline{From this point on,} we fix an admissible filtration $(\gf_{\lambda})$ and a weighted subcoercive operator $\Lc$ of degree $\dL$. When we say that an operator is weighted subcoercive, we shall mean with respect to $(\gf_{\lambda})$. Note that such operators exist, thanks to Proposition~\ref{prop:4}. For $\omega\in \R$, we shall set $\Lc_\omega \coloneqq \Lc+ \omega I$.

	 \medskip
	
	By~\cite[Theorem 2.3]{Martini}, the closure of $\Lc$ on $C_{c}^{\infty}(G)$ generates a semigroup of operators $(\ee^{-t\Lc})_{t>0}$, which satisfies gaussian-like estimates together with its derivatives. More precisely, there is a family $(k_t)_{t>0}$ of smooth functions on $G$ such that for all $f\in L^2(\beta)$ 
\begin{equation}\label{etLkt}
			\ee^{-t \Lc} f(x)= \int_G f(xy^{-1})k_{t}(y) \Delta_{R}^{-1/2}(y) \Delta_{R}(y^{-1})\,\dd \beta(y), \qquad x\in G, \, t>0,
	\end{equation}
			and such that for every $\lambda\Meg 0$ there are $\omega_\lambda,b_\lambda,C_\lambda>0$ such that
\begin{equation}\label{estkt}
			\abs{X k_t(x)}\meg C_\lambda \abs{X} t^{-\lambda /\dL} \ee^{\omega_\lambda t}  t^{-Q_*/\dL} \ee^{-b_{\lambda}   (\abs{x}_*^{\dL}/t)^{1/(\dL-1)}}
	\end{equation}
			for every $t>0$, $x\in G$ and $X\in U_\lambda$.
		 
\begin{deff}
	We denote with $(h_t)_{t>0}$ the convolution kernel of $(\ee^{-t\Lc})_{t>0}$, namely the family of functions such that  $\ee^{-t\Lc}f=f*h_t$ for all $f\in L^2(\beta)$,	 and call it \emph{heat kernel} of $\Lc$. We adopt the convention $h_0=\delta_e$. 
	\end{deff}
Notice that by~\eqref{etLkt},	for all $f\in L^2(\beta)$	
\[
			\ee^{-t \Lc} f=f*(\Delta_R^{-1/2} k_t), \qquad t>0,
			\]
whence for all $t>0$
			\[
			h_t=\Delta_R^{-1/2} k_t.
			\]
Our next goal is to obtain pointwise gaussian-like estimates for the heat kernel $h_t$, out of those in~\eqref{estkt} for $k_t$. To simplify the notation in what follows, we also give the following.
\begin{deff}
		For every $b,t>0$ and $d>1$, define 
		\[
		p_{b,t,d} (x) \coloneqq t^{-Q_*/d} \ee^{-b   (\abs{x}_*^{d}/t)^{1/(d-1)}}, \qquad x\in G,
		\]
		and for $f\in \Sc'(G)$
		\[
		T_{b,t,d}f\coloneqq \abs{f}*p_{b,t,d}.
		\]
		We shall simply write $T_{b,t}$ and $p_{b,t}$ instead of $T_{b,t,\dL}$ and $p_{b,t,\dL}$, respectively.
	\end{deff}
	
	\begin{oss}
Here and in what follows, by $T_{b,t} f$ for a general $f\in \Sc'(G)$ we mean
\[
T_{b,t} f : = \sup \big\{\abs{f*\phi} \colon \phi \in C^\infty_c, \, \,  \abs{\phi}\meg p_{b,t}\big\}.
\]
If $f$ is a measure, then
\[
(T_{b,t} f)(x)=\int_G  p_{b,t}(y^{-1}x) \Delta_L(y^{-1})\,\dd \abs{f}(y)=(\abs{f}*p_{b,t})(x),
\] 
otherwise $T_{b,t} f$ is identically infinite (indeed, as $p_{b,t}$ is locally bounded from below, if $(T_{b,t} f)(x)$ is finite for some $x\in G$, then $f$ is a Radon measure). Analogously, for $f\in \Sc'(G)$ and $p\in [1,\infty]$  we shall write
\[
\|f\|_{L^{p}(\beta)} : = \sup \big\{  |\langle f,\phi\rangle | \colon \phi \in C_{c}^{\infty}(G), \, \, \|\phi\|_{L^{p'}(\beta)} \leq 1\big\}.
\]
Note that if $\|f\|_{L^{p}(\beta)} $ is finite and $p>1$, then $f \in L^{p}(\beta)$, whereas if $p=1$ then $f$ is a finite measure.
\end{oss}
	
Then, we shall prove the following.
	\begin{teo}\label{teo:7}
There is $\omega\in \R$ such that for every $\lambda\Meg 0$ there are $b,C>0$ such that  
\[
			\abs{X Y^R h_t(x)}\meg C \abs{X}\abs{Y} t^{-(Q_* + \deg X+\deg Y)/\dL} \ee^{\omega t} \ee^{- b (\abs{x}_*^{\dL}/t)^{1/(\dL-1)}},
\]
equivalently
\begin{equation}\label{eq:teo7new}
			\abs{X Y^R h_t(x)}\meg C \abs{X}\abs{Y} t^{-(\deg X+\deg Y)/\dL} \ee^{\omega t} \, p_{b,t,\dL}(x),
	\end{equation}
	for all $X,Y\in U_\lambda$, $x\in G$ and $t>0$. In particular, for every $\lambda\geq 0$ there are $b,C>0$ such that
	\[
	\abs{X \ee^{-t\Lc} Yf(x) } \meg  C |X||Y| t^{-(\deg X+\deg Y)/\dL}\ee^{\omega t}\, T_{b,t,\dL}f(x)
	\]
for all $X,Y\in U_\lambda$, $x\in G$, $t>0$, and $f\in \mathcal{S}'(G)$.
			\end{teo}	
Let us stress that unlike in~\eqref{estkt}, in~\eqref{eq:teo7new} $\omega$ is independent of the orders of differentiation. The proof of Theorem~\ref{teo:7} requires some preliminary estimates of the functions $p_{b,t,d}$ which are the object of the next few results.

	\begin{lem}\label{lem:18}
		Suppose $b\Meg b'>0$ and $d'\Meg d>1$. Then
		\[
		p_{b,t,d}\meg \ee^{b'} p_{b',t^{d'/d},d'} \qquad \forall \, t>0.
		\]
	\end{lem}
	
	\begin{proof}
		It suffices to observe that
		\[
		\begin{split}
			p_{b,t^d,d}(x)= t^{-Q_*} \ee^{-b (\abs{x}/t)^{d/(d-1)}}\meg t^{-Q_*}\ee^{-b' [(\abs{x}/t)^{d'/(d'-1)}-1]}=\ee^{b'} p_{b', t^{d'},d'}(x)
		\end{split}
		\]
		for every $x\in G$ and  $t>0$, since $d/(d-1)\Meg d'/(d'-1)$.
	\end{proof}
	
	\begin{lem}\label{lem:32}
Suppose $b>0$ and $d>1$. Then
\begin{align*}
		\ee^{-2^{1/(d-1)}b(\abs{y^{-1}x}^d/t)^{1/(d-1)}}&(T_{2^{1/(d-1)}b,t,d} f)(y) \\
		&  \quad \meg (T_{b,t,d} f)(x)\meg \ee^{2^{1/(d-1)}b(\abs{y^{-1}x}^d/t)^{1/(d-1)}} (T_{2^{-1/(d-1)}b,t,d} f)(y)
\end{align*}
		for every $x,y\in G$ and $f\in \Sc'(G)$.
		\end{lem}
	
	\begin{proof}
		By symmetry, it suffices to prove the left inequality, which follows from the elementary estimate
		\[
		\abs{z^{-1}y}_*^{d/(d-1)}\meg (\abs{z^{-1}x}_*+\abs{x^{-1}y}_*)^{d/(d-1)}\meg  2^{1/(d-1)}(\abs{z^{-1}x}_*^{d/(d-1)}+\abs{x^{-1}y}_*^{d/(d-1)})
		\]
		for every $x,y,z\in G$.
	\end{proof}

	\begin{lem}\label{lem:3}
	Suppose $b,\rho >0$ and $d>1$. Then, there is a constant $C>0$ such that
\begin{equation}\label{est:lem3}
		\| \ee^{\rho \abs{\,\cdot\,}_*} p_{b,t,d} \|_{L^p(\beta)}\meg C t^{-Q_*/(d p')}\ee^{C t}
\end{equation}
		for every $t>0$ and $p\in [1,\infty]$. In particular, for all  $p\in [1,\infty]$ the map $t \mapsto \ee^{\rho\abs{\,\cdot\,}_*}p_{t,b,d}$ 		is continuous from $(0,+\infty)$ to $L^p( \beta)$.
	\end{lem}
	
	\begin{proof}
		The case $p=1$ of estimate~\eqref{est:lem3} is contained in the proof of~\cite[Theorem 1.4.1 (f)]{MartiniTesi} (cf.~\cite[Theorem 2.3]{Martini}).  It will then suffice to consider the case $p=\infty$. Then, take $\omega>0$ and observe that the function 
		\[
		\phi(t)= b \abs{x}_*^{d/(d-1)} t^{-1/(d-1)}-\rho\abs{x}_* + \omega t, \qquad t>0,
		\]
		attains its minimum at $t=\abs{x}_* (b/((d-1)\omega))^{1-1/d}$, so that 
		\[
		\min_{t>0} \phi(t)=\abs{x}_*( [\omega(d-1)]^{1/d}b^{1-1/d} d/(d-1)-\rho  ).
		\]
		Consequently, if we take $\omega$ sufficiently large, then for every $x\in G$ and $t>0$
		\[
		p_{b,t,d}(x) \ee^{\rho\abs{x}_*}\meg t^{-Q_*/d} \ee^{\omega t}.
		\]
The continuity of the map $t \mapsto \ee^{\rho\abs{\,\cdot\,}_*}p_{t,b,d}$ is a consequence of~\eqref{est:lem3}  and the dominated convergence theorem, when $p<\infty$. The case $p=\infty$ follows easily.
	\end{proof}
	
		The next result will allow us to estimate $p_{b,t}*p_{b',t'}$ in terms of $p_{ b'',t+t'}$, hence $T_{b,t} T_{b',t'}$ in terms of $T_{b'',t+t'}$, and in turn $p_{  b'',t+t'}$ in terms of $p_{b,t}*p_{b',t'}$, hence $T_{b'',t+t'}$ in terms of $T_{b,t} T_{b',t'}$.

	\begin{lem}\label{lem:4}
	Suppose $b,b', \rho>0$, $d>1$. 
\begin{enumerate}
\item If  $0<b''<2^{-1/(d-1)}\min(b,b')$, then there is a constant $C>0$ such that %for $x\in G$ and $t,t'>0$
		\[
		\qquad \qquad \int_G p_{b,t,d}(x y^{-1}) p_{b',t',d}(y) \ee^{\rho\abs{y}_*}\,\dd \beta(y) \meg C \ee^{C (t+t')} p_{ b'',t+t',d}(x), \qquad x\in G, \, t,t'>0.
		\]
\item If $b''> 2^{1/(d-1)}\max(b,b')$, then, there is a constant $C>0$ such that
		\[
		\qquad \qquad \int_G p_{b,t,d}(x y^{-1}) p_{b',t',d}(y) \ee^{-\rho\abs{y}_*}\,\dd \beta(y) \Meg C \ee^{-C(t+t')} p_{ b'',t+t',d}(x) \qquad x\in G, \, t,t'>0.
		\]
		
\end{enumerate}	
		\end{lem}

	\begin{proof}
	We start with (1), and observe first that we may assume that $\rho=0$, since Lemma~\ref{lem:3} shows that, if $\tilde b'\in ( 2^{1/(d-1)}b'',b')$, then  there is $C_1>0$ such that
		\[
		p_{b',t',d}(y) \ee^{\rho\abs{y}_*}\meg C_1 p_{\tilde b',t',d}(y) \ee^{C_1 t'}\meg C_1 p_{\tilde b',t',d}(y) \ee^{C_1 (t+t')}
		\]
		for every $t,t'>0$ and $y\in G$. By means of Lemma~\ref{lem:31} and the previous argument, we may further assume that $\Delta_R=1$. Then, observe  that
		\[
		\begin{split}
		\int_G p_{b,t,d}(x y^{-1}) p_{b',t',d}(y) \,\dd \beta(y)&=(p_{b,t,d}*p_{b',t',d})(x)\\
			&= \int_G p_{b,t,d}(y) p_{b',t',d}(y^{-1}x)  \Delta_L(y^{-1})\,\dd \beta(y)\\
			&=\int_G p_{b,t,d}(y) p_{b',t',d}(x^{-1}y^{-1})  \,\dd \beta(y)
		\end{split}
		\]
		where the last equality follows from the symmetry of $p_{b',t',d}$   and from the fact that $\Rc \beta=\Delta_L^{-1}\cdot \beta$ since $\beta$ is right-invariant. By the symmetry of $p_{b'',t+t',d}$, we may therefore assume that $t\Meg t'$, so that $t\meg t+t'\meg 2t $. Then, observe that 
\begin{multline*}
		b   (\abs{x y^{-1}}_*^{d}/t)^{1/(d-1)}+ b'   (\abs{y}_*^{d}/t')^{1/(d-1)} \\
		\Meg (b'-2^{1/(d-1)}b'')  (\abs{y}_*^{d}/t')^{1/(d-1)}+  b''   (\abs{x}_*^{d}/(t+t'))^{1/(d-1)}
	\end{multline*}
		for every $x,y\in G$ and $t,t'>0$. The assertion then follows from Lemma~\ref{lem:3}.
	
	We now prove~(2). In this case too we may assume that $\rho=0$, since Lemma~\ref{lem:3} shows that, if $\tilde b'\in (b', 2^{-1/(d-1)}b'')$, then  there is $C_1>0$ such that
		\[
		p_{\tilde b',t',d}(y) \ee^{\rho\abs{y}_*}\meg C_1 p_{ b',t',d}(y) \ee^{C_1 t'}\meg  C_1   p_{  b',t',d}(y) \ee^{C_1 (t+t')}
		\]
		for every $t,t'>0$ and $y\in G$. As in the proof of~(1), we may then also assume that $\Delta_R=1$.
		Then, as in the proof of~(1),
		\[
		\int_G p_{b,t,d}(x y^{-1}) p_{b',t',d}(y)  \,\dd \beta(y)= \int_G p_{b,t,d}(y) p_{b',t',d}(x^{-1}y^{-1})   \,\dd \beta(y).
		\]
		We may therefore assume that $t\Meg t'$, so that $t\meg t+t'\meg 2t $. 
		Take $\eps\in (0,1)$ so that $2^{-1/(d-1)}b''\Meg (\eps+1)^{d/(d-1)}b$, and observe that there is a constant $C_1>0$ such that 
		\[
		\int_{B(e,\eps t'^d)} p_{b',t',d}(y)\,\dd \beta(y)\Meg t'^{-Q_*/d}\beta(B(e,\eps t'^{1/d})) \ee^{-b'}\Meg C_1
		\]
		for every $t'>0$. In addition, if $\abs{x}_*\Meg t^{1/d}\Meg t'^{1/d}$, then for every $y\in B(e,\eps t'^{1/d})$
		\[
		b(\abs{x y^{-1}}^d_*/t)^{1/(d-1)}\meg b(1+\eps)^{d/(d-1)} ( \abs{x}^d_*/t)^{1/(d-1)}\meg   b''( \abs{x}^d_*/(t+t'))^{1/(d-1)},
		\]
		whereas, if $\abs{x}_*\meg t^{1/d}$, then for every $y\in B(e,\eps t'^{1/d})$
		\[
		b(\abs{x y^{-1}}^d_*/t)^{1/(d-1)}\meg b 2^{d/(d-1)}.
		\]
		The assertion follows.
	\end{proof}
	
	We are now ready to prove Theorem~\ref{teo:7}.
	
	\begin{proof}[Proof of Theorem~\ref{teo:7}.]
Using the fact that $X \Delta_R^{-1/2}=(X \Delta_R^{-1/2})(e) \Delta_R^{-1/2}$ for every $X\in U(\mathfrak{g})$ (cf.~Proposition~\ref{prop:9}), by~\eqref{estkt} and Lemma~\ref{lem:3} we see that for every $\lambda\Meg 0$ there are $\omega'_\lambda,b'_\lambda,C'_\lambda>0$ such that
		 \[
		 \abs{X h_t}\meg   C'_\lambda \abs{X}t^{-\lambda /\dL} \ee^{\omega'_\lambda t} p_{b'_\lambda,t}
		 \]
		 for every $t>0$ and $X\in U_\lambda$. Now, observe that $\Lc^{**}=\Lc$, so that also $\Lc^*$ is a weighted subcoercive operator. In particular, $\ee^{-t\Lc^*}$ is the adjoint of $\ee^{-t \Lc}$ on $L^2(\beta)$, so that, denoting with $(h'_t)$ the heat kernel associated with $\Lc^*$,
		 \[
		 \langle f\vert g*( \Delta_L^{-1} \widecheck{\overline{h'_t}} )\rangle=\langle f*h'_t\vert g \rangle=\langle \ee^{-t \Lc^*} f\vert g \rangle =\langle f\vert \ee^{-t \Lc} g\rangle= \langle f \vert g*h_t\rangle
		 \]
		 for every $f,g\in L^2(\beta)$, so that $h_t=\Delta_L^{-1} \widecheck{\overline{h'_t}} $. In particular, if $\lambda\Meg 0$ and $X\in U_\lambda$, then there is a unique $Y\in U_\lambda$ such that $Y^R f=\Delta_L X^R (\Delta_L^{-1} f) $ for every $f\in C^\infty(G)$, so that 
		 \[
		 \begin{split}
		 X^R h_t&= X^R(\Delta_L^{-1} \widecheck{\overline{h'_t}})\\
		 	&= \Delta_L^{-1} Y^R \widecheck {\overline{h'_t}}=\Delta_L^{-1} \widecheck{\overline{\overline {Y^+} h'_t}}.
		 \end{split}
		 \] 
		 Consequently, by means of Lemma~\ref{lem:3}  and the previous arguments, applied to $\Lc^*$, we see that there are  $\omega''_\lambda,b''_\lambda,C''_\lambda>0$ such that
		 \[
		 \abs{X^R h_t}\meg   C''_\lambda \abs{X}t^{-\lambda /\dL} \ee^{\omega''_\lambda t} p_{b''_\lambda,t}
		 \]
		for every $t>0$. Then, for $\lambda \geq 0$ take $X,Y\in U_\lambda$, and observe that 
		\[
		X Y^R h_t= Y^R h_{t'}* h_{t-2t'}*X h_{t'},
		\] where $t'=\min(1,t/3)$. Since for all $\varepsilon>0$ there exists $C=C_{\lambda,\epsilon}>0$ such that
		\[
		 t^{\eta /\dL}\min(1,t)^{-\eta  /\dL} \leq C \ee^{ \eps t} \qquad \forall t > 0,\, \eta \in [0, 2\lambda],
		\]
	by Lemma~\ref{lem:4}~(1) for every $\varepsilon>0$ and $\lambda\Meg 0$ there are $C=C_{\lambda,\varepsilon}, b=b_{\lambda,\varepsilon}>0$ such that
		\[
		\abs{X Y^R h_t}\meg C \abs{X} \abs{Y} t^{-(\deg X+\deg Y) /\dL} \ee^{(\omega_{0}+\varepsilon) t} p_{b,t}
		\]
		for every $t>0$ and $X,Y\in U_\lambda$, where $\omega_{0}$ is that of~\eqref{estkt}. The proof is complete.
	\end{proof}

	\begin{deff}
		Suppose $\omega\in \R$, $\rho\Meg 0$, $ b>0$, and $d>1$. Then, we define 
		\[
		T^{\omega,\rho}_{b,*,d} f(x)\coloneqq  \sup_{t>0} \ee^{-\omega t}(\abs{f}*(\ee^{\rho\abs{\,\cdot\,}_*} p_{b,t,d}))(x)
		\]
		for every $f\in \mathcal{S}'(G)$ and $x\in G$. We simply write  $T^{\omega}_{b,*,d}$ if $\rho=0$. Given $\kappa>0$, we also define
		\[
		T^*_{b,\kappa,d} f(x) \coloneqq  \sup_{t\in (0,\kappa]}  (T_{b,t,d} f)(x)
		\]
		for every $f\in \mathcal{S}'(G)$ and $x\in G$.  In both cases, we shall generally omit $d$ if it is $\dL$.
	\end{deff}

We have the following lemma. 	Part of its proof is inspired by that of~\cite[Corollary 4.10]{MartiniMedaVallarino}.	 
		\begin{lem}\label{lem:5}
Suppose $b>0$,  $d>1$, $\rho \Meg 0$, $p,q\in (1,\infty)$, and $(Y,\Mf, \mi)$ is a $\sigma$-finite measure space. Then there are $\omega,C>0$ such that 
		\[
		\big\| \| (T^{\omega,\rho}_{b,*,d} f(y,\,\cdot\,))(x) \|_{L^{q}_{y}(\mu)} \big\|_{L^{p}_{x}(\beta)}\meg C \big\| \| f(x,y) \|_{L^{q}_{y}(\mi)} \big\|_{L^{p}_{x}(\beta)}
		\]
		for every   $f\in L^{q,p}(\mi,\beta)$.	In particular, the same holds for $T^*_{b,\kappa,d}$ for every $\kappa>0$. 
	\end{lem}
	
	\begin{proof}
		Let us first prove that the mapping $(y,x)\mapsto  (T^{\omega,\rho}_{b,*,d} f(y,\,\cdot\,))(x)$ is $(\mi\otimes \beta)$-measurable for every $(\mi\otimes \beta)$-measurable function $f$ on $Y\times G$. Observe first that there is a $\nu$-negligible subset $N$ of $Y$ such that $f(y,\,\cdot\,)$ is $\beta$-measurable for every $y\in Y\setminus N$. Consequently, using either the dominated convergence theorem or Fatou's lemma, we see that the mapping
		\[
		(0,+\infty)\ni t \mapsto \ee^{-\omega t}(\abs{f}*(\ee^{\rho\abs{\,\cdot\,}_*} p_{b,t,d}))(x)\in [0,+\infty]
		\]
		is left-continuous for every $y\in Y\setminus N$ and $x\in G$. Consequently,
		\[
		(T^{\omega,\rho}_{b,*,d} f(y,\,\cdot\,))(x)=\sup_{t\in \Q} \ee^{-\omega t}(\abs{f}*(\ee^{\rho\abs{\,\cdot\,}_*} p_{b,t,d}))(x)
		\]
		for every $y\in Y\setminus N$ and $x\in G$. It then suffices to observe that the mapping $(y,x)\mapsto (\abs{f}*(\ee^{\rho\abs{\,\cdot\,}_*} p_{b,t,d}))(x) $ is $(\nu\otimes \mi)$-measurable by Tonelli's theorem for every $t\in (0,+\infty)$.

		Observe that, by Lemma~\ref{lem:3}, given $\rho'>0$ there are  $C_1,\omega>0$ such that
		\[
		\ee^{\rho\abs{x}_*-\omega t}p_{b,t,d}(x) \meg C_1\ee^{-\rho'\abs{x}_*}
		\]
		for every $x\in G$ and  $t\Meg 1$. Let $\rho'$ be so large that $\ee^{-(\rho'-c)\abs{\,\cdot\,}_*}\in L^1(\beta)$, where $c>1$ is such that $\Delta_R(x)\meg c\ee^{c\abs{x}_*}$ for all $x\in G$ (cf.~Lemma~\ref{lem:31}), then  by Young's inequality
				\[
		\begin{split}
		\Big\| \sup_{t\Meg 1}\ee^{-\omega t}(\abs{f(y,\,\cdot\,)}*(\ee^{\rho\abs{\,\cdot\,}_*} p_{b,t,d}) )(x)\Big\| _{L^{q,p}_{(y,x)}(\mi,\beta)}&\meg C_1\norm{f}_{L^{q,p}(\mi,\beta)} \norm{\ee^{-\rho' \abs{\,\cdot\,}_*} \Delta_R^{-1/p'}}_{L^1(\beta)}\\	
			&\meg C_1 c\norm{\ee^{-(\rho'-c) \abs{\,\cdot\,}_*}  }_{L^1(\beta)} \norm*{f}_{L^{q,p}(\mi,\beta)},
		\end{split}
		\]  
		for  $f\in L^{q,p}(\mi,\beta)$. Thus we may reduce to prove the assertion for $T^*_{b,\kappa,d}$ for some $\kappa>0$.
		
We split $T^*_{b,\kappa,d}$ into the sum of $T_0$ and $T_\infty$, where 
		\[
		T_0 (f)\coloneqq \sup_{t\in (0,\kappa]} (\abs{f}*(\chi_{B(e,1)} p_{b,t,d})), \qquad T_\infty(f)\coloneqq \sup_{t\in (0,\kappa]} (\abs{f}*(\chi_{G\setminus B(e,1)} p_{b,t,d}))
		\]
		for every $\beta$-measurable function $f$ on $G$. Observe first that  
		\[
		\begin{split}
			[T_\infty(f)](x)& \meg \int_{G\setminus B(x,1)} \abs{f(y)} \sup_{t\in (0,\kappa]} t^{-Q_*/d} \ee^{-b (\abs{y^{-1}x}^{d}/t)^{1/(d-1)}}\Delta_L(y^{-1})\,\dd \beta(y)\\
			&\meg C_1\int_{G } \abs{f(y)}  \ee^{-(b/2) (\abs{y^{-1}x}^{d}/\kappa)^{1/(d-1)}}\Delta_L(y^{-1})\,\dd \beta(y)\\
			&=C_1\kappa^{Q_*/d}(\abs{f}*p_{b/2,\kappa,d})(x)
		\end{split}
		\]
		where $C_1\coloneqq \sup_{t\in (0,\kappa]} t^{-Q_*/d} \ee^{-(b/2) t^{-1/(d-1)}}$. We may then conclude that $T_\infty$ satisfies the desired inequality as before.
		
		Let us then consider $T_0$, and prove that there is $C_2>0$ such that $
		T_0\meg C_2 \Mr_1$, where
		\[
		(\Mr_1 f)(x)\coloneqq \sup_{r\in (0,1]} \dashint_{B(x,r)} \abs{f(y)}\Delta_L(y^{-1})\,\dd \beta(y)
		\]
		for $f\in L^1_\loc(\beta)$ and $x\in G$.
		To this aim, observe that there is a constant $C_3>0$ such that
		\[
		\begin{split}
			p_{b,t,d}(x)&=t^{-Q_*/d} \ee^{-b(\abs{x}/t^{1/d})^{d/(d-1)}} \meg C_3 t^{-Q_*/d} (1+\abs{x}/t^{1/d} )^{-Q^*-1}=\frac{t^{1/d}}{(t^{1/d}+\abs{x}_*)^{Q_*+1}}
		\end{split}
		\]
		for every $x\in G$ and $t>0$.
		In addition, take $C_4>0$ such that 
		\[
		\int_{B(x,r)} \Delta_L(y^{-1})\,\dd \beta(y)=\int_{B(e,r)} \Delta_L(y^{-1})\,\dd \beta(y)\meg C_4 r^{Q_*}
		\]
		for every $r\in (0,1]$ and $x\in G$, by left invariance.
		Then, observe that
		\[
		\begin{split}
			(\abs{f}*(\chi_{B(e,1)} p_{b,t,d}))(x)&= C_3\int_{B(x,1)}  \abs{f(y)}  \frac{t^{1/d}}{(t^{1/d}+\abs{y^{-1}x}_*)^{Q_*+1}}\Delta_L(y^{-1})\,\dd \beta(y)\\
			&\meg C_3\sum_{j\in \N} \int_{B(x,2^{-j})\setminus B(x,2^{-j-1}) } \abs{f(y)} \frac{t^{1/d}}{(t^{1/d}+2^{-j-1})^{Q_*+1}}\Delta_L(y^{-1})\,\dd \beta(y)\\
			&\meg C_3 C_4 \sum_{j\in \N} \frac{(2^j t^{1/d}) }{(2^j t^{1/d}+1/2)^{Q_*+1}} (\Mc_1 f)(x)
		\end{split}
		\]
		for every $x\in G$, so that we may choose 
		\[
		C_2=C_3 C_4 \sup_{t\in (0,\kappa]}\sum_{j\in \N} \frac{(2^j t^{1/d}) }{(2^j t^{1/d}+1/2)^{Q_*+1}}.
		\]	
To conclude, we need the boundedness of $\Mc_1$ on $L^{q,p}(\mi,\beta)$, which follows from~\cite{Singular}.
	\end{proof}

	\begin{lem}\label{lem:2} 
		There is $\omega\in \R$ such that the following holds. Suppose that $p\in (1,\infty)$, $q\in [1,\infty]$, $b>0$, $d>1$, that $(Y,\Mf, \mi)$ is a  $\sigma$-finite measure space and $\kappa\colon Y\to (0,+\infty)$ is a $\mi$-measurable function.  Then there is  $C>0$ such that
		\[
		\big\| \|   \ee^{-\omega\kappa(y)} [T_{b,\kappa(y),d} f(y,\,\cdot\,)](x) \|_{L^{q}_{y}(\mu)} \big\| _{L^{p}_{x}(\beta)}\meg  C \big\| \| f(x,y)\|_{L^{q}_{y}(\mu)} \big\|_{L^{p}_{x}(\beta)}
		\]
	 for every $(\mi\otimes \beta)$-measurable function $f$ on $Y\times G$.
	 \end{lem}
	
	This result is the analogue of~\cite[Proposition 3.5]{BPV}. One may observe that the proof below also works in the case $p=q=1$; nonetheless, this case is in fact trivial and follows easily from Lemma~\ref{lem:3}.

	\begin{proof}
Denote with $T$ the operator defined so that 
		\[
		T f(y,x) =[T_{b,\kappa(y),d} f(y,\,\cdot\,)](x)
		\]
		for every $(\mi\otimes \beta)$-measurable function $f$ and $(y,x)\in Y\times G$.
		Let us first show that $Tf$	is $(\mi\otimes \beta)$-measurable. Observe that we may reduce to the case in which $f=\chi_N$ for some $(\mi\otimes \beta)$-negligible subset $N$ of $Y\times G$, or $f=f_0\otimes f_1$ for some $\mi$-measurable function $f_0$ and for some $\beta$-measurable function $f_1$. If $f=\chi_N$, then $T f(y,\,\cdot\,)=0$ on $G$ for $\mi$-almost every $y\in Y$, so that $Tf $ is $(\mi\otimes \beta)$-measurable.
		Then, assume that $f=f_0\otimes f_1$.
		We may also assume that both $f_0$ and $f_1$ are bounded and compactly supported. Then,
		\[
		[T_{b,\kappa(y),d} f(y,\,\cdot\,)](x)= \abs{f_0(y)} (\abs{f_1}*p_{b,\kappa(y,d})(x).
		\]
The map $(y,x)\mapsto \abs{f_0(y)}$ is clearly $(\mi\otimes \beta)$-measurable, while $(y,x)\mapsto (\abs{f_1}*p_{b,\kappa(y),d})(x)$ is $(\mi\otimes \beta)$-measurable since it is the composite of the $(\mi\otimes \beta)$-measurable mapping $(y,x)\mapsto (\kappa(y),x)$ and the continuous function $(s,x)\mapsto (\abs{f_1}*p_{b,s,d})(x)$  (cf.~Lemma~\ref{lem:3} and Young's inequality).
		
		The case $p,q\in (1,\infty)$ follows from Lemma~\ref{lem:5}. The case $q=\infty$ follows from the scalar-valued case in Lemma~\ref{lem:5} (which may be obtained choosing $\mi$ as a Dirac delta). It remains to consider the case $q=1$.
		Take a positive locally integrable function $w$ on $G$.  
		Then
		\[
		\begin{split}
			\int_G  \ee^{-\omega \kappa(y)}(T_{b,\kappa(y),d} g)(x) w(x)\,\dd \beta(x) &= \ee^{-\omega \kappa(y)}\int_G ( \abs{g}*p_{ b,\kappa(y),d} ) w\,\dd \beta\\
			&= \ee^{-\omega \kappa(y)}\int_G  \abs{g} [w*(\Delta_L^{-1} \widecheck p_{b,\kappa(y),d} ) ]\,\dd \beta\\ 
			&\meg c\int_G  \abs{g} (T^{*,\omega, c}_{b , d} w)\,\dd \beta
		\end{split}
		\] 
		by the symmetry of $p_{b,\kappa(y),d}$.
		Consequently,
		\[
		\begin{split}
			&\norm{\ee^{-\omega \kappa(y)}(T f)(y,x)  }_{L^{1,p}_{y,x}(\mi,\beta)} 	\\
			& \quad =\sup \bigg\{  \left( \int_Y\int_G  \ee^{-\omega \kappa(y)}[T_{b,\kappa(y),d} f(y,\,\cdot\,)](x) w(x)\,\dd \beta(x)\,\dd \mi(y)\right) \colon  w\Meg0, \;\, \norm{w}_{L^{p'}(\beta)}\meg 1 \bigg\} \\
			&\quad\meg c    \sup \bigg\{  \left( \int_G \int_Y  \abs{f(y,x)}\,\dd \mi(y) (T^{*,\omega, c}_{b , d} w)(x)\,\dd \beta(x)\right)\colon  w\Meg0, \;\, \norm{w}_{L^{p'}(\beta)}\meg 1 \bigg\}\\
			&\quad \meg c    \norm{f}_{L^{1,p}(\mi,\beta)} \sup \big\{\|T^{*,\omega, c}_{b , d} w\|_{L^{p'}(\beta)}  \colon  w\Meg0, \;\, \norm{w}_{L^{p'}(\beta)}\meg 1\big\}
		\end{split}
		\]
		whence the conclusion   since $T^{*,\omega, c}_{b , d}$ is of strong type $(p',p')$.
\end{proof}

The next result is inspired by~\cite[Lemma 3.1]{Feneuil}. Before stating it, let us give the following definition.
\begin{deff}\label{def:2}
For $k\in \N$, $t>0$, and $f\in \Sc'(G)$ we write
	\[
	W_t^{(k)} f\coloneqq (t \Lc)^k \ee^{-t \Lc} f,
	\]
and for $\alpha \geq 0$ and $\eps\in (0,1]$
	\[
	W_{t,\eps}^{(\alpha),*} f\coloneqq t^\alpha \max_{s\in [\eps t,t/\eps ]} \max_{\substack{\deg(X)+\deg(Y)\meg \alpha \dL \\ \abs{X} ,\abs{Y}\meg 1}} \abs{X \ee^{-s \Lc} Y f }.
	\]
	\end{deff}
We shall occasionally denote by $\Lc'$ another weighted subcoercive operator. In this case, $W_t'^{(k)}$ and $W'^{(\alpha),*}_{t,\eps}$ are defined accordingly.

\begin{lem}\label{lem:9}
Suppose $F\in \{\Sc(G), \Sc'(G)\}$ and $f\in F$. Then $f*h_t$ is well defined for all $t>0$. Moreover, the map 
\[
 t\mapsto f*h_t, \qquad [0,+\infty) \to F
 \]
 is smooth, and $\frac{\dd^k}{\dd t^k} (f*h_t)=\Lc^k (f*h_t)$   for every $k\in\N$. In particular, 
	\[
	f=\sum_{k=0}^m \frac{1}{k!} W_t^{(k)}f + \frac{1}{m!} \int_0^t W_s^{(m+1)} f  \,\frac{\dd s}{s} \qquad \forall\, t>0, \, m\in\N.
	\]
\end{lem}

\begin{proof}
	The third assertion follows from the second one and~\cite[Chapter II, \S 2, No.\ 6]{BourbakiFRV}. In order to prove the first two assertions, up to replace $\Lc$ with $\Lc^*$ and arguing by duality, we may reduce to the case $F=\Sc(G)$.  Define 
	\[
	\phi\colon [0,+\infty)\ni t \mapsto f*h_t\in \Sc(G).
	\] Observe first that, combining Theorem~\ref{teo:7} and Lemma~\ref{lem:3}, we see that for every $X\in U(\mathfrak{g})$ and $\rho>0$ there is a constant $C>0$ such that
	\[
	\norm{\ee^{\rho\abs{\,\cdot\,}_*}X^R\phi(t)}_{L^1(\beta)} \meg  \norm{(\ee^{\rho\abs{\,\cdot\,}_*}\abs{X^R f})*(\ee^{\rho\abs{\,\cdot\,}_*}\abs{h_t})}_{L^1(\beta)}\meg C \ee^{ C t}
	\]
	for every $t\Meg 0$.
	Consequently,  $\phi$ is locally bounded  by Theorem~\ref{teo:8}~(5), so that also  $\Lc^{k}\phi$ is locally bounded on $[0,+\infty)$.  Since these mappings are continuous with values in $C^\infty(G)$, by Theorem~\ref{teo:8}~(3) we then see that they are continuous with values in $\Sc(G)$. The fact that $\Lc^{k+1}\phi $ is the derivative of $\Lc^{k} \phi$ for every $k\in\N$ is then a consequence of the fundamental theorem of Calculus and the previous remarks, whence the conclusion.
\end{proof}

The following result is instead inspired by~\cite[Lemma 3.2]{BPV}.

\begin{lem}\label{lem:1}
	There is $\omega\in \R$ such that the following hold.  
	\begin{enumerate}
		\item[\textnormal{(1)}] If $b>0$ and $d>1$, then
		\[
		T_{b,t,d}f \meg \kappa^{-Q_*/d}  T_{b,t_0,d } f
		\]
		for every $f\in \Sc'(G)$, $t_0>0$, $\kappa\in (0,1)$, and $t\in [\kappa t_0, t_0]$.
		
		\item[\textnormal{(2)}]  For every $\lambda \Meg 0$ there are $C,b>0$ such that
		\[
		\abs{X \ee^{-t \Lc} Y  f}   \meg C \abs{X} \abs{Y}\ee^{\omega t} t^{-(\deg X+\deg Y)/\dL} T_{b,t}f
		\]
		for every $f\in \Sc'(G)$, $X,Y\in U_\lambda$ and $t>0$.
	\end{enumerate}
\end{lem}

\begin{proof}
	The first assertion is trivial. The second one follows from Theorem~\ref{teo:7}, since $X\ee^{-t \Lc}Y f= f*( X Y^{\dag R\dag} h_t)$ for every $t>0$ (cf.~Proposition~\ref{prop:9}), $X,Y\in U(\mathfrak{g})$, and $f\in \Sc'(G)$.
\end{proof}

We shall use assertion (1) of Lemma~\ref{lem:1} without reference in the sequel. Let us also state the following corollary, which is a simple consequence of Lemmas~\ref{lem:3} and~\ref{lem:1}, and Young's inequality; cf~\cite[Lemmas 3.3 and 3.4]{BPV}.
\begin{cor}\label{cor:1}
 There is $\omega\in \R$ such that for every $\lambda \Meg 0$ there is $C>0$ such that
	\[
	\norm{X \ee^{-t \Lc} Y (\Delta_L^{1/p_1-1/p_2} f)}_{L^{p_2}(\beta)} \meg C \abs{X} \abs{Y}  t^{-(\deg X+ \deg Y+Q_*/p_2')/\dL} \ee^{\omega t} \norm{f}_{L^{p_1}(\beta)}
	\]
	for every $f\in \Sc'(G)$,  $p_1,p_2\in [1,\infty]$, with $p_1\meg p_2$,   $X,Y\in U_\lambda$, and $t>0$.
\end{cor}

\section{Carleson measures and auxiliary results}\label{Carmeas:sec}
This section is quite technical and devoted to proving some results which will be needed in the remainder of the paper. However, we introduce a class of measures on the positive
half-line that will allow us to treat some of arguments in Sections
\ref{BTL:sec}, \ref{emb:sec} and \ref{dual:sec} involving
Besov and Triebel--Lizorkin norms in a unified way.

\subsection{Carleson measures} 
We define some classes of measures which we shall systematically use in the following.
\begin{deff}
	We denote by $\Mc_\Car$ (``Carleson measures'')\footnote{
          We adopt this definition, even if it is an abuse of
          language, since there exists an analogy with the classical
          definition of Carleson measures.}  the space of positive Radon measures $\mi$ on $(0,+\infty)$ with bounded support such that there is a constant $C>0$ such that
	\[
	\mi((r,2r])\meg C
	\]
	for every $r>0$. We denote with $\Mc_\RC$  (``reverse Carleson measures'') the space of positive Radon measures $\mi$ on $(0,+\infty)$ such that there are constants $\eps,C\in (0,1)$   such that
	\[
	\mi((\eps r, r])\Meg C
	\]
	for every $r\in (0,\eps]$. We define $\Mc_\Samp\coloneqq \Mc_\Car\cap \Mc_\RC$  (``sampling measures'').
\end{deff}	
The most relevant examples of sampling measures that the reader should have in mind are
\[
\mi= \sum_{j\in\N} \delta_{2^{-j}}, \qquad \dd\mi_\kappa(t)=\mathbf{1}_{(0,\kappa]}(t)\frac{\dd t}{t} \quad (\kappa>0).
\]
These are the measures that allow to express the forthcoming Triebel--Lizorkin and Besov norms in their classical continuous and discrete (of ``Littlewood--Paley type'') version; recall~\cite[Theorem 4.2]{BPV}. The use of arbitrary sampling measures though will allow us to prove both characterizations at once. The notation $\mi_\kappa$ for any $\kappa>0$ will be fixed from this point on and denote the above measure.

\begin{lem}\label{lem:11b}
Suppose $\mi\in \Mc_\Car$. Then, for every $\kappa>1$ there is a constant $C>0$ such that
	\[
	\mi((r, \kappa r])\meg C \qquad \forall \, r>0.
	\]
\end{lem}

\begin{proof}
	By assumption, there is a constant $C_1>0$ such that
	\[
	\mi((r,2r])\meg C_1
	\]
	for every $r>0$. It then suffices to take $C\coloneqq ([\log_2 \kappa]+1) C_1$, since
	\[
	\mi((r,\kappa r])\meg \mi((r,2^{C/C_1} r])=\sum_{k=0}^{C/C_1-1} \mi((2^k r, 2^{k+1}r])\meg C
	\]
	for every $r>0$.
\end{proof}

\begin{lem}\label{lem:20}
Suppose $a,b>0$ and let $\phi$ be the function on $(0,+\infty)$ given by $\phi(r)=a r^b$. Then the map $\mi\mapsto \phi_*\mi$ induces automorphisms of $ \Mc_\Car$, $\Mc_\RC$, and $\Mc_\Samp$.
\end{lem}

\begin{proof}
	Observe that the inverse of $\phi$ is $\psi(r)= a^{-1/b} r^{1/b}$, which is of the same form. Thus it will suffice to show that $\phi_*$ induces endomorphisms of $\Mc_\Car$ and $\Mc_\RC$. Take first $\mi\in \Mc_\Car$. Then there is $C>0$ such that $\supp \mi\subseteq (0,C]$ and $\mi( r,2^{1/b} r])\meg C$ for every $r>0$, thanks to Lemma~\ref{lem:11b}. Then,  $\supp\phi_*\mi\subseteq (0,\psi(C)]$ and 
	\[
	(\phi_*\mi)((r,2 r])= \mi(\psi((r,2 r]))=\mi(( \psi(r),2^{1/b} \psi(r)]  )\meg C
	\]
	for every $r>0$. Then, $\mi\in \Mc_\Car$.
	
	Next, take $\nu\in \Mc_\RC$, and take $c\in (0,1)$ so that $\nu((c r,r])\Meg c$ for every $r\in (0,c]$. Then 
	\[
	(\phi_*\nu) ((  c^b  r, r])= \nu(( c\psi(r),  \psi(r)] )\Meg c
	\]
	for every $r\in (0,\phi(c)]$, so that $\nu\in \Mc_\RC$.
\end{proof}

\begin{lem}\label{lem:25}
Suppose $\eta,\gamma>0$ and $\mi,\nu\in \Mc_\Car$. Then, there is a constant $C>0$ such that
	\[
	\norm*{\int_0^\infty \frac{s^\gamma t^{\eta}}{(s+t)^{\eta+\gamma}} \abs{f(t)}\,\dd \nu(t)  }_{L^q_s(\mi)} \meg C \norm{f}_{L^q(\nu)}
	\]
	for every  $q\in [1,\infty]$ and  $\nu$-measurable function $f$.
\end{lem}

\begin{proof}
	The assertion will follow from Schur's lemma (cf., e.g.,~\cite[Lemma I.1]{GrafakosClassical}) if we show that
	\[
	\sup_{s>0} \int_0^\infty \frac{s^\gamma t^{\eta }}{(s+t)^{\eta+\gamma}}\,\dd \nu(t) \qquad \text{and} \qquad
	\sup_{t>0} \int_0^\infty \frac{s^\gamma t^{\eta }}{(s+t)^{\eta+\gamma}}\,\dd \mi(s)
	\]
	are both finite. By symmetry, it will suffice to prove the first assertion. Then, take $C_1>0$   so that $ \nu((t,2t])\meg C_1$ for every $t>0$, and observe that
	\[
	\begin{split} 
		\int_0^\infty \frac{s^\gamma t^{\eta }}{(s+t)^{\eta+\gamma}}\,\dd \nu(t)&= \int_0^\infty \frac{t^\eta}{ (1+ t)^{\eta+\gamma}} \,\dd (s^{-1}_*\nu)(t)\\
		&=  \sum_{h\in\Z}  \int_{2^{h}}^{2^{h+1}} \frac{t^\eta}{  (1+t)^{\eta+\gamma}} \,\dd (s^{-1}_*\nu)(t) \meg C_1 2^\eta\sum_{h\in\Z} \frac{2^{h\eta}}{(1+2^{h})^{\eta+\gamma} },
	\end{split}
	\]
	which is finite.
\end{proof}

\begin{cor}\label{cor:carleson-measures}
  Take $\eta>0$, $q\in[1,+\infty]$, and $\mi\in\Mc_\Car$.  Then
  \[
    \norm{t^\eta}_{L^q_t(\mi)}<\infty.
    \]
  \end{cor}
  \begin{proof}
The case $q=\infty$ is obvious, whereas the case $q$ finite follows
from the proof of Lemma~\ref{lem:25}, applied with $q\eta$ in place of
$\gamma$. 
    \end{proof}
  
\begin{lem}\label{lem:25bis}
Suppose $\eta,\gamma>0$ and  let $\mi$ be a Haar measure on $(0,+\infty)$. Take $\nu\in \Mc_\Car$, and assume that either $\eta\Meg 1$ or $\nu\in L^\infty(\mi)\cdot \mi$.	
	Then, there is a constant $C>0$ such that
	\[
	\norm*{\int_0^\infty t^\eta s^\gamma \abs{f(s+t)}\,\dd \mi(t)  }_{L^q_s(\nu)} \meg C \norm{ t^{\gamma+\eta}f(t)}_{L^q_t(\mi)}
	\]
	for every  $q\in [1,\infty]$ and $\mi$-measurable function $f$.
\end{lem}

\begin{proof}
	Observe that, setting $g\colon t \mapsto t^{\gamma+\eta}f(t)$, then
	\[
	\int_0^\infty t^\eta s^\gamma \abs{f(s+t)}\,\dd \mi(t)  =\int_s^\infty (t-s)^{\eta-1} t^{-\gamma-\eta+1} s^{\gamma} \abs{g(t)}\,\dd \mi(t).
	\]
	The assertion will follow from Schur's lemma (cf., e.g.,~\cite[Lemma I.1]{GrafakosClassical}) if we show that
	\[
	\sup_{s>0} \int_s^\infty (t-s)^{\eta-1} t^{-\gamma-\eta+1} s^{\gamma} \,\dd \mi(t) \qquad \text{and} \qquad
	\sup_{t>0} \int_0^t  (t-s)^{\eta-1} t^{-\gamma-\eta+1} s^{\gamma} \,\dd \nu(s)
	\]
	are both finite. For the first assertion, observe that
	\[
	\int_s^\infty (t-s)^{\eta-1} t^{-\gamma-\eta+1} s^{\gamma} \,\dd \mi(t)= \int_1^\infty (t-1)^{\eta-1} t^{-\gamma-\eta+1}\,\dd \mi(t),
	\]
	which is finite.
	Analogously, if $\nu=h\cdot\mi$ for some $h\in L^\infty(\mi)$, then
	\[
	\begin{split} 
		\int_0^t  (t-s)^{\eta-1} t^{-\gamma-\eta+1} s^{\gamma} \,\dd \nu(s)&\meg \norm{h}_{L^\infty(\mi)}\int_0^1 (1-s)^{\eta-1} s^{\gamma}\,\dd \mi(s) ,
	\end{split}
	\]
	which is finite. Then, assume that $\eta\Meg 1$, and observe that there is $C_1>0$ such that $\nu((t,2t])\meg C_1$ for every $t>0$. Then,
	\[
	\begin{split} 
		\int_0^t  (t-s)^{\eta-1} t^{-\gamma-\eta+1} s^{\gamma} \,\dd \nu(s)&= \int_0^1 (1-s)^{\eta-1} s^{\gamma}\,\dd (t^{-1}_*\nu)(s)\\
		&\meg \sum_{h\in \N}   \int_{2^{-h-1}}^{2^{-h}} s^\gamma \,\dd  (t^{-1}_*\nu)(s)\meg C_1 \sum_{h\in \N} 2^{-h\gamma},
	\end{split}
	\]
	which is finite.
\end{proof}

 \subsection{Auxiliary results} We now prove some technical results
 which will be of use in the next sections. Some of these are stated
 with the family $(\ee^{-t \Lc} f)_{t>0}$ for some $f\in \Sc'(G)$ in
 mind. Nevertheless, some proofs (in particular that of
 Lemma~\ref{lem:27} below) turn out to be leaner if a more general
 family $(f_t)_{t>0}$ in $\Sc'(G)$ satisfying $f_{t+s}=\ee^{-s \Lc}
 f_t$ for every $t,s>0$  is considered. Because of Lemma~\ref{lem:27}, however, this more general formulation is \emph{a posteriori} redundant.
 
Let us also note that for statements involving mixed norms in both orders that share analogous proofs, we detail only the more complex case.
 
\begin{lem}\label{lem:26}
Suppose $\alpha\in \R$, $p,q\in [1,\infty]$,   $\mi\in \Mc_\Car$, and $\nu\in \Mc_\RC$. Then, there is a constant $C>0$ such that, for every family $(f_t)_{t>0}$ in $\Sc'(G)$ such that $f_{t+s}=\ee^{-s \Lc} f_t$ for every $t,s>0$, the following hold:
	\begin{enumerate}
		\item[\textnormal{(i)}] $\norm*{ t^\alpha \norm{f_t}_{L^p(\beta)} }_{L^q_t(\mi)}\meg C \norm*{ t^\alpha \norm{f_t}_{L^p(\beta)} }_{L^q_t(\nu)}$;
		
		\item[\textnormal{(ii)}] if, in addition, $p\in (1,\infty)$, then \[
		\norm*{\norm{ t^\alpha f_t(x) }_{L^q_t(\mi)}}_{L^p(\beta)}\meg C\norm*{\norm{ t^\alpha f_t(x) }_{L^q_t(\nu)}}_{L^p(\beta)}.
		\]
	\end{enumerate}
\end{lem}
\begin{proof}
We prove only (ii). Take $\eps\in (0,1)$ and $N\in\N$ so that $N\Meg 1$, $\supp \mi\subseteq (0,\eps^{-N}]$,
	\[
	\nu((\eps t, t])\Meg \eps^N
	\]
	for every $t\in (0,\eps^N]$, and
\begin{equation}\label{muesp1}
	\mi((\eps t,t])\meg \eps^{-N}
\end{equation}
	for every $t>0$ (cf.~Lemma~\ref{lem:11b}). By the Gaussian estimates (cf.~Theorem~\ref{teo:7}), there are $C_1,b>0$ such that $\abs{h_t}\meg C_1 p_{b,t}$ for every $t\in (0,\eps^{-N}]$.	
	Observe that, for every $t\in [\eps^{-N+h+1},\eps^{-N+h}]$, $h\in\N$, and $s\in [\eps^{N+h+1},\eps^{N+h}]$,
	\[
	\abs{f_t}= \abs{\ee^{-(t-s) \Lc} f_s}\meg C_1  T_{b,t-s}  f_s\meg C_1 (\eps^{N}(\eps^{1-N}-\eps^N) )^{-Q_*/\dL}  T_{b,\eps^{-N+h}} f_s
	\]
	since 
	\[
	\eps^{-N+h+1}-\eps^{N+h}\meg t-s\meg \eps^{-N+h}.
	\]
	In particular, setting $C_2\coloneqq  C_1 (\eps^{N}(\eps^{1-N}-\eps^N) )^{-Q_*/\dL}$,
	\[
	\abs{f_t} \meg C_2\min_{s\in [\eps^{N+h+1},\eps^{N+h}]} T_{b,\eps^{-N+h}} f_s
	\]
	for every $t\in [\eps^{-N+h+1},\eps^{-N+h}]$.

	Then, by~\eqref{muesp1} and Lemma~\ref{lem:2}, there is a constant $C_2>0$ such that
	\[
	\begin{split}
	&\norm*{\norm{ t^\alpha f_t(x) }_{L^q_t(\mi)}}_{L^p(\beta)}= \norm*{\norm*{\norm{ t^\alpha f_t(x) }_{L^q_t((\eps^{-N+h+1},\eps^{-N+h}],\mi)}}_{\ell_h^q(\N)}}_{L^p_x(\beta)}\\
		&\qquad \qquad \qquad\meg C_2 \eps^{-N\alpha-\alpha_- -N/q}\big\| \norm{ \eps^{h\alpha} \min_{s\in [\eps^{N+h+1},\eps^{N+h}]} T_{b,\eps^{-N+h}}  f_{s}(x)  }_{\ell^q_h(\N)}\big\|_{L^p_x(\beta)}\\
		&\qquad\qquad \qquad \meg C_2 \eps^{-N\alpha-\alpha_- -2N/q}\norm*{\big\| \eps^{h\alpha}\norm{  T_{b,\eps^{-N+h}}  f_{s}(x) }_{L^q_s((\eps^{N+h+1},\eps^{N+h}],\nu)} \big\|_{\ell^q_h(\N)}}_{L^p_x(\beta)}\\
		&\qquad \qquad\qquad\meg C_2 \eps^{-2N\alpha-\abs{\alpha} -2N/q-Q_*/\dL}\norm*{\norm{ s^{\alpha} T_{b,s\eps^{-2N-1}}  f_{s}(x) }_{L^q_s(\nu)}  }_{L^p_x(\beta)}\\
		&\qquad \qquad\qquad \meg C_3\norm*{\norm{ s^{\alpha}    f_{s}(x)  }_{L^q_s(\nu)}}_{L^p_x(\beta)}\\
	\end{split}
	\]
	whence the conclusion.
\end{proof}

\begin{lem}\label{lem:27}
Suppose $\alpha\in \mathbb{R}$,  $m\in\N$, $p,q\in [1,\infty]$, $\mu\in \Mc_\RC$,  and $(f_t)_{t>0}$ in $\Sc'(G)$ such that $f_{s+t}=\ee^{-s \Lc} f_t$ for every $t,s>0$, and such that either
\begin{enumerate}
\item $\norm{t^{\alpha} \Lc^m  f_{t}(x) }_{L^{p,q}_{x,t}(\beta,\mi)} $
  is finite,
\end{enumerate}
or
\begin{enumerate}\setcounter{enumi}{1}
\item $p\in (1,\infty)$ and $\norm{t^{\alpha} \Lc^m f_{t}(x) }_{L^{q,p}_{t,x}(\mi, \beta)}$ is finite.
\end{enumerate}
Then, there is a unique $f\in \Sc'(G)$ such that $f_t=\ee^{-t\Lc} f$ for every $t>0$.
\end{lem}

\begin{proof}
	By Lemma~\ref{lem:26}, we may reduce to the case in which $\mi=\sum_{j\in\N} \delta_{2^{-j}}$. Consequently, if (2) holds, by Minkowski's integral inequality and the (contractive) inclusion $\ell^{q/p}(\N)\subseteq\ell^{\max(1,q/p)}(\N)$,
\begin{align*}
\norm{t^{\alpha} \Lc^m  f_{t}(x) }_{L^{p,q}_{x,t}(\beta,\mi)} 
&=\big\| 2^{-j\alpha} \norm{\Lc^m f_{2^{-j}} }_{L^{p}(\beta)}  \big\|_{\ell_j^{\max(p,q)}(\N)}\\
&=\bigg\| 2^{-pj\alpha} \int_G\abs{\Lc^m f_{2^{-j}}}^p\,\dd \beta  \bigg\|_{\ell_j^{\max(1,q/p)}(\N)}^{1/p}\\
		&\meg \left( \int_G \norm{2^{-pj\alpha} \abs{\Lc^m f_{2^{-j}}}^p }_{\ell_j^{\max(1,q/p)}(\N)}\,\dd \beta\right) ^{1/p}		\\
		&\meg  \left( \int_G \norm{2^{-pj\alpha} \abs{\Lc^m f_{2^{-j}}}^p }_{\ell_j^{q/p}(\N)}\,\dd \beta\right) ^{1/p}	\\
		&= \big\| \norm{2^{-j\alpha} (\Lc^m f_{2^{-j}})(x)}_{\ell_j^q(\N)} \big\|_{L^p_x(\beta)} = \norm{t^{\alpha} \Lc^m f_{t}(x) }_{L^{q,p}_{t,x}(\mi, \beta)},
\end{align*}
namely $L^{q,p}(\mi,\beta)\subseteq L^{p,\max(p,q)}(\beta, \mi)$, whence assumption (1) holds with $q$ replaced by $\max (p,q)$. Then, we may reduce to case (1) and assume that $\norm{t^{\alpha} \Lc^m f_{t}(x) }_{L^{p,q}_{x,t}(\beta,\mi)} $ is finite. Take $m'\in \N$ so that $m+m'\Meg 1$ and $m+m'-\alpha>0$.	By Lemma~\ref{lem:9}
	\[
	\begin{split}
	f_t&=\sum_{h=0}^{m+m'-1} \frac{1}{h!} W_{1-t}^{(h)} f_t +\frac{1}{(m+m'-1)!}\int_0^{1-t} W^{(m+m')}_s f_t \,\frac{\dd s}{s}\\
		&=\sum_{h=0}^{m+m'-1} \frac{(1-t)^h}{h!} \Lc^h f_1 +\frac{1}{(m+m'-1)!}\int_0^{1-t} (s\Lc)^{m+m'} f_{t+s} \,\frac{\dd s}{s} \\
		&=\sum_{h=0}^{m+m'-1} \frac{(1-t)^h}{h!} \Lc^h f_1 +\frac{1}{(m+m'-1)!}\int_{t}^{1} (s-t)^{m+m'-1}\Lc^{m+m'} f_{s} \,\dd s
	\end{split}
	\]
	for every $t\in (0,1)$. Then, take $\phi \in \Sc(G)$, and observe that 
	\[
	\begin{split}
\mathbf{1}_{[t,1]}(s) (s-t)^{m+m'-1} \abs{\langle \Lc^{m+m'} f_{s} ,  \phi\rangle}& 
\meg \mathbf{1}_{[0,1]}(s) s^{m+m'-1}  \norm{\Lc^{\dag m'} \phi}_{L^{p'}(\beta)} \norm{\Lc^m f_{s}}_{L^p(\beta)}\\
	\end{split}
	\]
	where the right hand side is integrable, since 
		\[
	\begin{split}
\int_0^{1} s^{m+m'-1}  \norm{\Lc^m f_{s}}_{L^p(\beta)}\, \dd s
		&\meg \norm{s^{m+m'-\alpha} }_{L^{q'}(\mi_1)}\norm{s^{\alpha} \Lc^m f_{s}(x) }_{L^{p,q}_{x,s}(\beta,\mi_1)}.
	\end{split}
	\]
Hence, by dominated convergence $f_t$ converges to
	\[
	f:=\sum_{h=0}^{m+m'-1} \frac{1}{h!} \Lc^h f_1 +\frac{1}{(m+m'-1)!}\int_0^{1} (s\Lc)^{m+m'} f_{s} \,\frac{\dd s}{s} 
	\] 
	in $\Sc'(G)$ for $t\to 0^+$.  By the continuity of $\ee^{-t\Lc}$ on $\Sc'(G)$, then, $f_{t+s} = \ee^{-s\Lc} f_{t} \to \ee^{-s\Lc} f$ for $t\to 0^{+}$. But by the continuity of the map $t\mapsto f_{t}$, recall Lemma~\ref{lem:9}, we also have that $f_{t+s} \to f_{s}$ for $t\to 0^{+}$, whence $f_{s} = \ee^{-s\Lc}f$ for all $s>0$ as claimed.
\end{proof}
The next result is inspired by~\cite[Theorems 4.1, 4.2, and 4.4]{BPV}, and is somewhat the key result for all what follows. It will allow us to show that the Besov and Triebel–Lizorkin norms are independent of several parameters (among which the weighted sobcoercive operator in terms of which the spaces are defined).
Since its proof is long and technical, for the sake of
presentation, we postpone it to the later
Section~\ref{A1:sec}.
\begin{prop}\label{lem:8bis}
	Suppose $\alpha\in \R$, $p,q\in [1,\infty]$, $\lambda \Meg 0$, $\eps \in (0,1]$,   $t_0,t_1\Meg 0$, $\mi\in \Mc_\Car$, and $\nu\in \Mc_\RC$. Assume that $t_0>0$ if $\alpha\meg 0$. Let $\Lc'$ be another weighted subcoercive operator of degree $\dL'$, suppose  $m  \in\N$  and\footnote{Notice that $m'$ is \emph{not} required to be an integer.} $m'\in (\alpha/\dL',+\infty) $.
	
	Then, there is a constant $C>0$ such that, for every $X\in U_\lambda$, the following hold.
	\begin{enumerate}
		\item[\textnormal{(i)}] For every $f\in \Sc'(G)$
		\[
		\begin{split}
		&\norm{\ee^{-t_0\Lc'}X f}_{L^p(\beta)}+\big\| t^{-\alpha/\dL'}\norm{W'^{(m'),*}_{t,\eps} X f}_{L^p(\beta)}\big\|_{L^q_t(\mi)}\\
			&\qquad \qquad \qquad \qquad  \meg C\abs{X}  \big( \norm{\ee^{-t_1\Lc}f}_{L^p(\beta)}+\big\| t^{-(\alpha+\lambda)/\dL}\norm{W^{(m)}_{t}   f}_{L^p(\beta)}\big\|_{L^q_t(\nu)}\big).
		\end{split}
		\]

		\item[\textnormal{(ii)}] If $p\in (1,\infty)$,  then for every $f\in \Sc'(G)$
		\begin{equation}\label{eq:5}
		\begin{split}
		&\norm{\ee^{-t_0\Lc'}X f}_{L^p(\beta)}+\big\|\norm{t^{-\alpha/\dL' }(W'^{(m'),*}_{t,\eps} X f)(x)}_{L^q_t(\mi)}\big\|_{L^p_x(\beta)}\\
			&\qquad \qquad \meg C\abs{X} \big( \norm{\ee^{-t_1\Lc}f}_{L^p(\beta)}+\big\| \norm{t^{-(\alpha+\lambda)/\dL }(W^{(m)}_{t}   f)(x)}_{L^q_t(\nu)}\big\|_{L^p_x(\beta)}\big).
		\end{split}
		\end{equation}
	\end{enumerate}
      \end{prop}

      We conclude this section by recalling a result by Smulders~\cite[Theorem 2.2]{Smulders}
that will be used repeatedly in the sequel.
\begin{prop}\label{Smulders}
 Let $\Lc$ be a weighted subcoercive operator on $G$.  	Define 
	\[
	\Lambda(\phi)\coloneqq \Set{\ee^z\colon z\in \C, \abs{\Im z}<\phi}, \qquad \phi\in (0,\pi].
      \]
      Then,  the following
 hold.
 \begin{enumerate}   \item
For $p\in[1,\infty]$ and every $\omega\Meg \omega_0$,
$\Lc_\omega^{-1}$ induces an endomorphism of $ L^p(\beta)$;
\item the operator
  $\Lc_\omega$ has a bounded $H^\infty$ functional calculus in $L^p(\beta)$ on $\Lambda(\phi)$ for every $\omega\Meg \omega_0$ and $\phi\in (\pi/2-\theta, \pi]$. 
  \end{enumerate}
\end{prop}
  
      \section{Besov and Triebel-Lizorkin Spaces}\label{BTL:sec}
We are now ready to define Besov and Triebel--Lizorkin spaces on $G$.

\begin{deff}\label{BTL:def}
Suppose $\alpha \in \R$, $m_{\alpha}= ([\alpha/\dL]+1)_{+}$, and  $p,q\in [1,\infty]$. We define the Besov space $B^{p,q}_\alpha(\beta)$ and, if $p\in (1,\infty)$, the Triebel--Lizorkin space $F^{p,q}_\alpha(\beta)$, as the space of $f\in \Sc'(G)$ such that respectively
\[
\norm{f}_{B^{p,q}_\alpha(\beta)}\coloneqq \norm{\ee^{-\Lc} f}_{L^p(\beta)} + \bigg( \int_{0}^{1} \big(t^{-\alpha/\dL}\| (t\Lc)^{m_{\alpha}} \ee^{-t\Lc} f\|_{L^p(\beta)}\big)^{q} \, \frac{\dd t}{t}\bigg)^{1/q}
\]
and
\[
\norm{f}_{F^{p,q}_\alpha(\beta)}\coloneqq \norm{\ee^{-\Lc} f}_{L^p(\beta)} + \bigg\|\bigg( \int_{0}^{1} t^{-\alpha/\dL} | (t \Lc)^{m_{\alpha}} \ee^{-t\Lc} f|^{q} \, \frac{\dd t}{t}\bigg)^{1/q}\bigg\|_{L^p(\beta)}
\]
is finite  (with obvious modifications when $q=\infty$).
\end{deff}

For notational convenience, for $\alpha \in \R$, $m\in\N$ with $m>\alpha/\dL$, and  $p,q\in [1,\infty]$, we define for $f\in \Sc'(G)$
	\[
	\Bs^{p,q}_{\alpha,m}(f)\coloneqq \norm{   t^{-\alpha/\dL} \norm{ W_t^{(m)} f }_{L^p(\beta)}  }_{L^q_t(\mi_1)},
	\]
	and, if $p\in (1,\infty)$,
	\[
	\Fs^{p,q}_{\alpha,m}(f)\coloneqq \big\|  \norm{ t^{-\alpha/\dL}  W_t^{(m)} f(x)  }_{L^q_t(\mi_1)}\big\|_{L^p_x(\beta)},
	\]
so that
	\[
	\norm{f}_{B^{p,q}_\alpha(\beta)} =\norm{\ee^{-\Lc} f}_{L^p(\beta)}+ \Bs^{p,q}_{\alpha,([\alpha/\dL]+1)_+}(f),
	\]
	and, if $p\in (1,\infty)$,
	\[
	\norm{f}_{F^{p,q}_\alpha(\beta)} = \norm{\ee^{-\Lc} f}_{L^p(\beta)}+ \Fs^{p,q}_{\alpha,([\alpha/\dL]+1)_+}(f).
	\] 
As a consequence of Proposition~\ref{lem:8bis}, any positive integer $m>\alpha/\dL$ provides an equivalent Besov or Triebel--Lizorkin norm; and more importantly, the definition of $B^{p,q}_\alpha(\beta)$ and $F^{p,q}_\alpha(\beta)$ does \emph{not} depend on $\Lc$, in the sense than any other weighted subcoercive operator on $G$ gives rise to the same spaces. Actually more is true: as we state  for convenience in the next theorem, the Besov and Triebel--Lizorkin norms are independent of the choice of $t_{0}$, $\mu$, $\Lc$ and $m>\alpha/\dL$, in the sense that different choices generate the same spaces with equivalences of norms. 

\begin{teo}  \label{teo:2}
Suppose $\alpha\in \R$, $\eps\in (0,1]$, $t_0\in [0,+\infty)$,  $\mi\in \Mc_\Samp$, and $p,q\in [1,\infty]$. Assume that $t_0>0$ if $\alpha\meg 0$. Let $\Lc'$ be a  weighted subcoercive operator on $G$ and suppose $m\in \N$ with $m>\alpha/\dL'$.	
	Then, there is a constant $C\geq 1$ such that  the following hold.
\begin{enumerate}
		\item[\textnormal{(i)}] For every $f\in \Sc'(G)$
\[
		\norm{f}_{B^{p,q}_\alpha(\beta)} \asymp_{C} \norm{\ee^{-t_0\Lc'} f}_{L^p(\beta)}+\big\|   t^{-\alpha/\dL'} \norm{ W_t'^{(m)} f }_{L^p(\beta)}  \big\|_{L^q_t(\mi)}
		\]
%		\[
%		C^{-1} \norm{f}_{B^{p,q}_\alpha(\beta)}\meg \norm{\ee^{-t_0\Lc'} f}_{L^p(\beta)}+\big\|   t^{-\alpha/\dL} \norm{ W_t'^{(m)} f }_{L^p(\beta)}  \big\|_{L^q_t(\mi)}\meg C \norm{f}_{B^{p,q}_\alpha(\beta)}
%		\]
		and
		\[
		  \norm{f}_{B^{p,q}_\alpha(\beta)} \asymp_{C} \norm{\ee^{-t_0\Lc'} f}_{L^p(\beta)}+\big\|   t^{-\alpha/\dL'} \norm{ W_{t,\eps}'^{(m),*} f }_{L^p(\beta)}  \big\|_{L^q_t(\mi)}.
		\]
%		\[
%		C^{-1} \norm{f}_{B^{p,q}_\alpha(\beta)}\meg \norm{\ee^{-t_0\Lc'} f}_{L^p(\beta)}+\big\|   t^{-\alpha/\dL} \norm{ W_{t,\eps}'^{(m),*} f }_{L^p(\beta)}  \big\|_{L^q_t(\mi)}\meg C \norm{f}_{B^{p,q}_\alpha(\beta)}.
%		\]		
		\item[\textnormal{(ii)}] If $p\in ( 1,\infty)$, then for every $f\in \Sc'(G)$
\[
		  \norm{f}_{F^{p,q}_\alpha(\beta)} \asymp_{C} \norm{\ee^{-t_0\Lc'} f}_{L^p(\beta)}+\big\|  \norm{   t^{-\alpha/\dL'}  W_t'^{(m)} f  (x)}_{L^q_t(\mi)}\big\|_{L^p_x(\beta)}
		\]
%		\[
%		C^{-1} \norm{f}_{F^{p,q}_\alpha(\beta)}\meg \norm{\ee^{-t_0\Lc'} f}_{L^p(\beta)}+\big\|  \norm{   t^{-\alpha/\dL}  W_t'^{(m)} f  (x)}_{L^q_t(\mi)}\big\|_{L^p_x(\beta)} \meg C \norm{f}_{F^{p,q}_\alpha(\beta)}
%		\]
		and
		\[
		\norm{f}_{F^{p,q}_\alpha(\beta)}\asymp_{C} \norm{\ee^{-t_0\Lc'} f}_{L^p(\beta)}+\big\| \norm{   t^{-\alpha/\dL'}  W_{t,\eps}'^{(m),*} f  (x)}_{L^q_t(\mi)}\big\|_{L^p_x(\beta)}.
		\]
%		\[
%		C^{-1} \norm{f}_{F^{p,q}_\alpha(\beta)}\meg \norm{\ee^{-t_0\Lc'} f}_{L^p(\beta)}+\big\| \norm{   t^{-\alpha/\dL}  W_{t,\eps}'^{(m),*} f  (x)}_{L^q_t(\mi)}\big\|_{L^p_x(\beta)} \meg C \norm{f}_{F^{p,q}_\alpha(\beta)}.
%		\]
	\end{enumerate}
\end{teo}

As mentioned already, Theorem~\ref{teo:2} implies that the spaces $F^{p,q}_\alpha(\beta)$ and $B^{p,q}_\alpha(\beta)$ are \emph{intrinsic} in the sense that they do not depend on the specific weighted subcoercive operator used to define them; however, they do depend on the measure $\beta$ and the chosen filtration on $G$. In particular, if $\gf_1$ is a vector subspace of $\gf$ which generates $\gf$ as a Lie algebra, and if $(\gf_\lambda)$ is the filtration generated by $\gf_1$, then we recover the spaces introduced in~\cite{BPV} defined via sub-Laplacians with drift $\Delta_{\chi}$. By Theorem~\ref{teo:2}, these can be characterized by means of a broader class of operators (recall Remark~\ref{comparison}).

\begin{prop}\label{prop:5}
Suppose $\alpha\in \R$ and $p,q\in [1,\infty]$. Then 
\begin{enumerate}
\item[\textnormal{(i)}]  $B^{p,q}_\alpha(\beta)= \Delta_L^{-1/p} B^{p,q}_\alpha(\beta_L)$, and
\item[\textnormal{(ii)}]  $F^{p,q}_\alpha(\beta)= \Delta_L^{-1/p} F^{p,q}_\alpha(\beta_L)$  if $p\in (1,\infty)$.
\end{enumerate}
\end{prop}

\begin{proof}
	This follows from Theorem~\ref{teo:2}, once we observe that $\beta=\Delta_L \cdot \beta_L$  and 
	\[
	\Delta_L^{1/p}\ee^{- t \Lc }=\ee^{-t \Lc'} \Delta_L^{1/p},
	\]
	where $\Lc'$ is the left-invariant differential operator defined so that $\Lc' \phi= \Delta_L^{1/p} \Lc(\Delta_L^{-1/p}\phi)$ for every $\phi\in C^\infty(G)$; observe that $\deg(\Lc-\Lc')<\dL$ by Proposition~\ref{prop:9}, so that $\Lc'$ is in fact weighted subcoercive.
\end{proof}
We now prove that the spaces $B^{p,q}_\alpha(\beta)$ and $F^{p,q}_\alpha(\beta)$ defined above are actually Banach spaces.  To do so, we shall identify such spaces with suitable Banach spaces. Observe that for $f\in \Sc'(G)$, if $f_{0} = \ee^{-\Lc}f$ and $f_{t} = W^{(m)}_{t} f $ for $t>0$ and $m\in\N$,  one has
\[
(t \Lc)^m f_0= \ee^{-(1-t)\Lc} f_t  \qquad \mbox{if } 0<t\leq 1,
\]
while
\[
f_t= (t \Lc)^m \ee^{-(t-1)\Lc} f_0 \qquad \mbox{if } t>1.
\]

Therefore, we may identify the family $(f_{t})_{t\geq 0}$ (in terms of which Besov and Triebel--Lizorkin norms are defined) with a suitable subspace  of 
\begin{equation}\label{Fm}
	\Fc_m:= \bigg\{ (f_t)\in \Sc'(G)^{[0,+\infty)}\colon 
	\begin{cases}
	(t \Lc)^m f_0= \ee^{-(1-t)\Lc} f_t  \quad &\mbox{if } t\in (0,1],\\
	f_t= (t \Lc)^m \ee^{-(t-1)\Lc} f_0 &\mbox{if } t\in [1,+\infty)
	\end{cases}
	\bigg\}.
\end{equation}
Notice moreover that for $f\in \Sc'(G)$, if again $f_{0} = \ee^{-\Lc}f$ and $f_{t} = W^{(m)}_{t} f $ for $t>0$, and $m = ([\alpha/\dL]+1)_{+}$,  then
\begin{align*}
\norm{f}_{F^{p,q}_\alpha(\beta)} 
%&= \norm{\ee^{-\Lc} f}_{L^p(\beta)}+\big\|   t^{-\alpha/\dL} \norm{ W_t^{(m)} f }_{L^p(\beta)}  \big\|_{L^q_t(\mi)}\\
& =\norm{f_{0}}_{L^p(\beta)}+\big\|   \norm{t^{-\alpha/\dL} f_{t} }_{L^p(\beta)}  \big\|_{L^q_t(\mi)}\\
& =  \big\| \norm{f_{t}}_{L^p(\beta)} \big\|_{L^{q}_{t}(\delta_{0})}+\big\|   \norm{t^{-\alpha/\dL} f_{t} }_{L^p(\beta)}  \big\|_{L^q_t(\mi)} \asymp  \big\|   \norm{ \theta_\alpha  (\,\cdot\,,t) f_{t} }_{L^p(\beta)}  \big\|_{L^q_t(\mi+ \delta_{0})} 
\end{align*}
where  $\delta_{0}$ is the Dirac delta at $0$, and
\begin{equation}\label{thetaalpha}
\theta_\alpha  (x,t) =
\begin{cases}
 1   \quad &\mbox{if }  t=0 \\
 t^{\alpha/\dL}   &\mbox{if }  t\in (0,+\infty)
 \end{cases}
, \quad x\in G.
\end{equation}
Note  that $\mi+\delta_0$ may be interpreted as a Radon measure on
\begin{equation}\label{topological-sum:eq}
R\coloneqq \{0\} \coprod(0,+\infty),
  \end{equation}
  that is, the  topological
  sum of $\{0\}$ and $(0,+\infty)$.  With this topology, $\theta_\alpha$ becomes a continuous function.

Then we have the following  lemma, which will also be of use to deal with duality and interpolation (cf.~\cite[Theorem 6.1]{BPV}). The notation introduced in~\eqref{Fm} and~\eqref{thetaalpha} will be fixed from now on.

\begin{lem} \label{lem:22}
	Suppose $m\in\N$, and consider the map
	\[
	T_m\colon \Sc'(G)\to \Fc_m, \qquad (T_m f)_{t} =
	\begin{cases}
	\ee^{-\Lc}f \qquad &\mbox{if } t=0,\\
	W^{(m)}_{t} f &\mbox{if } t\in (0,+\infty).
	\end{cases}
	\]
Take $p,q\in [1,\infty]$, $\mi\in \Mc_\Samp$, and  $\alpha<m\dL$. Then $T_m$ induces an isomorphism
	\begin{enumerate}
		\item[\textnormal{(i)}]  of $B^{p,q}_\alpha(\beta)$ onto $(\theta_\alpha L^{p,q}(\beta,\mi+\delta_0))\cap \Fc_m$, and

		\item[\textnormal{(ii)}] of $F^{p,q}_\alpha(\beta)$ onto $(\theta_\alpha  L^{q,p}(\mi+\delta_0,\beta))\cap \Fc_m$, if $p\in (1,\infty)$.
	\end{enumerate}
\end{lem}

\begin{proof} 
	Observe first that, if $(f_t)\in \Fc_m$, then the mapping $G\times (0,+\infty)\ni (x,t)\mapsto f_t(x)$ is well defined and continuous. This follows from the fact that 
	\[
	f_{t+s}=(1+t/s)^m\ee^{-t\Lc} f_s, \qquad t,s>0,
	\]
which is a consequence of the conditions which define $\Fc_m$ and the injectivity of $\ee^{-t'\Lc}$ on $\Sc'(G)$, for every $t'>0$.\footnote{Since $\ee^{-t' \Lc}$ is one-to-one on $L^2(\beta)$, by (sesquilinear) transposition $\ee^{-t'\Lc^*}L^2(\beta)$ is dense in $L^2(\beta)$ for the weak topology. Consequently, $\ee^{-t'\Lc^*}\Sc'(G)\supseteq \ee^{-t'\Lc^*}L^2(\beta)$ is dense in $\Sc'(G)$ for the weak topology $\sigma(\Sc'(G),\Sc(G))$, since $L^2(\beta)$ is dense in $\Sc'(G)$ and the inclusion $L^2(\beta)\subseteq \Sc'(G)$ is continuous for the weak topologies. Consequently, by (sesquilinear) transposition we see that $\ee^{-t'\Lc}$ is one-to-one on $\Sc'(G)$.} Consequently, there are no measurability issues.
	
	(i) Notice that by Theorem~\ref{teo:2}, $T_m$ induces an isomorphism of  $B^{p,q}_\alpha(\beta)$ onto a subspace of $ (\theta_\alpha L^{p,q}(\beta,\mi+\delta_0))\cap \Fc_m$. We only need to prove that $T_m$ is actually surjective, that is, for $(f_t)\in (\theta_\alpha L^{p,q}(\beta,\mi+\delta_0))\cap \Fc_m$ we need to show that there is $f\in \Sc'(G)$ such that $(f_t)=T_m f$.  If $m=0$, then the assertion follows directly from Lemma~\ref{lem:27}. If, otherwise, $m\Meg 1$, then in view of Lemma~\ref{lem:9} it is tempting to set 
	\[
	f\coloneqq \sum_{h=0}^{m-1}\frac{1}{h!} \Lc^h f_0+ \frac{1}{(m-1)!} \int_0^1 f_t \,\frac{\dd t}{t}.
	\]
However, in order to show that the above $f$ is well defined and has the required properties, it is convenient to first define
\[
g_t \coloneqq 
	\begin{cases}
	\displaystyle \sum_{h=0}^{m-1}\frac{(1-t)^h}{h!} \Lc^h f_0+\frac{1}{(m-1)!} \int_t^1 f_s (1-t/s)^{m-1} \,\frac{\dd s}{s}  \qquad &\mbox{if } t\in (0,1)\\
\ee^{-(t-1)\Lc} f_0  &\mbox{if } t\in [1,\infty).
	\end{cases}
\]
Then $g_t$ is well defined (and belongs to $L^p(\beta)$) for every $t\in (0,1)$ as $m\geq 1$. We first prove that
\begin{equation}\label{gtt'}
\ee^{-t'\Lc} g_t=g_{t+t'}
\end{equation}
for every $t,t'>0$. This is clear if $t\Meg 1$. Then, assume that $t\in (0,1)$. Assume first that $t+t'\meg 1$. Then,
	\[
	\begin{split}
		\ee^{-t'\Lc} g_t&=\sum_{h=0}^{m-1}\frac{(1-t)^h}{h!} \Lc^h \ee^{-t' \Lc}f_0+\frac{1}{(m-1)!} \int_t^1 \ee^{-t'\Lc}f_s (1-t/s)^{m-1} \,\frac{\dd s}{s} \\
			&=	\sum_{h=0}^{m-1}\frac{(1-t)^h}{h!} \Lc^h \ee^{-t' \Lc}f_0+\frac{1}{(m-1)!} \int_t^1  f_{t'+s} (1+t'/s)^{-m}(1-t/s)^{m-1} \,\frac{\dd s}{s} \\
			&=	\sum_{h=0}^{m-1}\frac{(1-t)^h}{h!} \Lc^h \ee^{-t' \Lc}f_0+\frac{1}{(m-1)!} \int_t^1  f_{t'+s} (s+t')^{-m}(s-t)^{m-1} \,\dd s\\
			&=	\sum_{h=0}^{m-1}\frac{(1-t)^h}{h!} \Lc^h \ee^{-t' \Lc}f_0+\frac{1}{(m-1)!} \int_{t+t'}^{1+t'}  f_{s} s^{-m}(s-t-t')^{m-1} \,\dd s.
	\end{split}
	\]	
	It remains to observe that, since 
	\[
	\Lc^{m} \ee^{-(s-1)\Lc} f_{0} = (-1)^{m} \frac{\dd^{m}}{\dd s^{m}}\ee^{-(s-1)\Lc} f_{0}
	\]
	by Lemma~\ref{lem:9},
\begin{align*}
 \frac{1}{(m-1)!} \int_{1}^{1+t'} \!\!\! f_{s} s^{-m}(s-t-t')^{m-1} \,\dd s 
 & = \frac{1}{(m-1)!} \int_{1}^{1+t'}\!\!\! \Lc^{m} \ee^{-(s-1)\Lc} f_{0} (s-t-t')^{m-1} \,\dd s \\
 &= \sum_{h=0}^{m-1}\frac{(1-t-t')^h}{h!} \Lc^h  f_0 - \sum_{h=0}^{m-1}\frac{(1-t)^h}{h!} \Lc^h \ee^{-t' \Lc}f_0,
 \end{align*}
as one can obtain integrating by parts. Then~\eqref{gtt'} follows. The case $t+t'> 1$ then follows too, since 
	\[
	\ee^{-t'\Lc} g_t= \ee^{-(t'+t-1)\Lc} \ee^{-(1-t)\Lc} g_t=\ee^{-(t'+t-1)\Lc} g_1= g_{t+t'}.
	\]
	Let us then prove that $f_t=(t\Lc)^m g_t$ for every $t>0$. This is clear when $t\Meg 1$. If, otherwise, $t<1$, then observe that 
	\[
	\ee^{-(1-t)}f_t=(t \Lc)^m f_0= \ee^{-(1-t)} (t\Lc)^m g_t,
	\] so that the assertion follows from the injectivity of $\ee^{-(1-t)\Lc}$ on $\Sc'(G)$ as before. Therefore, 
	\[
	\norm{t^{m-\alpha} (\Lc^m g_t)(x)}_{L^{p,q}_{x,t}(\beta,\mi)}=\norm{t^{-\alpha} f_t(x)}_{L^{p,q}_{x,t}(\beta,\mi)}
	\]
	is finite, so that Lemma~\ref{lem:27} ensures the existence of $g\in \Sc'(G)$ such that $\ee^{-t\Lc} g=g_t$ for every $t>0$, so that $(f_t)=T_m g$. One may then verify, \emph{a posteriori}, that $f=g$.
	
	The case (ii) is analogous (and follows from this one by means of Proposition~\ref{prop:1}).
\end{proof}

\begin{prop}\label{prop:3}
If $\alpha\in \R$ and $p,q\in [1,\infty]$, then $B^{p,q}_\alpha(\beta)$ is a Banach space. If in addition $p\in (1,\infty)$, then $F^{p,q}_\alpha(\beta)$ is a Banach space.
\end{prop}
\begin{proof}
The statement follows from Lemma~\ref{lem:22} after noticing that $(\theta_\alpha L^{p,q}(\beta,\mi+\delta_0))\cap \Fc_m$ and $(\theta_\alpha  L^{q,p}(\mi+\delta_0,\beta))\cap \Fc_m$ if $p\in (1,\infty)$ are Banach spaces, as they are closed (whence complete) in $\theta_\alpha L^{p,q}(\beta,\mi+\delta_0))$ and $\theta_\alpha  L^{q,p}(\mi+\delta_0,\beta))$ respectively.
\end{proof}

In order to get the real and complex interpolation properties of the spaces $B^{p,q}_\alpha(\beta)$ and $F^{p,q}_\alpha(\beta)$, we shall reduce to the interpolation properties of the spaces $(\theta_\alpha L^{p,q}(\beta,\mi+\delta_0))\cap \Fc_m$ and $(\theta_\alpha L^{q,p}( \mi+\delta_0,\beta))\cap \Fc_m$ via a retraction argument. We shall need the following. 
\begin{lem}\label{lem:24}
Suppose $p,q\in [1,\infty]$, $\alpha\in\R$, $\mi\in \Mc_\Samp$, and $m\in\N$ with $m>\alpha/\dL$. Let $\kappa>1$ and $\eps\in (0,1)$ be so that the function
	\[
\xi\colon  (0,2] \to (0,+\infty),  \quad \xi( t) = 1/\mi([\eps t/\kappa, t/\kappa ])
	\]
	is bounded.  For $m\in\N$ define 
	\[
	P_m (f_t) \coloneqq \sum_{h=0}^{2m-1} \frac{2^h}{h!} \Lc^h \ee^{-\Lc}f_0+ \frac{1}{(2m-1)!} \int_{(0,2\eps/\kappa]} s^{-m}\int_{\kappa s}^{\kappa s/\eps} \xi(t) t^{2m}\Lc^m\ee^{-(t-s) \Lc} f_{s} \,\frac{\dd t}{t}\,\dd \mi(s) 
	\]
	for every $(f_t)\in \Sc'(G)^{[0,+\infty)}$ for which the
integral is defined. Then, 
	\[
	T_m(\Sc'(G))\subseteq \mathrm{dom} \, P_m, \qquad P_m T_m f=f \quad \forall \, f\in \Sc'(G),
	\]
	and $\Pc_m\coloneqq T_m P_m $ is a projector of $\mathrm{dom}
        P_m$ onto $T_m(\Sc'(G))$.

        In addition, for every $p,q\in [1,\infty]$ and $\alpha\in (-m \dL,m \dL)$, $\Pc_m$ induces a continuous linear projector  
	\begin{enumerate}
	\item of $\theta_\alpha L^{p,q}(\beta,\mi+\delta_0) $ onto $(\theta_\alpha L^{p,q}(\beta,\mi+\delta_0))\cap \Fc_m$, and
	\item of $ \theta_\alpha L^{q,p}( \mi+\delta_0,\beta)$ onto $(\theta_\alpha L^{q,p}( \mi+\delta_0,\beta))\cap \Fc_m$  if $p\in (1,\infty)$.	\end{enumerate}
\end{lem}

\begin{proof}
	Take $g\in \Sc'(G)$. Then, by means of Lemma~\ref{lem:22} and Fubini's theorem we see that $T_m g\in \mathrm{dom} P_m$ and
	\[
	\begin{split}
	P_m T_m g&= \sum_{h=0}^{2m-1} \frac{2^h}{h!} \Lc^h \ee^{-2 \Lc} g+ \frac{1}{(2m-1)!} \int_{(0,2\eps/\kappa]} \int_{\kappa s}^{\kappa s/\eps} \xi(t) t^{2m} \Lc^{2m} \ee^{-t\Lc}g\,\frac{\dd t}{t}\,\dd \mi(s)\\
		&=\sum_{h=0}^{2m-1} \frac{2^h}{h!} \Lc^h \ee^{-2 \Lc} g+ \frac{1}{(2m-1)!} \int_0^2 \int_{[\eps t/\kappa, t/\kappa ]} \,\dd \mi(s)  \xi(t) t^{2m} \Lc^{2m} \ee^{-t\Lc}g\,\frac{\dd t}{t}\\
		&=\sum_{h=0}^{2m-1} \frac{2^h}{h!} \Lc^h \ee^{-2 \Lc} g+ \frac{1}{(2m-1)!} \int_0^2   t^{2m} \Lc^{2m} \ee^{-t\Lc}g\,\frac{\dd t}{t}\\
		&=g.
	\end{split}
	\]
	Consequently, $\Pc_m$ is a projector of $\mathrm{dom}(P_m)$ onto $T_m(\Sc'(G))$.
	Concerning the continuity of $\Pc_m$, we only prove (2).
	
	Pick $(f_t) \in  \theta_\alpha L^{q,p}( \mi+\delta_0,\beta)$ and set
\[
P_{m,1} (f_t)  \coloneqq \sum_{h=0}^{2m-1} \frac{2^h}{h!} \Lc^h\ee^{-\Lc} f_0, \qquad P_{m,2}\coloneqq P_m-P_{m,1}, \qquad \Pc_{m,j}=T_m P_{m,j}, \quad j=1,2.
\]
By the Gaussian estimates  (cf.~Theorem~\ref{teo:7}), there are $b, C_0 >0$ such that $\abs{\Lc^k h_t}\meg C_0  t^{-k} p_{b,t}$ for every $k=0,\dots, 2m$ and $t\in (0,2(1+\sup \supp \mi)]$.	
	Then, by Young's inequality, 
	\[
	\begin{split}
		\norm{[\Pc_{m,1} (f_t) ]_0}_{L^p(\beta)}&\meg C_1
                \norm{T_{b,1} f_0}_{L^p(\beta)}\meg C_1 \norm{p_{b,1} \Delta_R^{-1/p'}}_{L^1(\beta)} \norm{f_0}_{L^p(\beta)}.
	\end{split}
	\]
	In addition, by Lemma~\ref{lem:2} there is a constant $C_2>0$ such that
	\[
	\begin{split}
		\big\|\norm{t^{-\alpha/\dL } [\Pc_{m,1} (f_t) ]_{t}(x)  }_{L^q(\mi)} \big\|_{L^p_x(\beta)}&\meg \sum_{h=0}^{2m-1}\frac{2^h}{h!}\big\|\norm{t^{m-\alpha/\dL } ( \Lc^{h+m} \ee^{-(t+1)\Lc}f_0 )(x)  }_{L^q_t(\mi)} \big\|_{L^p_x(\beta)}\\
		&\meg C_1\sum_{h=0}^{2m-1}\frac{1}{h!}\big\|\norm{t^{m-\alpha/\dL } (   T_{b,1+t}f_0 )(x)  }_{L^q_t(\mi)} \big\|_{L^p_x(\beta)}\\
		&\meg C_2 \big\| \norm{t^{m-\alpha/\dL } f_0(x)  }_{L^q_t(\mi)} \big\|_{L^p_x(\beta)}\\
		&= C_2 \norm{t^{m-\alpha/\dL}}_{L^q_t(\mi)} \norm{f_0}_{L^p(\beta)},
	\end{split}
      \]
      thanks to Corollary~\ref{cor:carleson-measures}, 
	whence the continuity of $\Pc_{m,1}$. Then, observe that 
		\[
	\begin{split}
		\| [\Pc_{m,2} &(f_t) ]_0\|_{L^p(\beta)} \meg   \bigg\|\int_{ (0,2\eps/\kappa]} s^{-m} \int_{\kappa s}^{\kappa s/\eps} \xi(t)t^{2m}\abs{\Lc^m \ee^{-(t-s+1)\Lc} f_{s}(x)} \,\frac{\dd t}{t}\,\dd \mi(s)\bigg\|_{L^p_x(\beta)}\\
		&\meg C_3\norm{\xi}_{L^\infty} \bigg\|\int_{ (0,2\eps/\kappa]} s^{-m} \int_{\kappa s}^{\kappa s/\eps} t^{2m} T_{b,t-s+1} f_{s}(x) \,\frac{\dd t}{t}\,\dd \mi(s)\bigg\|_{L^p_x(\beta)} \\
		&\meg C_3\norm{\xi}_{L^\infty}(\kappa/\eps)^{2m} 3^{Q_*/\dL} \abs{\log \eps}\bigg\|\int_{ (0,2\eps/\kappa]} s^{m}   T_{b,3} f_{s}(x)\,\dd \mi(s)   \bigg\|_{L^p_x(\beta)} \\
		& \meg C_3\norm{\xi}_{L^\infty}(\kappa/\eps)^{2m} 3^{Q_*/\dL}\abs{\log \eps} \norm{s^{m+\alpha/\dL}}_{L^{q'}_s(\mi)} \big\|\norm{ s^{-\alpha/\dL}   T_{b,3} f_{s}(x)}_{L^q_s(\mi)}  \big\|_{L^p_x(\beta)}  ,
	\end{split}
	\]
	which gives the desired estimate of $\norm{[\Pc_{m,2} (f_t) ]_0}_{L^p(\beta)}$, thanks to Lemma~\ref{lem:2}.
	Next, observe that, for every $s,s'>0$ and $t\in [\kappa
        s,\kappa s/ \eps]$,
        \[
	\begin{split}
		\abs{\Lc^{2m} \ee^{-(t-s+s')\Lc}f_{s}}&\meg C_1
                T_{b,s'}\Lc^{2m} \ee^{-(t-s)\Lc} f_{s}\\
		&\meg C_1^2 (t-s)^{-2m} T_{b,s'}T_{b,t-s}  f_{s}\\
		&\meg C_1^2 s^{-2m} (\kappa-1)^{-2m} (\kappa/[\eps(\kappa-1)])^{Q_*/\dL}  T_{b,s'}T_{b,\kappa s/\eps}  f_{s},
	\end{split}
	\]
since $(\kappa-1)s\meg t-s\meg \kappa s/\eps-s\meg \kappa s/\eps$, but also
	\[
	\begin{split}
		\abs{\Lc^{2m} \ee^{-(t-s+s')\Lc}f_{s}}&\meg C_1 s'^{- 2m} T_{b,s'}  \ee^{-(t-s)\Lc}f_{s}\\
		&\meg C_1^2 s'^{-2m} T_{b,s'}T_{b,t-s}  f_{s} \meg C_1^2 s'^{-2 m} (\kappa/[\eps(\kappa-1)])^{Q_*/\dL}    T_{b,s'}T_{b,\kappa s/\eps}  f_{s},
	\end{split}
	\]
	so that there is a constant $C_2>0$ such that
	\[
	\abs{\Lc^{2m} \ee^{-(t-s+s')\Lc}f_{s}}\meg C_2 (s+s')^{-2m} T_{b,s'}T_{b,\kappa s/\eps}  f_{s}.
	\]
	Therefore, by Fubini's theorem and Lemma~\ref{lem:2} there is a constant $C_3>0$ such that
	\[
	\begin{split}
		&\big\|\norm{  s'^{-\alpha/\dL} [\Pc_{m,2} (f_t) ]_{s'}}_{L^q_{s'}(\mi)}\big\|_{L^p_x(\beta)}\\
		&\meg \bigg\| \bigg\|  s'^{m-\alpha/\dL} \int_{(0,2\eps/\kappa]} s^{-m}   \int_{\kappa s}^{\kappa s/\eps} \xi(t) t^{2m} \Lc^{2m} \ee^{-(t-s+s')\Lc}f_{s} \,\frac{\dd t}{t}\,\dd \mi(s)\bigg\|_{L^q_{s'}(\mi)}\bigg\|_{L^p_x(\beta)}\\
		&  \meg C_2\norm{\xi}_{L^\infty}\Big( \frac{\kappa}{\eps}\Big)^{2m}\abs{\log \eps}\bigg\| \bigg\|  s'^{m-\frac{\alpha}{\dL}} \left( T_{b,s'}\int_{(0,\frac{2\eps}{\kappa}]}   \frac{s^{m}}{(s+s')^{2m}}  T_{b, \kappa s/\eps}f_{s}  \,\dd \mi(s)\right) (x)\bigg\|_{L^q_{s'}(\mi)}\bigg\|_{L^p_x(\beta)}\\
		&  \meg C_3  \bigg\|\bigg\|  s'^{m-\alpha/\dL} \int_{(0,2\eps/\kappa]}   \frac{s^{m}}{(s+s')^{2m}} T_{b, \kappa s/\eps}f_{s}  \,\dd \mi(s)\bigg\|_{L^q_{s'}(\mi)}\bigg\|_{L^p_x(\beta)}.
	\end{split}
	\]
	Then, Lemma~\ref{lem:25}  shows that there is a constant $C_4>0$ such that
	\[
	\begin{split}
		&\big\| \norm{  s'^{-\alpha/\dL} [\Pc_{m,2} (f_t) ]_{s'}}_{L^q_{s'}(\mi)}\big\|_{L^p_x(\beta)}\meg C_4 \big\|\norm*{  s^{-\alpha/\dL}   T_{b, \kappa s/\eps}f_{s}   }_{L^q_{s}(\mi)}\big\|_{L^p_x(\beta)}.
	\end{split}
	\]
	The conclusion then follows from Lemma~\ref{lem:2}.	
\end{proof}

The next result is the analogue of~\cite[Theorem 4.5]{BPV}, and is a characterization of Besov and Triebel--Lizorkin spaces by recursion.

\begin{teo}\label{teo:3}
	Suppose $\alpha\in \R$, $t_0\Meg 0$, $p,q\in [1,\infty]$, and $h\in\N$. Let $(X_j)_{j\in J}$ be a minimal basis in $\gf$, and set $d_j\coloneqq \deg X_j$ for every $j\in J$. Assume that $t_0>0$ if $\alpha\meg 0$. Take $f\in \Sc'(G)$ such that $\ee^{-t_0\Lc }f\in L^p(\beta)$. Then, the following hold.
	\begin{enumerate}
		\item[\textnormal{(i)}] $f\in B^{p,q}_{\alpha+h\dd}(\beta)$ if and only if $X_j^{h\dd/d_j} f\in B^{p,q}_{\alpha}(\beta)$ for every $j=1,\dots, k$.

		\item[\textnormal{(ii)}] If $p\in  (1,\infty)$, then $f\in F^{p,q}_{\alpha+h\dd}(\beta)$ if and only if $X_j^{h\dd/d_j} f\in F^{p,q}_{\alpha }(\beta)$ for every $j\in J$.
	\end{enumerate}
\end{teo}

\begin{proof}
	By means of Theorem~\ref{teo:2} and Proposition~\ref{prop:4}, we may assume that  
	\[
	\Lc= \sum_{j\in J} (X_j^{h\dd/d_j})^\dag X_j^{h\dd/d_j}.
	\]
	In particular, $\dL= 2h\dd$.

	(ii) Observe that, by Proposition~\ref{lem:8bis} and Theorem~\ref{teo:2}, there is a constant $C_1>0$ such that, for every  $X\in U_{h\dd}$, 
	\[
	\norm{X f}_{F^{p,q}_\alpha(\beta)}\meg C_1\abs{X} \norm{f}_{F^{p,q}_{\alpha+\deg(X)}(\beta)}
	\]
	for every $f\in \Sc'(G)$.  
	This implies that the mapping
        $F^{p,q}_{\alpha+\deg(X)}(\beta)\ni f\mapsto Xf\in
        F^{p,q}_\alpha(\beta) $ is continuous. In particular, if $f\in
        F^{p,q}_{\alpha+ h\dd}(\beta)$, then $X_j^{h\dd/d_j} f\in F^{p,q}_{\alpha}(\beta)$ for every $j\in J$.

	Conversely, assume that $X_j^{h\dd/d_j} f\in F^{p,q}_{\alpha}(\beta)$ for every $j\in J$, and let us prove that $f\in F^{p,q}_{\alpha+h\dd}(\beta)$. Take $m\in\N$ so that $m>(\alpha+h\dd)/\dL$.
	Arguing as before, we see that   there is a constant $C_2>0$ such that 
	\[
	\begin{split}
	\Fs^{p,q}_{\alpha+h\dd,m} (f) &=\Fs^{p,q}_{\alpha-h\dd,m-1}(\Lc f) \\
		&\meg C_2\sum_{j\in J} \norm{(X_j^{h\dd/d_j})^* X_j^{h\dd/d_j} f}_{F^{p,q}_{\alpha-h\dd}(\beta)} \meg C_2^2 \sum_{j\in J}\norm{ X_j^{h\dd/d_j} f}_{F^{p,q}_{\alpha}(\beta)} 
	\end{split}
	\]
	for every $f\in \Sc'(G)$, whence the result since $\ee^{-t_0\Lc}f\in L^p(\beta)$ by assumption.	
\end{proof}

The next result should be compared with~\cite[Theorem 13]{BPV2}.
\begin{teo}\label{teo:5}
	There is $\omega_0\in \R$ such that for every $\alpha, \alpha'\in\R$, $p,q\in [1,\infty]$ and $\omega \Meg \omega_0$, the operators 
	\[
	(1+\abs{\gamma})^{-1/2}\ee^{-\pi\abs{\gamma}/2}\Lc_\omega^{(\alpha-\alpha')/\dL+ i \gamma}, \qquad \gamma\in\R,
	\]
	induce equicontinuous canonical isomorphisms of $B^{p,q}_\alpha(\beta)$ onto $B^{p,q}_{\alpha'}(\beta)$ and, if $p\in (1,\infty)$, of $F^{p,q}_\alpha(\beta)$ onto $F^{p,q}_{\alpha'}(\beta)$.
\end{teo}

Let us stress that $\Lc_\omega^{(\alpha-\alpha')/\dL}$ need \emph{not} be defined on the whole of $\Sc'(G)$ (unless $\alpha-\alpha'\in \dL \N$, of course). In fact, the convolution kernel of $(\Lc^*_\omega)^{(\alpha-\alpha')/\dL}$ decays exponentially, but not more than exponentially (in general), so that $(\Lc^*_\omega)^{(\alpha-\alpha')/\dL}$ need  \emph{not} preserve $\Sc(G)$. %By sesquilinear transposition, this means that $\Lc_\omega^{(\alpha-\alpha')/\dL}$ need not be defined on the whole of $\Sc'(G)$.

\begin{proof}
As a first step, we assume that $\alpha-\alpha'=k \dL$ for some
$k\in\N$. Take $\omega_0$ so that $\Lc_\omega^{-1}$ induces an
endomorphism of $L^s(\beta)$ for every $s\in [1,\infty]$ and  $\omega
\Meg \omega_0$, see Proposition~\ref{Smulders}. Then, by means of  Proposition~\ref{lem:8bis}, we may reduce to the case in which $\omega=0$.
	The fact that $\Lc^k$ induces a continuous linear mapping from $B^{p,q}_{\alpha}(\beta)$ into $B^{p,q}_{\alpha-k\dL}(\beta)$ follows from Proposition~\ref{lem:8bis}.

Then, take $g\in B^{p,q}_{\alpha-k\dL}(\beta)$. Using
Lemma~\ref{lem:27}, we are going to show that that $g=\Lc^k f$ for some $f\in
B^{p,q}_{\alpha}(\beta)$, so that the conclusion will follow
thanks to Proposition~\ref{prop:3} and the open mapping theorem.  Since
  $g\in B^{p,q}_{\alpha-k\dL}(\beta)$,  $f_t\coloneqq\Lc^{-k} e^{-t\Lc}g \in L^p(\beta)$ for all
$t>0$, and $f_{t+s}=  e^{-t\Lc} f_s$, for all $s,t>0$. Then,
Lemma~\ref{lem:27} shows that there exists $f\in
 \Sc'(G)$ such that $f_t = e^{-t\Lc} f$ for all $t>0$, that is, $ g=
 \Lc^k f$.  It is clear that $f\in B^{p,q}_{\alpha}(\beta)$.
The Triebel--Lizorkin case is
 proved similarly.

By means of the above, we may  reduce to the case in which
$\alpha-\alpha'<0$. Indeed, it suffices to write $\alpha-\alpha'= k\dL-r$, for some
$k\in\N$ and $r>0$.  Thus, assume that $\alpha-\alpha'<0$ and take $m\in \N$ such that $m>\alpha'/\dL$.
	Observe that there are $b,\omega_1,C_1>0$ such that 
	\[
	\abs{\Lc^j h_t}\meg C_1 \ee^{\omega_1 t} p_{b,t}, \quad t>0, \; j=0,\dots, m, \qquad c\norm{p_{b,t} \ee^{c \abs{\,\cdot\,}_*}}_{L^1(\beta)}\meg C_1 \ee^{\omega_1 t}, \qquad t>0
	\] 
	(cf.~Lemma~\ref{lem:3} and Theorem~\ref{teo:7}), where $c>1$ is such that $\Delta_R(x)\meg c\, \ee^{c \abs{x}_*}$ for every $x\in G$.  We then take $\omega_0\in \R$ so that $\omega_0>2\omega_1$ and so that Lemma~\ref{lem:2} applies to $\omega_0$ (hence to every $\omega\Meg \omega_0$). Notice that this implies that $\Lc_\omega^{-\alpha''}$ induces an endomorphism of $L^s(\beta)$ for every $s\in [1,\infty]$, for every $\omega \Meg \omega_0$ and $\alpha''>0$, as its convolution kernel
	\[
	\Lc_\omega^{-\alpha''+i\gamma}\delta_0 = \frac{1}{\Gamma(\alpha''+i\gamma)} \int_0^\infty t^{\alpha''+i\gamma} \ee^{-t \omega} h_t\,\frac{\dd t}{t}
	\]
is in $L^{1}$, since
	\[
	\int_0^\infty t^{\alpha''} \ee^{-t \omega}\norm{h_t \Delta_R^{1/p'}}_{L^1(\beta)}\,\frac{\dd t}{t}\meg C_1^2 \int_0^\infty t^{\alpha''} \ee^{-t(\omega_0-2 \omega_1)}\,\frac{\dd t}{t},
	\]
	which is finite. 
	Hence, our assumptions on $\omega_0$ are compatible with those of the first step.
	In addition, by Stirling's formula,  
	\[
	1/\abs{\Gamma(\alpha''+i \gamma)}\sim \abs{ \alpha'' +i\gamma}^{1/2-\alpha''}/\sqrt{2\pi} \ee^{\alpha''+ \gamma \arg(\alpha''+i \gamma)}
	\]
	for $\gamma\to \infty$, so that the quantities 
	\[
	(1+\abs{\gamma})^{-1/2} \ee^{-\pi\abs{\gamma}/2}/\abs{\Gamma(\alpha''+i \gamma)}, \qquad \gamma\in \R,
	\]
	are uniformly bounded for every (fixed) $\alpha''>0$.
	As in the first step, in the remainder of the proof we shall reduce to the case that   $\omega=0$.  We prove only the second assertion.
	
By Theorem~\ref{teo:2}, it suffice  to show that
\begin{multline*}
	\abs{\Gamma((\alpha'-\alpha)/\dL+i \gamma)}\big\| \norm{t^{-\alpha'/\dL} (\Lc^{(\alpha-\alpha')/\dL} W^{(m)}_t f )(x)}_{L^q_t(\mi_1)}\big\| _{L^p_x(\beta)}\\
	\leq C \Big(\norm{\ee^{-(1/2)\Lc}f}_{L^p(\beta)}+\big\|\norm{t^{-\alpha/\dL} (W^{(m)}_t f)(x) }_{L^q_t(\mi_1)}\big\|_{L^p_x(\beta)}\Big)
\end{multline*}
for some $C>0$. Notice that $  \Lc^{(\alpha-\alpha')/\dL} W^{(m)}_t f$ is well defined since $  W^{(m)}_t f\in L^p(\beta)$ for every $t>0$. Observe that
	\[
	\abs{\Gamma((\alpha'-\alpha)/\dL+i \gamma)}\abs{ \Lc^{(\alpha-\alpha')/\dL}W^{(m)}_t f }\meg   \int_0^\infty t^m s^{(\alpha'-\alpha)/\dL} \abs{\ee^{-(t+s)\Lc}\Lc^m f  }\,  \frac{\dd s}{s},
	\]
	so that by means of Lemma~\ref{lem:25bis}, we see that there is a constant $C_2>0$ such that
	\[
	\abs{\Gamma((\alpha'-\alpha)/\dL+i \gamma)}\norm{ t^{-\alpha'/\dL}\Lc^{(\alpha-\alpha')/\dL}W^{(m)}_t f }_{L^q_t(\mi_1)}\meg \norm{ t^{-\alpha/\dL} W^{(m)}_t f }_{L^q_t(\mi)},
	\]
	where $\mi$ is the Haar measure on $(0,+\infty)$ such that $\dd \mi(t)=\frac{\dd t}{t}$. Finally, observe that  
	\[
	\begin{split}
	\norm{ t^{-\alpha/\dL} W^{(m)}_t f }_{L^q_t((1,+\infty),\mi)} \meg C_1 \norm{t^{-\alpha/\dL} \ee^{\omega_1 (t-1/2)}T_{b,t-1/2} \ee^{-(1/2)\Lc}f }_{L^q_t((1,+\infty),\mi)}.
	\end{split}
	\]
	We may then conclude by Lemma~\ref{lem:2} (notice that $\omega_1<0$ since we are assuming $\omega=0$).
\end{proof}

\begin{cor}
	There is $\omega_0\in \R$ such that the following hold. Suppose $p,q\in [1,\infty]$, $\alpha\in \R$, $m\in\N$ with $m>\alpha/\dL$, and $\mi\in \Mc_\Samp$. Then there is a constant $C>0$ such that
	\begin{enumerate}
		\item[\textnormal{(i)}] for every $f\in \Sc'(G)$  
				\[
		 \norm{f}_{B^{p,q}_\alpha(\beta)} \asymp_{C} \big\|   t^{-\alpha/\dL} \norm{ (t\Lc_\omega)^m \ee^{-t \Lc_\omega} f }_{L^p(\beta)}  \big\|_{L^q_t(\mi)};
		\]
%		\[
%		\frac{1}{C} \norm{f}_{B^{p,q}_\alpha(\beta)}\meg  \big\|   t^{-\alpha/\dL} \norm{ (t\Lc_\omega)^m \ee^{-t \Lc_\omega} f }_{L^p(\beta)}  \big\|_{L^q_t(\mi)}\meg C \norm{f}_{B^{p,q}_\alpha(\beta)}.
%		\]
		
		\item[\textnormal{(ii)}] if $p\in ( 1,\infty)$, then for every $f\in \Sc'(G)$
		\[
		 \norm{f}_{F^{p,q}_\alpha(\beta)} \asymp_{C}  \big\|  \norm{   t^{-\alpha/\dL}  (t\Lc_\omega)^m \ee^{-t \Lc_\omega} f  (x)}_{L^q_t(\mi)}\big\|_{L^p_x(\beta)}.
		\] 
%		\[
%		\frac{1}{C} \norm{f}_{F^{p,q}_\alpha(\beta)}\meg  \big\|  \norm{   t^{-\alpha/\dL}  (t\Lc_\omega)^m \ee^{-t \Lc_\omega} f  (x)}_{L^q_t(\mi)}\big\|_{L^p_x(\beta)} \meg C \norm{f}_{F^{p,q}_\alpha(\beta)}.
%		\] 
	\end{enumerate}
\end{cor}

\begin{proof}
	(i) We take $\omega_0$ so that $\Lc_\omega^{-1}$ induces an endomorphism of $L^p(\beta)$ for every $p\in [1,\infty]$ (Proposition~\ref{Smulders}) 
	Using Lemma~\ref{lem:26} with $\mi=\delta_1$ and $\nu=\mi$, we see that there is a constant $C_1>0$ such that
	\[
	\norm{\Lc_\omega^m \ee^{-  \Lc_\omega} f }_{L^p(\beta)}\meg C_1 \big\| t^{-\alpha/\dL} \norm{ (t\Lc_\omega)^m \ee^{-t \Lc_\omega} f }_{L^p(\beta)}  \big\|_{L^q_t(\mi)}.
	\]
	The assertion follows by means of Theorem~\ref{teo:2}. The proof of (ii) is analogous.
\end{proof}

\section{Embeddings}\label{emb:sec}

\begin{prop}\label{prop:1}
If $p\in (1,\infty)$, $q\in [1,\infty]$ and $\alpha\in \R$, then the following continuous embeddings hold:
	\[
	B^{p,\min(p,q)}_\alpha (\beta) \subseteq F^{p,q}_\alpha(\beta) \subseteq B^{p,\max(p,q)}_\alpha(\beta).
	\]
\end{prop}

\begin{proof}
	Take $m\in \N$, $m>\alpha/\dL$. Then, by Minkowski's integral inequality and the (contractive) inclusion $\ell^{q/p}(\N)\subseteq\ell^{\max(1,q/p)}(\N)$,
	\[
	\begin{split}
	\big\| 2^{j\alpha/\dL} \norm{W^{(m)}_{2^{-j}}f}_{L^{p}(\beta)}  \big\|_{\ell_j^{\max(p,q)}(\N)}&=\bigg\| 2^{pj\alpha/\dL} \int_G\abs{W^{(m)}_{2^{-j}} f}^p\,\dd \beta  \bigg\|_{\ell_j^{\max(1,q/p)}(\N)}^{1/p}\\
		&\meg \left( \int_G \norm{2^{pj\alpha/\dL} \abs{W^{(m)}_{2^{-j}} f}^p }_{\ell_j^{\max(1,q/p)}(\N)}\,\dd \beta\right) ^{1/p}		\\
		&\meg  \left( \int_G \norm{2^{pj\alpha/\dL} \abs{W^{(m)}_{2^{-j}} f}^p }_{\ell_j^{q/p}(\N)}\,\dd \beta\right) ^{1/p}	\\
		&= \big\| \norm{2^{j\alpha/\dL} W^{(m)}_{2^{-j}}f }_{\ell_j^q(\N)} \big\|_{L^p(\beta)}
	\end{split}
	\]
	so that $F^{p,q}_\alpha(\beta) \subseteq B^{p,\max(p,q)}_\alpha(\beta)$ continuously by Theorem~\ref{teo:2}.
	In addition, by the (contractive) inclusion $\ell^{\min(p,q)}(\N)\subseteq \ell^q(\N)$ and Minkowski's integral inequality,
	\[
	\begin{split}
		 \big\| \norm{2^{j\alpha/\dL} (W^{(m)}_{2^{-j}} f)(x)}_{\ell_j^q(\N)} \big\|_{L^p_x(\beta)}&\meg \big\| \norm{2^{j\alpha/\dL} (W^{(m)}_{2^{-j}} f)(x)}_{\ell_j^{\min(p,q)}(\N)} \big\|_{L^p_x(\beta)}\\
		 	&=\bigg\| \sum_{j\in \N}(2^{j\alpha/\dL} W^{(m)}_{2^{-j}}f)^{\min(p,q)}  \bigg\|_{L^{p/\min(p,q)}(\beta)}^{1/\min(p,q)}\\
		 	&\meg\bigg(\sum_{j\in \N}  \norm*{ 2^{j\alpha/\dL} W^{(m)}_{2^{-j}}f }^{\min(p,q)}_{L^{p}(\beta)}\bigg)^{1/\min(p,q)}\\
		 	&=\big\| 2^{j\alpha/\dL} \norm{W^{(m)}_{2^{-j}}f}_{L^p(\beta)}  \big\|_{\ell_j^{\min(p,q)}(\N)},
	\end{split}
	\]
	so that $B^{p,\min(p,q)}_\alpha (\beta) \subseteq F^{p,q}_\alpha(\beta)$  continuously by Theorem~\ref{teo:2} again.
\end{proof}

The next result is the analogue of~\cite[Theorem 5.2 (iii)]{BPV}. Combining it with Theorem~\ref{teo:5}, it implies that $F^{p,2}_\alpha(\beta)$ may be identified with the  Sobolev space   
\begin{equation}\label{Sobolev}
L^{p,\alpha}(\beta) := \Lc_\omega^{-\alpha/\dL} L^p(\beta),
\end{equation}
  provided that $\omega$ is sufficiently large (independently of $\alpha$); cf.~\cite{BPTV}.

\begin{teo}\label{teo:10}
	For every $p\in(1,\infty)$, $F^{p,2}_0(\beta)=L^p(\beta)$ with equivalence of norms. 
	\end{teo}

\begin{proof}
	We may assume that $\Lc$ is formally self-adjoint and positive, thanks to Theorem~\ref{teo:2} and Proposition~\ref{prop:4}.
	
	Define 
	\[
	\Lambda(\phi)\coloneqq \Set{\ee^z\colon z\in \C, \abs{\Im z}<\phi}, \qquad \phi\in (0,\pi].
	\]
	By~\cite[Theorem 1.1.IV]{ElstRobinson}, there is $\theta\in
        (0,\pi/2)$ such that $\Lc$ generates a holomorphic semigroup
        in $L^q(\beta)$ on $\Lambda(\theta)$ for every $q\in
        (1,\infty)$. In addition, by Proposition~\ref{Smulders} there is  $\omega_0>0$ such that $\Lc_\omega$ has a bounded $H^\infty$ functional calculus in $L^p(\beta)$ on $\Lambda(\phi)$ for every $\omega\Meg \omega_0$ and $\phi\in (\pi/2-\theta, \pi]$. 
	Observe that we may assume that $\omega_0$ is so large that there are $b,C'>0$ such that 
\begin{equation}\label{Lcj1}
  \abs{\Lc_{\omega_0}^j h_t}\meg C' t^{-j} p_{b,t}\ee^{-2c t}, \qquad t>0, \quad j=0,1,
\end{equation}
	where $c>0$ is such that 
\begin{equation}\label{Lcj2}
	c'\norm{p_{b,t}\ee^{c'\abs{\,\cdot\,}_*}}_{L^1(\beta)}\meg \ee^{c t},
\end{equation}
	with $c'>1$ such that $\Delta_R \meg
        c'\ee^{c'\abs{\,\cdot\,}_*}$ (cf.~Lemma~\ref{lem:3}). By
        Proposition~\ref{lem:8bis}, up to replacing $\Lc$ with
        $\Lc_{\omega_0}$, we may assume that $\omega_0=0$.
	
	By~\cite[Theorem 1]{Cowling}, for every $q\in (1,\infty)$ there is a constant $C_q>0$ such that
	\[
	\big\| \norm{t (   \Lc\ee^{-t \Lc} f )(x) }_{L^2_t(\mi)}\big\|_{L^q_x(\beta)}\meg C_q \norm{f}_{L^q(\beta)}
	\]
	for every $f\in L^q(\beta)$, where $\mi$ denotes the Haar measure on $(0,+\infty)$ with $\dd \mi(t)=\frac{\dd t}{t}$.
	Consequently,
	\[
	\begin{split}
		C_q \norm{f}_{L^q(\beta)}\Meg\big\| \norm{t (  \Lc \ee^{-t \Lc} f )(x) }_{L^2_t(\mi)}\big\|_{L^q_x(\beta)}\Meg  \big\| \norm{t (   \Lc \ee^{-t \Lc } f )(x) \big\|_{L^2_t(\mi_1)}}_{L^q_x(\beta)},
	\end{split}
	\]
	so  that $F^{q,2}_0(\beta)\subseteq L^p(\beta)$ continuously.
	
	Now, observe that, for every $f\in \Sc(G)$, 
	\[
	\int_0^\infty t    \Lc \ee^{-t \Lc } f  \,\dd \mi(t)   =f
	\]
	as one sees by means of the functional calculus for $  \Lc$  (or thanks to Lemma~\ref{lem:9}). Consequently, for every $f,g\in \Sc(G)$,
	\[
	\int_G\int_{0}^\infty t^2 (  \Lc  \ee^{-t \Lc   } f )(x) (   \Lc  \ee^{-t \Lc } \overline g)(x) \,\dd \mi(t)\,\dd \beta(x)= \int_G f\overline g\,\dd \beta.
	\]
	Therefore,
	\[
	\begin{split}
		\norm{f}_{L^p(\beta)}&= \sup_{\norm{g}_{L^{p'}(\beta)}\meg 1} \abs*{\int_G f\overline g\,\dd \beta}\\
		&\meg   \sup_{\norm{g}_{L^{p'}(\beta)}\meg 1}  \abs*{\int_G\int_{0}^\infty t^2 ( \Lc\ee^{-t \Lc} f )(x) (  \Lc \ee^{-t \Lc} \overline g)(x) \,\dd \mi(t)\,\dd \beta(x)}\\
		&\meg \sup_{\norm{g}_{L^{p'}(\beta)}\meg 1}\big\| \norm{t (  \Lc \ee^{-t \Lc} f )(x) }_{L^2_t(\mi)}\big\|_{L^p_x(\beta)}\big\|\norm{t ( \Lc \ee^{-t \Lc} g )(x) }_{L^2_t(\mi)}\big\|_{L^{p'}_x(\beta)}\\
		&\meg C_{p'} \big\| \norm{t ( ( \Lc \ee^{-t \Lc} f )(x) }_{L^2_t(\mi)}\big\|_{L^p_x(\beta)}.
	\end{split}
	\]
	Now, observe that
	\[
	\begin{split}
		\big\| \norm{t ( ( \Lc \ee^{-t \Lc} f )(x) }_{L^2_t(\mi)}\big\|_{L^p_x(\beta)}&\meg \big\| \norm{t ( ( \Lc \ee^{-t \Lc} f )(x) }_{L^2_t(\mi_1)}\big\|_{L^p_x(\beta)}\\
		&\qquad +\big\|\norm{t ( ( \Lc \ee^{-t \Lc} f )(x) }_{L^2_t((1,+\infty),\mi)}\big\|_{L^p_x(\beta)}.
	\end{split}
	\]
	The first term may be controlled by means of $\norm{f}_{F^{p,2}_{\alpha}(\beta)}$. In order to deal with the second term, by~\eqref{Lcj1} and~\eqref{Lcj2} there is a constant $C''>0$ such that 
	\[
	t\norm{    \Lc \ee^{-t \Lc } f }_{L^p(\beta)}\meg C''  \ee^{-c t}\norm{\ee^{-(1/2)\Lc}f}_{L^p(\beta)}
	\]
	for every $t\Meg 1 $, whence the result by means of Theorem~\ref{teo:2}.
\end{proof}
The next result should be compared with~\cite[Theorems 5.1 and 5.2 (i), (ii), (iv)]{BPV}.

\begin{prop}\label{prop:2}
Suppose $p_1,p_2,q_1,q_2\in [1,\infty]$ and $\alpha_1,\alpha_2\in \R$. Then, the following continuous embeddings hold.
	\begin{enumerate}
		\item[\textnormal{(1$\:$)}] $B^{p_1,q_1}_{\alpha_1}(\beta)\subseteq B^{p_1,q_2}_{\alpha_2}(\beta)$ if either $\alpha_2<\alpha_1$, or $\alpha_2=\alpha_1$ and $ q_1\meg q_2$.
		
		\item[\textnormal{(1$'$)}] $F^{p_1,q_1}_{\alpha_1}(\beta)\subseteq F^{p_1,q_2}_{\alpha_2}(\beta)$ if $p_1\in (1,\infty)$ and either $\alpha_2<\alpha_1$ or $\alpha_2=\alpha_1$ and $ q_1\meg q_2$.
		
		\item[\textnormal{(2$\:$)}] $\Delta_L^{1/p_1-1/p_2} B^{p_1,q_1}_{\alpha_1}(\beta)\subseteq B^{p_2,q_1}_{\alpha_2}(\beta)$ if $p_1\meg p_2$ and $\alpha_1-\frac{Q_*}{p_1}\Meg\alpha_2-\frac{Q_*}{p_2} $.
		
		\item[\textnormal{(2$'$)}] $\Delta_L^{1/p_1-1/p_2} F^{p_1,q_1}_{\alpha_1}(\beta  )\subseteq F^{p_2,q_2}_{\alpha_2}(\beta )$  if $1<p_1< p_2<\infty$ and $\alpha_1-\frac{Q_*}{p_1}\Meg\alpha_2-\frac{Q_*}{p_2} $.
		\item[\textnormal{(3$\:$)}] $B^{p_1,1}_{0}(\beta)\subseteq L^{p_1}(\beta)$, and if $p_1<p_2$ and $\alpha_1> \frac{Q_*}{p_1}-\frac{Q_*}{p_2} $, then $\Delta_L^{1/p_1-1/p_2}B^{p_1,q_1}_{\alpha_1}(\beta )\subseteq   L^{p_2}(\beta )$.
		
		\item[\textnormal{(3$'$)}]  	$\Delta_L^{1/p_1-1/p_2}F^{p_1,q_1}_{\alpha_1}(\beta )\subseteq L^{p_2}(\beta )$ if $1<p_1< p_2$ and $\alpha_1\Meg  \frac{Q_*}{p_1}-\frac{Q_*}{p_2} $, and either $p_2<\infty$ or $\alpha_1>  \frac{Q_*}{p_1}-\frac{Q_*}{p_2} $.
	\end{enumerate}
\end{prop}

\begin{proof}
	Using Proposition~\ref{prop:5}, we may reduce to the case $\beta=\beta_L$, whence $\Delta_{L}=1$.
	
	(1) The case $\alpha_2=\alpha_1$ and $q_2\meg q_1$ follows from Theorem~\ref{teo:2} (choosing $\mi= \sum_{j\in\N} \delta_{2^{-j}}$) and the inclusion $\ell^{q_2}(\mi)\subseteq \ell^{q_1}(\mi)$. The case  $\alpha_2<\alpha_1$ follows from    H\"older's inequality when $q_1\meg q_2$, whence the result by the previous case.
	
	(1$'$) This follows from (1) and Proposition~\ref{prop:1} when $\alpha_2<\alpha_1$. The case $\alpha_2=\alpha_1$ and $q_2\meg q_1$ follows from  Theorem~\ref{teo:2} (choosing $\mi= \sum_{j\in\N} \delta_{2^{-j}}$) and the inclusion $\ell^{q_2}(\mi)\subseteq \ell^{q_1}(\mi)$.
	
	(2) Observe  that, by Corollary~\ref{cor:1}, there is a constant $C_1>0$ such that
	\[
	\norm{\ee^{-t\Lc} f}_{L^{p_2}(\beta)}  \meg C_1 t^{Q_*(1/p_2-1/p_1)}\norm{ \ee^{-(t/2)\Lc} f}_{L^{p_1}(\beta)}
	\]
	for every $f\in \Sc'(G)$ and $t\in (0,1]$ (say). The assertion
        follows by means of (1) and Definition~\ref{BTL:def}.

	(2$'$)  By (1$'$), it will suffice to show that $ F^{p_1,\infty}_{\alpha_1}(\beta)\subseteq F^{p_2,1}_{\alpha_2}(\beta)$ continuously when $\alpha_2= \alpha_1+\frac{Q_*}{p_2}-\frac{Q_*}{p_1}$. Take $m\in\N$ with $m>\alpha_1/\dL$ and $f\in   F^{p_1,\infty}_{\alpha_1}(\beta)$. We shall  set $\Tc_j^k \coloneqq 2^{j\alpha_k/\dL} W^{(m)}_{2^{-j}}$ for $j\in \N$ and $k=1,2$. It will then suffice to show that there is a constant $C_2>0$ such that
	\[
	\bigg\| \sum_{j\in \N} \abs{\Tc^2_j f}\bigg\|_{L^{p_2}(\beta)}\meg C_2\Big\| \sup_{j\in \N} \abs{\Tc^1_j f}\Big\|_{L^{p_1}(\beta)}
	\]
	(cf.~Corollary~\ref{cor:1}). We may reduce to the case in which $ \|\sup_{j\in \N} \abs{\Tc^1_j f}\|_{L^{p_1}(\beta)}=1$.
	Observe that, by Corollary~\ref{cor:1} there is a constant $C_3>0$ such that
	\[
	\begin{split}
		\norm{\Tc^2_j f}_{L^\infty(\beta)}&=2^{-j(\alpha_1-\alpha_2)/\dL} \norm{2^{-j(m-\alpha_1/\dL)} \ee^{-2^{-j-1}\Lc} \ee^{-2^{-j-1}\Lc} \Lc^m f}_{L^\infty(\beta)}\\
			&\meg C_3 2^{-j(\alpha_1-\alpha_2)/\dL+j Q_*/(\dL p_1)} \norm{\Tc^1_{j+1}f}_{L^{p_1}(\beta)}\\
			&=  C_3 2^{ j Q_*/(\dL p_2)} \norm{\Tc^1_{j+1}f}_{L^{p_1}(\beta)}.
	\end{split}
	\]
	Then, there is a constant $C_4>0$ such that, for every $N\in\N$,
	\[
	\begin{split}
	\sum_{j=0}^N \abs{\Tc_j^2 f}&\meg C_4 \sum_{j=0}^N 2^{ j Q_*/(\dL p_2)} \Big\|\sup_{j\in \N} \abs{\Tc^1_j f}\Big\|_{L^{p_1}(\beta)}=C_4 \sum_{j=0}^N 2^{j Q_*/(\dL p_2)} 
	\end{split}
	\]
	and 
	\[
	\sum_{j\Meg N}  \abs{\Tc_j^2 f}=\sum_{j\Meg N} 2^{-j(\alpha_1-\alpha_2)/\dL} \abs{\Tc_j^1 f}\meg C_4 2^{-N(\alpha_1-\alpha_2)/\dL} \sup_{j\in \N} \abs{\Tc^1_j f}.
	\]
Since
	\[
	\Big\|\sum_{j\in \N} \abs{\Tc^2_j f} \Big\|^{p_2}_{L^{p_2}(\beta)}=p_2 \int_0^\infty t^{p_2} \beta \bigg( \Big\{x\in G\colon \sum_{j\in \N} \abs{(\Tc^2_j f)(x)}>t  \Big\} \bigg)\,\frac{\dd t}{t},
	\]
one has
	\[
	\Big\{x\in G\colon \sum_{j\in \N} \abs{(\Tc^2_j f)(x)}>t  \Big\}\subseteq \Big\{x\in G\colon \sup_{j\in \N} \abs{(\Tc^1_j f)(x)}>C_4^{-1}t  \Big\},
	\]
	so that
	\[
	\begin{split}
		\int_0^1 t^{p_2} \beta \bigg( &\Big\{x\in G\colon \sum_{j\in \N} \abs{(\Tc^2_j f)(x)}>t  \Big\}\bigg) \, \frac{\dd t}{t}\\
		&\meg \int_0^\infty t^{p_1} \beta \bigg( \Big\{x\in G\colon \sup_{j\in \N} \abs{(\Tc^1_j f)(x)}>C_4^{-1}t  \Big\}\bigg)\,\frac{\dd t}{t}\meg\frac{C_4^{p_1}}{p_1} .
	\end{split}
	\]
For $t\Meg 1$, let now $N(t)$ be the largest integer $N$ such that
	\[
	\sum_{j=0}^N 2^{j Q_*/(\dL p_2)} <\frac{t}{2 C_4 }.
	\] 
	Observe that there is a constant $c>1$ such that $t/c\meg 2^{N(t) Q_*/(\dL p_2)} \meg c t$. Moreover,
	\[
	\begin{split}
		\bigg\{ x\in G\colon \sum_{j\in \N} \abs{(\Tc^2_j f)(x)}>t \bigg\}&\subseteq \bigg\{x\in G\colon \sum_{j>N(t)} \abs{(\Tc^2_j f)(x)}>t/2 \bigg\}\\
			&\subseteq \bigg\{x\in G\colon \sup_{j\in \N} \abs{(\Tc^1_j f)(x)} > \frac{t}{2C_4} 2^{N(t)(\alpha_1-\alpha_2)/\dL}  \bigg\}
	\end{split}
	\]
	and 
	\[
	t 2^{N(t)(\alpha_1-\alpha_2 )/\dL}\Meg c^{(\alpha_2-\alpha_1)p_2/Q_*} t\, t^{  (\alpha_1-\alpha_2) p_2/Q_* }=c^{1-p_2/p_1 }t^{p_2/p_1},
	\]
	so that 
	\[
	\begin{split}
			&\int_1^\infty t^{p_2} \beta \bigg( \bigg\{x\in G\colon \sum_{j\in \N} \abs{(\Tc^2_j f)(x)}>t  \bigg\}\bigg) \, \frac{\dd t}{t}\\
				&\qquad\meg \int_1^\infty t^{p_2} \beta \left( \Set{x\in G\colon \sup_{j\in \N} \abs{(\Tc^1_j f)(x)}>\frac{ c^{1-p_2/p_1}}{2 C_4}t^{p_2/p_1}  }\right)\,\frac{\dd t}{t}\\
			&\qquad\meg \frac{p_1}{p_2}\int_0^\infty t^{p_1} \beta \left( \Set{x\in G\colon \sup_{j\in \N} \abs{(\Tc^1_j f)(x)}>\frac{  c^{1-p_2/p_1 }}{2 C_4} t   }\right)\,\frac{\dd t}{t} =\frac{(2C_4)^{p_1} }{p_2  c^{p_1-p_2}}.
	\end{split}
	\]
	The assertion follows.
	
	(3)   By~(2), we may reduce to proving the first assertion. But this follows since by Lemma~\ref{lem:9},
	\[
	f=\ee^{-\Lc} f+ \int_0^1 \Lc \ee^{-t \Lc} f\,\dd t,
	\]
	so that for $f\in \Sc'(G)$
	\[
	\begin{split}
		\norm{f}_{L^{p_1}(\beta)}\meg \norm{\ee^{-\Lc}f}_{L^{p_1}(\beta)}+ \int_0^1 \norm{W^{(1)}_t f}_{L^{p_1}(\beta)}\,\frac{\dd t}{t}= \norm{f}_{B^{p_1,1}_0(\beta)}.
	\end{split}
	\]
	(3$'$) This follows from (2$'$) and Theorem~\ref{teo:10} when $p_2<\infty$. The case $p_2=\infty$ follows from (3) and Proposition~\ref{prop:1}.
\end{proof}

\section{Duality}\label{dual:sec}

In order to deal with duality we shall make use of the following auxiliary ``Goodman-type'' Sobolev spaces.

\begin{deff}
Suppose  $\alpha\Meg0$ and $p\in [1,\infty]$. We define $L^p_0(\beta)$ as the closure of $C_c(G)$ in $L^p(\beta)$. Then, we define the Sobolev space $W^{\alpha,p}(\beta)$ (resp.\ $W^{\alpha,p}_0(\beta)$) as the space of $f\in L^p(\beta)$ (resp.\ $f\in L^p_0(\beta)$) such that $X f\in L^p(\beta)$ (resp.\ $X f\in L^p_0(\beta)$) for every $X\in U_\alpha$, endowed with the norm 
\[
f\mapsto \max_{\substack{X\in U_\alpha, \;  \abs{X}\meg 1}} \norm{X f}_{L^p(\beta)}.
\]
\end{deff}
Observe that by definition $L^p_0(\beta)= L^p(\beta)$ if $p<\infty$,
while $L^\infty_0(\beta) =  C_0(G)$. By means of Theorems~\ref{teo:10} and~\ref{teo:3}, moreover, recalling also~\eqref{Sobolev}, it is readily seen that 
\[
L^{p,\alpha}(\beta) = F^{p,2}_\alpha(\beta)=W^{\alpha,p}(\beta) \qquad p\in (1,\infty), \: \: \alpha\in \dd \N.
\]
It is important to observe, though, that there are in general no
further exact correspondences between the two families of spaces, 
cf.\ the classical Euclidean case~\cite{Triebel}.

\begin{prop}\label{prop:7}
Suppose $\alpha\Meg 0$, $\alpha'\in \R$ such that $([\alpha'/\dL]+1)_+\meg \alpha/\dL$, and $p,q\in [1,\infty]$. Then, the following hold.
	\begin{enumerate}
		\item[\textnormal{(i)}] $W^{\alpha,p}(\beta)$ is a Banach space.
		
		\item[\textnormal{(ii)}]  $W^{\alpha,p}_0(\beta)$ is the closure of $C^\infty_c(G)$  in $W^{\alpha,p}(\beta)$.
		
		\item[\textnormal{(iii)}]  $W^{\alpha,p}(\beta)\subseteq B^{p,q}_{\alpha'}(\beta)$ continuously.
		
		\item[\textnormal{(iv)}]  If $p\in (1,\infty)$, then $W^{\alpha,p}(\beta)\subseteq F^{p,q}_{\alpha'}(\beta)$ continuously.
	\end{enumerate}
	In particular, $\Sc(G)\subseteq B^{p,q}_\alpha(\beta)$, and if $p\in(1,\infty)$ then $\Sc(G)\subseteq F^{p,q}_\alpha(\beta)$, continuously.
\end{prop}

Since this result is purely instrumental, we claim no optimality in (iii)--(iv).

\begin{proof}
	(i) This follows from the usual arguments.
	
	(ii) It is readily seen that $W^{\alpha,p}_0(\beta)$ is a Banach space, so that it will suffice to show that $C^\infty_c(G)$ is dense in $W^{\alpha,p}_0(\beta)$. Take $f\in W^{\alpha,p}_0(\beta)$, and let $(\phi_j)$ be a sequence of positive elements of $C^\infty_c(G)$ such that $\supp\phi_j \subseteq B(e,2^{-j})$ and $\int_G \phi_j\,\dd \beta_R=1$ for every $j\in \N$. Then  $\phi_j*f$ converges to $f$ in $W^{\alpha,p}_0(\beta)$. 
	We may then reduce to the case in which $f\in C^\infty(G)$. Then, define $\psi_j\coloneqq \chi_{B(e, j+1)}*\phi_0$, so that 
	\[
	\chi_{B(e,j)}\meg \psi_j\meg \chi_{B(e,j+2)}
	\]
	for every $j\in \N$. In addition, $\sup_{j\in \N}
        \norm{X\psi_j}_{L^\infty(\beta)}$ is finite for every $X\in
        U(\mathfrak{g})$. Then $f \psi_j$ converges to $f$ in
        $W^{\alpha,p}_0(\beta)$, when $p\in[1,\infty]$ (recalling
        that 
        $f\in C^\infty$),
whence the conclusion.
	
	(iii)--(iv)	By Propositions~\ref{prop:1} and~\ref{prop:2}, it will suffice to show that $W^{\alpha,p}(\beta)\subseteq B^{p,1}_{\alpha'}(\beta)$ continuously. Set $m=([\alpha'/\dL]+1)_+$, so that $m\dL\meg \alpha$ by the assumptions. Then, take $f \in W^{\alpha,p}(\beta)$ and observe that 
	\[
	\norm{\Lc^h\ee^{-t \Lc} f}_{L^p(\beta)}\meg C\norm{\Lc^h f}_{L^p(\beta)}, \qquad h=0,\dots, m, \; t\in (0,1],
	\]
	where $C=\sup_{t\in (0,1]} \norm{h_t
          \Delta_R^{1/p'}}_{L^1(\beta)}$. Then, using
        Corollary~\ref{cor:carleson-measures}
	\[
	\norm{f}_{B^{p,1}_\alpha(\beta)}\meg \norm{f}_{L^p(\beta)}+C_1\norm{t^{m-\alpha/\dL}}_{L^1(\mi_1)} \norm{\Lc^m f}_{L^p(\beta)},
	\]
	whence the conclusion. 
\end{proof}
\begin{deff}
	Suppose $p,q\in [1,\infty]$ and $\alpha\in \R$. We define $\mathring B^{p,q}_\alpha(\beta)$ and, if $p\in(1,\infty)$, $\mathring F^{p,q}_\alpha(\beta)$ as the closure of $\Sc(G)$ in $B^{p,q}_\alpha(\beta)$ and in $F^{p,q}_\alpha(\beta)$ respectively.
\end{deff}

For the next lemma, let us denote by $L^{p,q}_0(\beta,\mi+\delta_0)$
the closure of $C_c(G\times R)$ in $L^{p,q}(\beta,\mi+\delta_0)$,
where $R$ is as in~\eqref{topological-sum:eq}; $L^{q,p}_0(\mi+\delta_0,\beta)$ is defined analogously. We recall
that $\theta_\alpha$ is defined in~\eqref{thetaalpha}, and $T_m$ in Lemma~\ref{lem:22}.
\begin{teo}\label{teo:4}
	Take $p,q\in [1,\infty]$ and $\alpha\in \R$. Then, the following hold.
	\begin{enumerate}
		\item[\textnormal{(i)}] Suppose $\mi\in \Mc_\Samp$ and $m\in\N$ with $m>\alpha/\dL$. Then $T_m$ induces an isomorphism of $\mathring B^{p,q}_\alpha(\beta)$ onto $(\theta_\alpha L^{p,q}_0(\beta,\mi+\delta_0))\cap \Fc_m$ and, if $p\in(1,\infty)$, of $\mathring F^{p,q}_\alpha(\beta)$ onto $(\theta_\alpha L^{q,p}_0(\mi+\delta_0,\beta))\cap \Fc_m$.

		\item[\textnormal{(ii)}]  The map
		\[
		B\colon (f,g)\mapsto \lim_{t\to 0^+} \langle \ee^{-t \Lc}f \vert \ee^{-t \Lc^*}g\rangle
		\]
		induces a well-defined continuous sesquilinear map on $B^{p,q}_\alpha(\beta)\times B^{p',q'}_{-\alpha}(\beta)$ and on $F^{p,q}_\alpha(\beta)\times F^{p',q'}_{-\alpha}(\beta)$ if $p \in (1,\infty)$.

		\item[\textnormal{(iii)}]   $B$ induces an antilinear  isomorphism of $B^{p',q'}_{-\alpha}(\beta)$ onto the dual of $\mathring B^{p,q}_\alpha(\beta)$, and also of $F^{p',q'}_{-\alpha}(\beta)$ onto the dual of $\mathring F^{p,q}_\alpha(\beta)$ if $p\in (1,\infty)$.
	\end{enumerate}
\end{teo}

Observe that, by means of Lemma~\ref{lem:22}, (i) shows that $\mathring B^{p,q}_\alpha(\beta)= B^{p,q}_\alpha(\beta)$ for every $p,q\in [1,\infty)$, and also that $\mathring F^{p,q}_\alpha(\beta)=  F^{p,q}_\alpha(\beta)$ for every $p\in (1,\infty)$ and  $q\in [1,\infty)$. For the classical cases, cf.~\cite[Theorem 2.11.2]{Triebel} and also~\cite[Remark 5.14 and Section 12]{FrazierJawerth} for (iii).

\begin{proof}
We begin by proving (ii), and leave the proof of (i) for the end. Take $m\in\N$ such that $m>\abs{\alpha}/\dL$, and observe that, for every $f\in F^{p,q}_\alpha(\beta)$ and $g\in F^{p',q'}_{-\alpha}(\beta)$, by Lemma~\ref{lem:9}
	\[
	\begin{split}
		&\langle \ee^{-t \Lc}f  \vert \ee^{-t \Lc^*}g\rangle\\
		&=\sum_{h=0}^{2m-1} \frac{2^h}{h!} \langle \Lc^h \ee^{-(t+2) \Lc}f\vert \ee^{-t \Lc^*}g\rangle+\frac{1}{(2m-1)!} \int_0^2 s^{2m} \langle \Lc^{2m} \ee^{-(s +t)\Lc}f\vert  \ee^{-t\Lc^*} g\rangle\,\frac{\dd s}{s}\\
		&=\sum_{h=0}^{2m-1} \frac{2^h}{h!} \langle \Lc^h \ee^{-(t+1) \Lc}f\vert \ee^{-(t+1) \Lc^*}g\rangle+\frac{4^m}{(2m-1)!} \int_0^1 s^{2m} \langle \Lc^{m} \ee^{-(s +t)\Lc}f\vert \Lc^{*m} \ee^{-(s+t)\Lc^*} g\rangle\,\frac{\dd s}{s}\\
		&=\sum_{h=0}^{2m-1} \frac{2^h}{h!} \langle \Lc^h \ee^{-(t+1) \Lc}f\vert \ee^{-(t+1) \Lc^*}g\rangle+\frac{4^m}{(2m-1)!} \int_t^{t+1} \!\!\!\!(s-t)^{2m-1} \langle \Lc^{m} \ee^{-s\Lc}f\vert  \Lc^{*m}\ee^{-s\Lc^*} g\rangle\,\dd s.
	\end{split}
	\]
Now, clearly
	\[
	\lim_{t\to 0^+} \sum_{h=0}^{2m-1} \frac{2^h}{h!} \langle \Lc^h \ee^{-(t+1) \Lc}f\vert \ee^{-(t+1) \Lc^*}g\rangle=\sum_{h=0}^{2m-1} \frac{2^h}{h!} \langle \Lc^h \ee^{-  \Lc}f\vert \ee^{- \Lc^*}g\rangle.
	\]
	In addition, 
	\[
	\begin{split}
		\int_t^{t+1} &(s-t)^{2m-1}\abs{\langle \Lc^{m} \ee^{-s\Lc} f\vert  \Lc^m\ee^{-s\Lc^*} g\rangle}\,\dd s\\
		& \qquad \meg  \int_G \int_0^{t+1} s^{2m-1} \abs{  (\Lc^{m} \ee^{-s\Lc}f)(x)   (\Lc^{*m}\ee^{-s\Lc^*} g)(x)}\,\dd s\,\dd \beta(x)\\
		&\qquad\meg  \int_G\norm{s^{-\alpha/\dL} (W^{(m)}_s f)(x) }_{L^q_s(\mi_{t+1})}\norm{s^{\alpha/\dL} (W^{*(m)}_s g)(x) }_{L^{q'}_s(\mi_{t+1})} \,\dd \beta(x)\\
		&\qquad\meg \big\| \norm{s^{-\alpha/\dL} (W^{(m)}_s f)(x) }_{L^q_s(\mi_{t+1})}\big\|_{L^p(\beta)} \big\|\norm{s^{\alpha/\dL} (W^{*(m)}_s g)(x) }_{L^{q'}_s(\mi_{t+1})}\big\|_{L^{p'}(\beta)},
	\end{split}
	\]
	where $W^{*(m)}_s= (s \Lc^*)^m \ee^{-s \Lc^*}$.
	It then follows that $B(f,g)$ is well defined and by dominated convergence
	\[
	B(f,g)=\sum_{h=0}^{2m-1} \frac{2^h}{h!} \langle \Lc^h \ee^{- \Lc}f\vert \ee^{-  \Lc^*}g\rangle+\frac{4^m}{(2m-1)!} \int_0^{1} s^{2m-1} \langle \Lc^{m} \ee^{-s\Lc}f\vert  \Lc^{*m}\ee^{-s\Lc^*} g\rangle\,\dd s.
	\]
	The proof for Besov spaces is similar.
	
	\smallskip
	
We now prove (iii).  Pick $\nu\coloneqq \delta_0+ \sum\nolimits_{j\in
  \N} \delta_{2^{-j}}$ 	and $L\in B^{p,q}_\alpha(\beta)'$.  Assume
first that $p<\infty$. Observe that, by~\cite[Theorem 1 of \S\
3]{BenedekPanzone},\footnote{The cited reference only deals with the
  case $p,q<\infty$. The case $q=\infty$ may be dealt with in a
  similar way since we chose $\nu$ as a discrete measure.} we may
identify the dual of $\theta_\alpha L^{p,q}_0(\beta,\nu)$   with
$\theta_{-\alpha} L^{p',q'}(\beta,\nu)$, so that by Lemma~\ref{lem:24}
and the Hahn--Banach theorem\footnote{Indeed, notice that 
  $\theta_\alpha L^{p,q}_0(\beta,\mu+\delta_0) \cap \Fc_m \ni \varphi \mapsto
  L(f)\in \C $ is bounded, where $\varphi=T_m f$, so that it extends to a continuous linear functional on $\theta_\alpha L^{p,q}_0(\beta,\mu+\delta_0)$.}
  there is   $g\in \theta_{-\alpha} L^{p',q'}( \beta,\nu)$ such that for every $f\in B^{p,q}_\alpha(\beta)$
	\[
	L( f)=\langle T_m f\vert g \rangle.
	\]
We set $g_t\coloneqq g(t,\,\cdot\,)$ for every $t\Meg 0$ to simplify the notation. Then, for  $f\in \Sc(G)$,
	\[
	\begin{split}
		L(f)& =  \langle T_m f\vert g \rangle\\
		&=  \langle \ee^{-\Lc} f\vert g_0\rangle + \sum_{j\in\N}\langle (2^{-j}\Lc)^m \ee^{-2^{-j}\Lc} f\vert g_{2^{-j}}\rangle   \\
		&=\langle  f\vert \ee^{-\Lc^*} g_0\rangle + \sum_{j\in \N}\langle  f\vert (2^{-j}\Lc^*)^m \ee^{-2^{-j}\Lc^*}g_{2^{-j}}\rangle  =   \langle  f\vert \widetilde g \rangle,
	\end{split}
	\]
	where
	\[
	\widetilde g\coloneqq \ee^{-\Lc^*} g_0+\sum_{j\in \N} (2^{-j}\Lc^*)^m \ee^{-2^{-j}\Lc^*}g_{2^{-j}}   .
	\]
The above computations show that the above sum defines an element of $\Sc'(G)$. We shall prove that $\widetilde g\in B^{p',q'}_{-\alpha}(\beta)$. 

To this aim, observe that   $\ee^{-\Lc^*} g_0\in W^{m,p'}(\beta)\subseteq B^{p',q'}_{-\alpha}(\beta)$ by Corollary~\ref{cor:1} and Proposition~\ref{prop:7}. Next, by Theorem~\ref{teo:7}, take $b,C_1>0$ such that $\abs{\Lc^{*k} h'_t}\meg C_1 t^{-k} p_{t,b}$ for all $t\in (0,2]$ and $k=0,\dots, 2m$, where $(h'_t)$ denotes the heat kernel of $\Lc^*$. 	
	Then,  
	\[
	\begin{split}
		\abs{W^{*(m)}_{s} (t\Lc^*)^m \ee^{-t\Lc^*}g_t}&\meg C_1 (s t)^m (t+s)^{-2m}  T_{b,s+t} g_t
	\end{split}
	\]
	for every $s,t\in (0,1]$. Consequently, given $b'< 2^{-1/(\dL-1)}b$, Lemma~\ref{lem:4}~(2) shows that there is a constant $C_2>0$ such that
	\[
	\abs{W^{*(m)}_{s} (t\Lc^*)^m \ee^{-t\Lc^*}g_t}\meg C_2 (s t)^m (t+s)^{-2m}  T_{b',s} T_{b',t} g_t
	\]
	for every $s,t\in (0,1]$. Then, by Lemmas~\ref{lem:2} and~\ref{lem:25} there is a constant $C_3>0$ such that
	\[
	\begin{split}
		&\bigg\| s^{\alpha/\dL}\sum_j  \norm{W^{*(m)}_{s} (2^{-j}\Lc^*)^m \ee^{-2^{-j}\Lc^*}g_{2^{-j}}}_{L^{p'}(\beta)}   \bigg\|_{L^{q'}_s(\mi_1)}\\
		&\qquad\meg C_3  \bigg\|   \sum_j \frac{s^{m+\alpha/\dL}2^{-j m} }{(2^{-j}+s)^{2m}}   \norm{g_{2^{-j}}}_{L^{p'}(\beta)}  \bigg\|_{L^{q'}_s(\mi_1)}   \meg C_3^2 \big\|  t^{\alpha/\dL} \norm{g_t}_{L^{p'}(\beta)}   \big\|_{L^{q'}_t((0,+\infty),\nu)} 
	\end{split}
	\]
	which is finite.
	Analogously, 
	\[
	\begin{split}
		\bigg\|  \sum_j  (\ee^{-\Lc^*} (2^{-j}\Lc^*)^m \ee^{-2^{-j}\Lc^*}g_{2^{-j}})(x) &  \bigg\|_{L^{p'}_x(\beta)}
		\meg C_1 2^{Q_*/\dL} \bigg\| \sum_j 2^{-jm}  (T_{b,2} g_{2^{-j}})(x)   \bigg\|_{L^{p'}_x(\beta)}\\
		&\meg C_1 2^{Q_*/\dL} \norm{t^{m-\alpha/\dL}}_{L^q_t(\nu)}\big\|  \norm{ (T_{b,2} g_t)(x)}_{L^{q'}_t(\nu)} \big\|_{L^{p'}_x(\beta)},
	\end{split}
	\]
	which is finite thanks to Lemma~\ref{lem:2} again. By
        Theorem~\ref{teo:2} then $\widetilde g \in B^{p',q'}_{-\alpha}(\beta)$.
	Then, by Lemma~\ref{lem:9}, 
	\[
	L(f)= \langle f\vert \widetilde g\rangle=\lim_{t\to 0^+} \langle \ee^{-2 t\Lc}f\vert \widetilde g \rangle= \lim_{t\to 0^+} \langle \ee^{-t \Lc} f\vert \ee^{-t \Lc^*}\widetilde g \rangle= B(f,\widetilde g) 
	\]
	for every $f\in \Sc(G)$. Thus, the canonical antilinear map
        $B^{p',q'}_{-\alpha}(\beta)\to \mathring
        B^{p,q}_{\alpha}(\beta)'$ is onto, (since
       $\mathring B^{p,q}_{\alpha}(\beta)$ is
          the closure of $\Sc(G)$ in the $B^{p,q}_{\alpha}(\beta)$ norm). The fact that it is also one-to-one follows from the density of $\Sc(G)$ in  $\mathring B^{p,q}_{\alpha}(\beta)$ and from the fact that $B(f,g)=\langle f\vert g \rangle$ for every $f\in \Sc(G)$ and $g\in B^{p',q'}_{-\alpha}(\beta)$ (argue as above).
	
	The case $p=\infty$ requires a minor modification of the above argument, since in this case the dual of $\theta_\alpha L^{p,q}_0(\nu,\beta)$   \emph{cannot} be identified with $\theta_{-\alpha} L^{p',q'}(\nu,\beta)$, but rather with $\theta_{-\alpha}\ell^{q'}(\Supp{\nu}; \Mc^1(G))$, that is, the space of families $(\mi_t)_{t\in \Supp{\nu}}$ of bounded Radon measures on $G$ such that
	\[
	\bigg( \norm{\mi_0}_{\Mc^1(G)}^q+\sum_{j\in\N}  (2^{-j\alpha/\dL} \norm{\mi_{2^{-j}}}_{\Mc^1(G)})^q\bigg) ^{1/q}
	\]
	is finite, endowed with the corresponding norm (with the usual modifications when $q=\infty$). In other words, we may take $g_t$ as above, but this time $g_t\in \Mc^1(G)$ for every $t\in \supp \nu$. Then the above computations may be repeated without further issues since $\ee^{-s \Lc}g_t\in L^1(\beta)$ for every $s>0$ and  $t\in \supp \nu$. 
	
	The proof of (iii) for Triebel--Lizorkin spaces follows the same route (we purposefully  highlighted how to deal with the operators $T_{b,s}$ in the above argument even when it was superfluous for Besov spaces).

\smallskip	
	
We then prove (i) for Besov spaces. Since 
\[
T_m \Sc(G)\subseteq (\theta_\alpha L^{p,q}_0(\beta,\mi+\delta_0))\cap \Fc_m,
\]
it will suffice to take $f\in B^{p,q}_\alpha(\beta)$ such that $T_m f\in (\theta_\alpha L^{p,q}_0(\beta,\mi+\delta_0))\cap \Fc_m$ and prove that $f\in \mathring B^{p,q}_\alpha(\beta)$. Moreover, for $ s\to 0^+$
	\[
	\ee^{-s\Lc} \ee^{-\Lc} f\to \ee^{-\Lc}f, \qquad \ee^{-s\Lc} W^{(m)}_t f\to W^{(m)}_t f 	\]
	in $L^p(\beta)$, since $\ee^{-\Lc}f,W^{(m)}_t f\in L^p_0(\beta)$ for every $t>0$. Then, observe that by Corollary~\ref{cor:1} there is a constant $C_1>0$ such that
	\[
	\norm{(I-\ee^{-s\Lc})W^{(m)}_t f}_{L^p(\beta)}\meg C_1 \norm{ W^{(m)}_t f}_{L^p(\beta)}
	\] 
	for every $t>0$ and  $s\in (0,1)$, so that by dominated convergence (if $p<\infty$) or by equicontinuity (if $p=\infty$), we see that $\ee^{-s\Lc}f$ converges to $f$ in $B^{p,q}_\alpha(\beta)$.
	
	Now, observe that $\ee^{-s \Lc} f\in W^{\alpha',p}_0(\beta)$  for every $\alpha' \Meg 0$ and $s>0$, thanks to Corollary~\ref{cor:1}. The conclusion then follows from  (ii) and (iii) of Proposition~\ref{prop:7}.  

\smallskip
	
We finally prove (i) for Triebel--Lizorkin spaces. Assume first that $q\in (1,\infty)$. Then, $L^{q,p}(\mi+\delta_0,\beta)$ is reflexive by~\cite[Theorem 1 of \S\ 4]{BenedekPanzone}, so that also $(\theta_\alpha L^{q,p}(\mi+\delta_0,\beta))\cap \Fc_m$ is reflexive. Now, $T_m$ induces an isomorphism of $\mathring F^{p,q}_\alpha(\beta)$ onto a closed subspace of $(\theta_\alpha L^{q,p}(\mi+\delta_0,\beta))\cap \Fc_m$, so that $\mathring F^{p,q}_\alpha(\beta)$ is reflexive as well. By (iii), the dual of $\mathring F^{p,q}_\alpha(\beta)$ may be  identified with $F^{p',q'}_{-\alpha}(\beta)$, and the dual of the closed subspace $\mathring F^{p',q'}_{-\alpha}(\beta)$ may be identified with $F^{p,q}_\alpha(\beta)$. This shows that $F^{p,q}_\alpha(\beta)$ may be identified with a quotient of the bidual of $\mathring F^{p,q}_{\alpha}(\beta)$, which is reflexive. Consequently, $F^{p,q}_\alpha(\beta)=\mathring F^{p,q}_{\alpha}(\beta)$ (and $F^{p',q'}_{-\alpha}(\beta)=\mathring F^{p',q'}_{-\alpha}(\beta)$). By Lemma~\ref{lem:22}, this is sufficient to prove (i) in this case.\footnote{Notice that we could have followed a `direct' route as in the case of Besov spaces, using the maximal function $T^*_{b,1}$ in order to apply the dominated convergence theorem. Nonetheless, this procedure cannot be extended to the case $q\in \{1,\infty\}$, so that we preferred this indirect route for consistency.}
	
	Now, observe that $T_m$ induces an isomorphism of $\mathring F^{p,\infty}_{\alpha}(\beta)$ onto a closed subspace of $(\theta_{\alpha} L^{\infty,p}_0(\nu,\beta))\cap \Fc_m$, and that the bidual of $L^{\infty,p}_0(\nu,\beta)$ is canonically isomorphic to $L^{\infty,p}(\nu,\beta)$ since $\nu$ is discrete and $p\in (1,\infty)$.\footnote{One first shows that the dual of $L^{\infty,p}_0(\nu,\beta)$ is canonically isomorphic with $L^{1,p'}(\nu,\beta)$ since $\nu$ is discrete and $p<\infty$, proceeding as in the proof of~\cite[Theorem 1 of \S\ 3]{BenedekPanzone}. Then,~\cite[Theorem 1 of \S\ 3]{BenedekPanzone} shows that the dual of $L^{1,p'}(\nu,\beta)$ may be canonically identified with $L^{\infty,p}(\nu,\beta)$ since $p'<\infty$.} 
	Since $T_m$ is defined on the reflexive space $\Sc'(G)$, this proves that the bidual of $\mathring F^{p,\infty}_\alpha(\beta)$ may be identified with a closed subspace of $F^{p,\infty}_\alpha(\beta)$, with equality if and only if $T_m\mathring F^{p,\infty}_\alpha(\beta)= (\theta_{\alpha} L^{\infty,p}_0(\nu,\beta))\cap \Fc_m$. Now, by (iii), $F^{p',1}_{-\alpha}(\beta)$ may be identified with the dual of $\mathring F^{p,\infty}_\alpha(\beta)$, and $F^{p,\infty}_\alpha(\beta)$ may be identified with the dual of the closed subspace $\mathring F^{p',1}_{-\alpha}(\beta)$ of $F^{p',1}_{-\alpha}(\beta)$. This implies that $F^{p,\infty}_\alpha(\beta)=\mathring F^{p,\infty}_\alpha(\beta)$ and $F^{p',1}_{-\alpha}(\beta)=\mathring F^{p',1}_{-\alpha}(\beta)$. This proves (i) for $q=1$ thanks to Lemma~\ref{lem:22}. This also proves (i) for $q=\infty$ and $\mi+\delta_0=\nu$. The case of general $\mi$ follows since one may show, repeating the proof of Lemma~\ref{lem:26}, that $(\theta_{\alpha} L^{\infty,p}_0(\nu,\beta))\cap \Fc_m=(\theta_{\alpha} L^{\infty,p}_0(\mi+\delta_0,\beta))\cap \Fc_m$.	
\end{proof}

\begin{oss}\label{oss:1}
	Observe that Propositions~\ref{prop:5},~\ref{prop:3},~\ref{prop:1}, and~\ref{prop:2}, as well as Theorem~\ref{teo:3}, elementary extend to the spaces $\mathring B^{p,q}_\alpha(\beta)$ and $\mathring F^{p,q}_\alpha(\beta)$.
	
	Also Theorem~\ref{teo:5} extends to these spaces, even though in general 
	\[
	\Lc_\omega^{(\alpha-\alpha')/\dL+i \gamma}\Sc(G)\not \subseteq \Sc(G).
	\] 
	In fact, if $\omega$ is so large that the convolution kernel of $\Lc_\omega^{(\alpha-\alpha')/\dL+i \gamma}$ belongs to $\Delta_R^\delta L^1(\beta)$ for every $\delta\in [0,1]$ (which may happen), then $\Lc_\omega^{(\alpha-\alpha')/\dL+i \gamma}\Sc(G)\subseteq W^{\alpha'',p}_0(\beta)$ for every $\alpha''$ and $p\in [1,\infty]$, so that the assertion follows by means of Proposition~\ref{prop:7}.
\end{oss}

\section{Interpolation}\label{Interpolation:sec}

\begin{teo}\label{teo:9}
Suppose  $\theta\in (0,1)$, $\alpha,\alpha_0,\alpha_1\in \R$, $p,p_0,p_1,q,q_0,q_1\in [1,\infty]$, and
\[
\alpha_\theta=(1-\theta)\alpha_0+\theta \alpha_1, \qquad \frac{1}{p_\theta}=\frac{1-\theta}{p_0}+\frac{\theta}{p_1}, \qquad \frac{1}{q_\theta}=\frac{1-\theta}{q_0}+\frac{\theta}{q_1}.
\]
 Then, the following equalities hold with equivalence of norms.
	\begin{enumerate}
		\item[\textnormal{(1)}] $(B_{\alpha_0}^{p,
                    q_0}(\beta),B_{\alpha_1}^{p,
                    q_1}(\beta))_{\theta,q}= B_{\alpha_\theta}^{p,
                    q}(\beta)$, provided $\alpha_0\neq
                  \alpha_1$.
		
		\item[\textnormal{(2)}] $(B_{\alpha}^{p, q_0}(\beta),B_{\alpha}^{p, q_1}(\beta))_{\theta,q_\theta}= B_{\alpha}^{p, q_\theta}$.
		
		\item[\textnormal{(3)}] $(B_{\alpha_0}^{p_0, q_0}(\beta),B_{\alpha_1}^{p_1, q_1}(\beta))_{\theta,q_\theta}= B_{\alpha_\theta}^{p_\theta, q_\theta}(\beta)$ ($q_0,q_1<\infty$ and $p_\theta=q_\theta$). 
				
		\item[\textnormal{(4)}] $(B_{\alpha_0}^{p_0, q_0}(\beta),B_{\alpha_1}^{p_1, q_1}(\beta))_{[\theta]}= B_{\alpha_\theta}^{p_\theta, q_\theta}(\beta)$ ($q_\theta\neq \infty$ or $\alpha_0=\alpha_1$).
				
		\item[\textnormal{(5)}] $(F_{\alpha_0}^{p, q_0}(\beta),F_{\alpha_1}^{p, q_1}(\beta))_{\theta,q}= B_{\alpha_\theta}^{p, q}(\beta)$ ($\alpha_0\neq \alpha_1$, $p\in (1,\infty)$).
		 		
		\item[\textnormal{(6)}] $(F_{\alpha}^{p_0, q}(\beta),F_{\alpha}^{p_1, q}(\beta))_{\theta,p_\theta}= F_{\alpha}^{p_\theta,q}(\beta)$ ($p_0,p_1\in (1,\infty)$).
		
		\item[\textnormal{(7)}] $(F_{\alpha_0}^{p_0, q_0}(\beta),F_{\alpha_1}^{p_1, q_1}(\beta))_{\theta,q_\theta}= F_{\alpha}^{p_\theta,q_\theta}(\beta)$ ($p_0,p_1\in (1,\infty)$, $p_\theta=q_\theta$).
		
		\item[\textnormal{(8)}] $(F_{\alpha_0}^{p_0, q_0}(\beta),F_{\alpha_1}^{p_1, q_1}(\beta))_{[\theta]}= F_{\alpha_\theta}^{p_\theta, q_\theta}(\beta)$ ($p_0,p_1\in (1,\infty)$ and either $q_\theta\neq \infty$ or $\alpha_0=\alpha_1$).
	\end{enumerate}
\end{teo}

\begin{proof}
	The assertions follow from Lemmas~\ref{lem:22} and~\ref{lem:24},~\cite[Theorem 6.4.3]{BerghLofstrom}, and the corresponding interpolation results for the relevant mixed norm spaces:~\cite[Theorem 5.6.1]{BerghLofstrom} for (1) and (2),~\cite[Theorem 1.18.1]{Triebel} for (3),~\cite[Theorem 1.18.1 and the following Remark 2]{Triebel} and~\cite[Theorem  5.1.2]{BerghLofstrom} for (4), and~\cite[Theorems 1.18.1 and the following Remark 2, and Theorem 1.18.4]{Triebel}  for (6)--(8). It remains to prove (5). This follows from Proposition~\ref{prop:1} and (1): since 
	\[
	B_{\alpha_j}^{p,\min(p,q_j)}(\beta)\subseteq F_{\alpha_j}^{p,q_j}(\beta)\subseteq B_{\alpha_j}^{p,\max(p,q_j)}(\beta)
	\]
	continuously ($j=0,1$), one has 
\begin{align*}
	B_{\alpha_\theta}^{p,q}(\beta)
	&=(B_{\alpha_0}^{p, \min(p,q_0)}(\beta),B_{\alpha_1}^{p, \min(p,q_1)}(\beta))_{\theta,q}\\
	& \subseteq (F_{\alpha_0}^{p, q_0}(\beta),F_{\alpha_1}^{p, q_1}(\beta))_{\theta,q}\subseteq (B_{\alpha_0}^{p, \max(p,q_0)}(\beta),B_{\alpha_1}^{p, \max(p,q_1)}(\beta))_{\theta,q}=B_{\alpha_\theta}^{p,q}(\beta)
\end{align*}
	continuously, whence the conclusion.
\end{proof}

\begin{oss}\label{oss:3}
	Arguing as in the proof of Theorem~\ref{teo:9}, one may show that 
	\[
	(\mathring B_{\alpha_0}^{p_0, q_0}(\beta),\mathring B_{\alpha_1}^{p_1, q_1}(\beta))_{[\theta]}= \mathring B_{\alpha_\theta}^{p_\theta, q_\theta}(\beta), \qquad (\mathring F_{\alpha_0}^{p_0, q_0}(\beta),\mathring F_{\alpha_1}^{p_1, q_1}(\beta))_{[\theta]}= \mathring F_{\alpha_\theta}^{p_\theta, q_\theta}(\beta)
	\]
	($p_0,p_1\in (1,\infty)$). Thus, the spaces $\mathring B^{p,q}_\alpha(\beta)$ and $\mathring F^{p,q}_\alpha(\beta)$ behave more nicely with respect to complex interpolation,  as the restriction $q_\theta<\infty$ unless $\alpha_0=\alpha_1$ is not needed. The analogues of (1)--(3) and (5)--(7) also hold, under the same assumptions, replacing $B$ and $F$ with $\mathring B$ and $\mathring F$, respectively, and the interpolation functor $(\,\cdot\,,\,\cdot\,)_{\theta, \infty}$ with $(\,\cdot\,,\,\cdot\,)_{\theta,\infty,0}$, defined so that $(A,B)_{\theta,\infty,0}$ is the closure of the intersection $\Delta(A, B)$ in $(A,B)_{\theta,\infty}$.
      \end{oss}

\section{Proof of
  Proposition~\ref{lem:8bis}}\label{A1:sec}

	(ii) 	 Observe that we may assume that $m'\Meg 0$, since
        otherwise $W^{'(m'),*}_{t,\eps}=0$, and also that $t_1\Meg
        1+\max \supp \mi$, by the boundedness of $\ee^{-t\Lc}$ on
        $L^{p}$. We divide the proof into four steps:
	\begin{itemize}
	\item[\textsc{I}.] We prove~\eqref{eq:5} with $W'^{(m'),*}_{t,\eps}$ replaced by $W'^{(m')}_t= (t\Lc')^{m'}\ee^{-t \Lc'}$ under the assumption that  $\alpha+\lambda >0$, $t_0>0$, and   $m'\in \N$.
	\item[\textsc{II}.] Assuming that $\Lc=\Lc'$, $t_0>0$, and $\lambda=0$, we prove~\eqref{eq:5} with $W'^{(m'),*}_{t,\eps}$ replaced by the operator defined by
\begin{equation}\label{mstarstar}
	W^{(m'),**}_{t,\eps} f \coloneqq t^{m'} \max_{s\in [\eps t,
          t/\eps]} \max_{\substack{Y\in U_{m' \dL}\\ \abs{Y}\meg 1}} \abs{ Y \ee^{-s \Lc} f }, \qquad f\in \Sc'(G).
\end{equation}
	   \item[\textsc{III}.] We assume that $t_0>0$ and $m'\in \N$, and  prove~\eqref{eq:5} with $W'^{(m'),*}_{t,\eps}$ replaced by $W'^{(m')}_{t}$.
	    \item[\textsc{IV}.]  We complete the proof.
	\end{itemize}
	We now discuss the above steps in detail.
	
	\textsc{Step I} Assume that $\alpha+\lambda >0$, $t_0>0$, and   $m'\in \N$. Let us begin by noticing that by Lemma~\ref{lem:9} and a change of variables,
\begin{equation}\label{decfkey}
	\begin{split}
	f&= \sum_{h=0}^{m-1}\frac{1}{h!} W^{(h)}_{2 t_1} f+ \frac{1}{(m-1)!}\int_0^{2 t_1} W^{(m)}_s f\,\dd \mi_{2t_1}(s)\\
		&= \sum_{h=0}^{m-1}\frac{1}{h!} W^{(h)}_{2 t_1} f+ \frac{2^{m}}{(m-1)!}\int_0^{ t_1} \ee^{-s \Lc} W^{(m)}_s f\,\dd \mi_{t_1}(s),
	\end{split}
\end{equation}
whence
\begin{multline}\label{fdecstep1}
t^{-\alpha/\dL' }W'^{(m')}_{t} X f\\
 = \sum_{h=0}^{m-1}\frac{1}{h!} t^{-\alpha/\dL' }W'^{(m')}_{t} X W^{(h)}_{2 t_1} f+ \frac{2^{m}}{(m-1)!}\int_0^{ t_1} t^{-\alpha/\dL' }W'^{(m')}_{t} X\ee^{-s \Lc} W^{(m)}_s f\,\dd \mi_{t_1}(s).
\end{multline}
Therefore, it will be enough to estimate the norms of $\ee^{-t_0\Lc'}X f $ and of the terms
\begin{equation}\label{twoterms}
t^{-\alpha/\dL' }W'^{(m')}_{t} X W^{(h)}_{2 t_1} f  , \qquad \int_0^{ t_1} t^{-\alpha/\dL' }W'^{(m')}_{t} X \ee^{-s \Lc} W^{(m)}_s f\,\dd \mi_{t_1}(s)
\end{equation}
separately. Let us start from the second term in~\eqref{twoterms}, and observe that for all $g\in\mathcal{S}'(G)$
	\[
	\big\|\norm{t^{-\alpha/\dL'}(W'^{(m')}_{t} X g)(x)}_{L^q_t(\mi)}\big\|_{L^p_x(\beta)}=\big\| \norm{t^{-\alpha/\dL }(W'^{(m')}_{t^{\dL'/\dL} } X g)(x)}_{L^q_t(\phi_*\mi)}\big\|_{L^p_x(\beta)},
	\]
where $\phi(t)=t^{\dL/\dL'}$, and notice that $\phi_*\mi\in \Mc_\Car$ by Lemma~\ref{lem:20}.  

If now $(h'_t)$ is the heat kernel associated with $\Lc'$, by Theorem~\ref{teo:7} we may take $C_1,b>0$ so that 
\begin{align}
	\abs{ Y \Lc'^{k'}\Lc^k h_t}&\meg C_1\abs{Y}  t^{-(\deg Y+k d +k'\dL')/\dL} p_{b,t,\dL} \label{eqC1},\\
	\abs{Y^{\dag R\dag}  \Lc'^{k'}   h'_t}&\meg C_1\abs{Y}   t^{-(\deg Y+k' \dL')/\dL'} p_{b,t,\dL'} \label{eqC2}
\end{align}
	for all $k=0,\dots, m$, all $k'=0,\dots, m'$, all $Y\in U_{\lambda+(m+1)d+(m'+1)\dL'}$,  and all $t\in (0,2t_1]$. Observe that for $s\in (0,2t_{1}]$, by~\eqref{eqC2} 
	\[
	\begin{split}
	\abs{ \Lc'^{m'} \ee^{-t^{\dL'/\dL}\Lc'} X \ee^{-s\Lc} W^{(m)}_s f}
	& = \abs{ X [(\ee^{-s\Lc} W^{(m)}_s f) *\Lc'^{m'} h'_{t^{\dL'/\dL}} ]}\\
	& = \abs{(\ee^{-s\Lc} W^{(m)}_s f) *X^{\dag R\dag} \Lc'^{m'} h'_{t^{\dL'/\dL}} }\\
	&\meg C_1  \abs{X}  t^{-(m' \dL'+\lambda)/\dL} T_{b,t^{\dL'/\dL},\dL'}\abs{\ee^{-s\Lc} W^{(m)}_s f}\\
		&\meg  C_1^2 \abs{X}  t^{-(m' \dL'+\lambda)/\dL} T_{b,t^{ \dL'/\dL},\dL'} T_{b, s,\dL}  W^{(m)}_s f,
	\end{split}
	\]
 and analogously, by~\eqref{eqC1}
	\[
	\begin{split}
		\abs{   \Lc'^{m'}\ee^{-t^{\dL'/\dL}\Lc'}   X \ee^{-s\Lc} W^{(m)}_s f}&\meg C_1   T_{b, t^{\dL'/\dL},\dL'}\abs{  \Lc'^{m'} X\ee^{-s\Lc} W^{(m)}_s f}\\
		&\meg  C_1^2  \abs{X}  s^{-(m'\dL'+\lambda)/\dL} T_{b,t^{\dL'/\dL},\dL'} T_{b,s,\dL}  W^{(m)}_s f.
	\end{split}
	\]
	Then,  since $\min(s^{-1},t^{-1}) \leq 2 (s+t)^{-1}$, there is a constant $C_2>0$ such that
		\[
	\abs{W'^{(m')}_{t^{\dL'/\dL}} X \ee^{-s\Lc} W^{(m)}_s f}\meg C_2 \abs{X}  t^{m'\dL'/\dL} (t+s)^{-(m' \dL'+\lambda)/\dL} T_{b,t^{\dL'/\dL},\dL'} T_{b,s,\dL}   W^{(m)}_s f.
	\]
	Then, Fubini's theorem and Lemma~\ref{lem:2} show that (notice that $\mu \in  \Mc_\Car$ has bounded support, hence if $\kappa(t)= t^{\dL'/\dL}$ and $\omega$ is that of Lemma~\ref{lem:2}, then $\ee^{-\omega \kappa(t)}$ is bounded from above and below $\mu$-a.e.) there is a constant $C_3>0$ such that 
	\[
	\begin{split}
		&\bigg\| \norm*{t^{-\alpha/\dL }\left(W'^{(m')}_{t^{\dL'/\dL}} X\int_0^{t_1}\ee^{-s\Lc} W^{(m)}_s f\,\dd \mi_{t_1}(s)\right)(x)}_{L^q_t(\phi_*\mi)}\bigg\|_{L^p_x(\beta)}\\
		&\meg C_2 \abs{X}  \bigg\|\norm*{t^{\frac{m' \dL'-\alpha}{\dL} } T_{b,t^{\dL'/\dL},\dL'}\left( \int_0^{t_1} \!\!(t+s)^{-\frac{m'\dL'+\lambda}{d}} T_{b,s,\dL} W^{(m)}_s f\,\dd \mi_{t_1}(s)\right)(x) }_{L^q_t(\phi_*\mi)}\bigg\|_{L^p_x(\beta)}\\
		&\meg C_3 \abs{X}   \bigg\|\norm*{t^{\frac{m' \dL'-\alpha}{\dL} }  \int_0^{t_1} (t+s)^{-\frac{m'\dL'+\lambda}{\dL}}  (T_{b,s,\dL} W^{(m)}_s f)(x)\,\dd \mi_{t_1}(s) }_{L^q_t(\phi_*\mi)}\bigg\|_{L^p_x(\beta)}.
	\end{split}
	\]
Since $\lambda+\alpha>0$, as well as $m'\dL'-\alpha>0$ by assumption, and
	\[
	t^{\frac{m' \dL'-\alpha}{\dL} } (t+s)^{-(m'\dL'+\lambda)/\dL} = \frac{t^{\frac{m'\dL'-\alpha}{\dL}} s^{\frac{\alpha + \lambda}{\dL}}}{(t+s)^{\frac{m'\dL'+\lambda}{\dL}}}  s^{-\frac{\alpha + \lambda}{\dL}},
	\]
we may use Lemma~\ref{lem:25} and then Lemmas~\ref{lem:2} and~\ref{lem:26}, to see that there is a constant $C_4>0$ such that

	\[
	\begin{split}
		& \bigg\|\norm*{t^{\frac{m' \dL'-\alpha}{\dL} }  \int_0^{t_1} (t+s)^{-(m'\dL'+\lambda)/\dL}  (T_{b,s,\dL} W^{(m)}_s f)(x)\,\dd \mi_{t_1}(s) }_{L^q_t(\phi_*\mi)}\bigg\|_{L^p_x(\beta)}\\
		&\qquad\qquad\meg C_4     \big\|\norm{  {s^{-(\alpha+\lambda)/\dL}} (T_{b,s,d}W^{(m)}_s f)(x)  }_{L^q_s(\mi_{t_1})}\big\|_{L^p_x(\beta)}\\
		&\qquad\qquad\meg C_4^2   \big\| \norm{   {s^{-(\alpha+\lambda)/\dL}}   (W^{(m)}_s f)(x)  }_{L^q_s(\mi_{t_1})}\big\|_{L^p_x(\beta)}\\
		&\qquad\qquad\meg  C_4^3  \big\|\norm{   {s^{-(\alpha+\lambda)/\dL}}   (W^{(m)}_s f)(x)  }_{L^q_s(\nu)}\big\|_{L^p_x(\beta)}.
	\end{split}
	\]
 Let us now estimate the norm of the first terms in~\eqref{twoterms}. For $h=0,\dots, m-1$,  by~\eqref{eqC1}, 
	\[
	\begin{split}
		\abs{ \Lc'^{m'}\ee^{-t \Lc'} X W^{(h)}_{2 t_1} f}&\meg C_1 (2 t_1)^h T_{b,t,\dL'}   \Lc'^{m'} X \Lc^h \ee^{-(2 t_1)\Lc}f \\
		&\meg C_1^2 \abs{X}   (2 t_1)^h t_1^{-(m'\dL'+\lambda+h\dL)/\dL} T_{b,t,\dL'} T_{b,t_1,\dL}  \ee^{-t_1\Lc}f ,
	\end{split}
	\]
	so that there is a constant $C_5>0$ such that
	\[
	\abs{W'^{(m')}_{t}   X W^{(h)}_{2 t_1} f}\meg  C_5 \abs{X}  t^{m'}T_{b,t,\dL'} T_{b,t_1,\dL}  \ee^{-t_1\Lc}f .
	\]
	Consequently, by Lemma~\ref{lem:2}, applied twice, we see that there are constants $C_6,C_{7}>0$ such that 
	\[
	\begin{split}
		\big\| \norm{t^{-\alpha/\dL'}(W'^{(m')}_{t} X W^{(h)}_{2 t_1} f)(x)}_{L^q_t(\mi)}\big\|_{L^p_x(\beta)}&\meg C_6 \abs{X}  \big\|  \norm{ t^{m'-\alpha/\dL'}( \ee^{-t_1\Lc}f)(x)   }_{L^q_t(\mi)}  \big\|_{L^p_x(\beta)}\\
		&\meg C_6\abs{X} \norm{ t^{m'-\alpha/\dL'}   }_{L^q_t(\mi)} \norm{\ee^{-t_1\Lc} f}_{L^p(\beta)}\\
		& {\leq C_7\abs{X} \norm{\ee^{-t_1\Lc} f}_{L^p(\beta)},}
	\end{split}
	\]
 where the last inequality holds since $\norm{ t^{m'-\alpha/\dL'}   }_{L^q_t(\mi)} $ is finite (cf.~the proof of Lemma~\ref{lem:25}).  We have thus estimated the norms of the two terms in~\eqref{twoterms}, whence by~\eqref{fdecstep1} there is a constant $C_8>0$ such that
	\[
	\begin{split}
	&\big\|\norm{t^{-\alpha/\dL'}(W'^{(m') }_{t} X f)(x)}_{L^q_t(\mi)}\big\|_{L^p_x(\beta)} \\
	& \qquad \qquad \meg C_8 \Big( \abs{X} \norm{\ee^{-t_1\Lc} f}_{L^p(\beta)}+ \big\|  \norm{      t^{-(\alpha+\lambda)/\dL}  (W^{(m)}_t f)(x)}_{L^q_t( \nu)}  \big\|_{L^p_x(\beta)}\Big).
		\end{split}
	\]
	Applying the previous inequality with $(m',\alpha,\lambda,\mi)$ replaced by $(0,\alpha-m' \dL', m'\dL'+\lambda,\delta_{t_0})$, we then infer that there is a constant $C_9>0$ such that
	\[
	\begin{split}
		\norm{\ee^{-t_0\Lc'} X f}_{L^p(\beta)}&=	 t_0^{\alpha/\dL'-m'}
		\big\|\norm{t^{m'-\alpha/\dL'}(\ee^{-t \Lc'} X f)(x)}_{L^q_t(\delta_{t_0})}\big\|_{L^p_x(\beta)}\\
		&\meg C_9 \Big( \abs{X} \norm{\ee^{-t_1\Lc} f}_{L^p(\beta)}+   \big\|  \norm{      t^{-(\alpha+\lambda)/\dL}  (W^{(m)}_t f)(x)}_{L^q_t( \nu)}  \big\|_{L^p_x(\beta)}\Big)
	\end{split}
	\] 
	for every $X\in U_{m' \dL'+\lambda}$ (in particular, for every $X\in U_\lambda$), whence the conclusion of the first step.

	\textsc{Step II}  Assume now that $\Lc=\Lc'$, $t_0>0$, and $\lambda=0$. Consider the operator $W^{(m'),**}_{t,\eps} $ defined in~\eqref{mstarstar}.  We shall use the decomposition~\eqref{decfkey} as in the previous step.
		
	By~\eqref{eqC1}, for every $Y\in U_{m' \dL}$, with $\abs{Y}\meg 1$, every $s'\in [\eps t,t/\eps]$, and $t,s\in (0,2 t_1]$,
	\[
	\begin{split}
	\abs{Y \ee^{- s'\Lc} W^{(m)}_{s} f}&=s^m\abs{  Y\ee^{-(s+s')\Lc}\Lc^m f }\\
		&\meg C_1 s^m (s'-\eps^2 t+ s(1-\eps^2))^{-\deg Y/\dL} T_{b,  s'-\eps^2 t+ s(1-\eps^2) ,\dL} \, \ee^{-\eps^2 (t+s)\Lc}  \Lc^m f\\
		&\meg C_1 s^m (s+t)^{-\deg Y/\dL}
                (\eps^2-\eps^3)^{-Q_*/\dL-m'}  T_{b, (t+s)/\eps,\dL} \, \ee^{-\eps (t+s)\Lc}  \Lc^m f,
	\end{split}
	\]
	where we used that
	\[
	(\eps-\eps^2)(t+s)  \meg s'-\eps^2 t+ s(1-\eps^2)\meg (t+s)/\eps .
	\]
	Then, there is a constant $C_{10}>0$ such that
	\[
	 W^{(m'),**}_{t,\eps} W^{(m)}_{s} f\meg C_{10} t^{m'} s^m (s+t)^{-m'} T_{b, (t+s)/\eps,\dL} \ee^{-\eps (t+s)\Lc}  \Lc^m f
	\]
	for every $t,s\in (0,2 t_1]$.	
	Then, if $m\geq 1$, by Lemmas~\ref{lem:25bis} and~\ref{lem:2} there is $C_{11}>0$ such that\footnote{Notice that, in order to apply Lemma~\ref{lem:25bis}, we have to interpret the integral with respect to $\mi_{2 t_1}$ as an integral with respect to the corresponding Haar measure. Consequently, we truncate the integrand accordingly and thus we obtain  an estimate in terms of the $L^q(\mi_{3 t_1})$ norm since $\mi$ is supported in $(0,t_1]$.}
		\[
	\begin{split}
	&\bigg\|  \norm*{t^{-\alpha/\dL}\left(W^{(m'),**}_{t,\eps} \int_0^{2 t_1}W^{(m)}_{s} f\,\dd \mi_{2 t_1}(s)\right)(x)}_{L^q_t(\mi)} \bigg\|_{L^p_x(\beta)}\\
		&\qquad\meg C_{10}\bigg\| \norm*{  t^{m'-\alpha/\dL}\int_0^{2 t_1} s^m(s+t)^{-m'}(T_{b, (t+s)/\eps,\dL} \ee^{-\eps (t+s)\Lc}  \Lc^m f)(x)\,\dd \mi_{2 t_1}(s) }_{L^q_t(\mi)} \bigg\|_{L^p_x(\beta)}\\
		&\qquad \meg C_{11}\big\| \norm{   s^{m-\alpha/\dL}(T_{b, s,\dL} \ee^{-\eps s\Lc}  \Lc^m f)(x) }_{L^q_s(\mi_{3 t_1})} \big\|_{L^p_x(\beta)}\\
		&\qquad\meg  C_{11}^2\big\| \norm{   s^{m-\alpha/\dL}( \ee^{-\eps s\Lc}  \Lc^m f)(x) }_{L^q_s(\mi_{3 t_1})} \big\|_{L^p_x(\beta)}\\
		&\qquad\meg  C_{11}^2 \eps^{-2(m-\alpha/\dL)}\big\| \norm{   s^{m-\alpha/\dL}( \ee^{- s\Lc}  \Lc^m f)(x) }_{L^q_s(\mi_{3\eps^2 t_1})} \big\|_{L^p_x(\beta)},
	\end{split}
	\]
	which may be estimated as desired by means on Lemma~\ref{lem:26} as in the previous step.  The terms 
	\[
	\big\|  \norm{t^{-\alpha/\dL}\big(W^{(m'),**}_{t,\eps} W^{(h)}_{2t_1} f \big)(x)}_{L^q_t(\mi)} \big\|_{L^p_x(\beta)},  \qquad h=0,\dots, m-1,
	\]
	may be estimated as in~\textsc{step I}.	
	Applying these arguments to $\mi=\delta_{t_0}$, we also get the desired estimate of $\norm{ \ee^{-t_0 \Lc}f}_{L^p(\beta)}$. 
	
	We are left with dealing with the case $m=0$. Because of the previous arguments, we may reduce to estimating 
	\[
	\big\|  \norm{t^{-\alpha/\dL}(W^{(1)}_{t}  f )(x)}_{L^q_t(\mi_{2 t_1})} \big\|_{L^p_x(\beta)} .
	\] This is done by means of Lemma~\ref{lem:2}, since $\abs{W^{(1)}_t f}\meg 2 C_1 T_{b,t/2} \ee^{-(t/2)\Lc} f$ for every $t\in (0,2 t_1]$. %\textcolor{blue}{Ma perché considerare $m=0$? Nel passo 1 mi sembra che implicitamente fosse $m\geq 1$...?}
	
	\textsc{Step III} Assume now that $t_0>0$ and $m'\in \N$. Take
        $k\in \N$ with $k \dL+\lambda+\alpha>0$, and        recall that there is $\omega\in \R$ such that
        $\Lc_\omega^{-1}$ induces an endomorphism of $L^p(\beta)$
        (Proposition~\ref{Smulders}). By means of
        Lemma~\ref{lem:27}, we then deduce that there is $g\in
        \Sc'(G)$ such that $  \Lc_\omega^{k} g=f$, provided that
        $\big\|\norm{t^{-(\alpha+\lambda)/\dL}(W^{(m)}_{t }
          f)(x)}_{L^q_t(\nu )}  \big\|_{L^p_x(\beta)}$ is finite (as
        we may assume).\footnote{Notice that, since
          $\Lc^{-k}_\omega\delta_e$ need not decay more than
          exponentially, $\Lc_\omega^{-k}$ does not in general induce
          an endomorphism of $\Sc'(G)$ (its adjoint need not preserve
          $\Sc(G)$), even though $\Lc^{-k}_\omega f$ may be defined
          for every fixed $f$ provided that $\omega$ is sufficiently
          large. Since we need to fix $\omega$ independently of $f$,
          we have to follow a different route.}  
	Observe that we may reduce to the case in which $m>\alpha/\dL$ and $\omega=0$, thanks to~\textsc{step II}, applied with $m'$ replaced by any $m''\in \N$ such that $m''>\alpha/\dL$ -- in fact, observe that clearly $\abs{(t \Lc_\omega)^{m''} \ee^{-t \Lc_\omega}g} \meg\abs{\Lc_\omega^{m''}} \ee^{-\omega t}W^{(m''),**}_{t,\eps}g$ for every $g\in \Sc'(G)$. 
	 
	Then, applying~\textsc{step I} to  $(X\Lc^{k},  m+k, g) $ in place of $(X,  m, f)$, we see that there is a constant $C_{12}>0$ such that
	\[
	\begin{split}
		&\norm{\ee^{-t_0\Lc'} Xf}_{L^p(\beta)}+ \big\|\norm{t^{-\alpha/\dL'}(W'^{(m')}_{t} X f)(x)}_{L^q_t(\mi )}  \big\|_{L^p_x(\beta)}\\
		&\qquad=\norm{\ee^{-t_0\Lc'} X\Lc^{k}g}_{L^p(\beta)}+\big\|\norm{t^{-\alpha/\dL'}(W'^{(m')}_{t} X \Lc^{k}g)(x)}_{L^q_t(\mi )}  \big\|_{L^p_x(\beta)}\\
		&\qquad\meg C_{12}\abs{X}  \Big( \norm{\ee^{-t_1 \Lc}g}_{L^p(\beta)}+\big\|\norm{t^{-(\alpha+\lambda+k)/\dL}(W^{(m+k)}_{t }   g)(x)}_{L^q_t(\nu )}  \big\|_{L^p_x(\beta)}\Big)\\
		&\qquad\meg C_{12}\abs{X}  \Big( \norm{\Lc^{-1}}^k_{  L^p(\beta) \to L^p(\beta) }\norm{\ee^{-  t_1 \Lc}f}_{L^p(\beta)}+ \big\|\norm{t^{-(\alpha+\lambda)/\dL}(W^{(m)}_{t }  f)(x)}_{L^q_t(\nu )}  \big\|_{L^p_x(\beta)}\Big),
	\end{split}
	\]
	whence the conclusion. 
	
	\textsc{Step IV} We now consider the general case. Take $m''\in\N$ such that $m''>\alpha/\dL'$. Observe that~\textsc{steps II} (applied with $(\Lc,\alpha,m',m,f)$ replaced by $(\Lc', \alpha-\deg Y, m'-\deg Y/\dL',m'',YXf)$) and~\textsc{III} (applied with $(\alpha,m',\lambda,X)$ replaced by $(\alpha-\deg Y,m'',\lambda+\deg(Y),YX)$) show that there is a constant $C_{13}>0$ such that
	\[
	\begin{split}
	&\big\|\norm{t^{(\deg Y-\alpha)/\dL'}(W'^{(m'-\deg Y/\dL'),**}_{t,\eps}Y X f)(x)}_{L^q_t(\mi )}  \big\|_{L^p_x(\beta)}\\
		&\qquad\meg C_{13} \Big( \norm{\ee^{-t_1 \Lc}Y X f}_{L^p(\beta)}+  \big\|\norm{t^{-(\deg Y-\alpha)/\dL'}(W'^{(m'')}_{t} Y X f)(x)}_{L^q_t(\mi )}  \big\|_{L^p_x(\beta)}\Big)\\
		&\qquad\meg C_{13}^2 \Big(\abs{X} \norm{\ee^{-  t_1 \Lc}f}_{L^p(\beta)}+\abs{X}  \big\|\norm{t^{-(\alpha+\lambda)/\dL}(W^{(m)}_{t }  f)(x)}_{L^q_t(\nu )}  \big\|_{L^p_x(\beta)}\Big),
	\end{split}
	\]
	for every $Y\in U_{m' \dL'}$ with $\abs{Y}\meg 1$ and $X\in U_\lambda$. By the arbitrariness of $Y$, this proves that there is a constant $C_{14}>0$ such that
	\begin{multline*}
		\big\|\norm{t^{-\alpha/\dL'}(W'^{(m'),*}_{t,\eps}X f)(x)}_{L^q_t(\mi )}  \big\|_{L^p_x(\beta)}\\
		 \meg C_{14} \abs{X} \Big(\norm{\ee^{-  t_1 \Lc}f}_{L^p(\beta)}+  \big\|\norm{t^{-(\alpha+\lambda)/\dL}(W^{(m)}_{t }  f)(x)}_{L^q_t(\nu )}  \big\|_{L^p_x(\beta)}\Big).
\end{multline*}
	If $t_0>0$, then applying the previous arguments to the case $\mi=\delta_{t_0}$ one shows that there is a constant $C_{15}>0$ such that
	\[
	\begin{split}
	\norm{\ee^{-t_0\Lc'} Xf}_{L^p(\beta)}&\meg t_0^{\alpha/\dL'-m'} \big\|\norm{t^{-\alpha/\dL'}(W'^{(m'),*}_{t,\eps}X f)(x)}_{L^q_t(\delta_{t_0} )}  \big\|_{L^p_x(\beta)}\\
		& \meg C_{15} \abs{X} \Big(\norm{\ee^{-  t_1 \Lc}f}_{L^p(\beta)}+  \big\|\norm{t^{-(\alpha+\lambda)/\dL}(W^{(m)}_{t }  f)(x)}_{L^q_t(\nu )}  \big\|_{L^p_x(\beta)}\Big).
	\end{split}
	\]
	Finally, assume that $t_0=0$, so that $\alpha>0$ by the assumptions. Then, by means of Corollary~\ref{cor:1} we see that there is a constant $C_{16}>0$ such that
	\[
	\begin{split}
		&\norm{X f}_{L^p(\beta)}\meg\sum_{h=0}^{m-1}\frac{1}{h!} \norm{X W^{(h)}_{2 t_1} f}_{L^p(\beta)}+ \frac{2^m}{(m-1)!}\norm*{\int_0^{t_1}  (X\ee^{-s \Lc} W^{(m)}_s f)(x)\,\dd \mi_{t_1}(s) }_{L^p_x(\beta)}\\
			&\meg C_{16} \norm{\ee^{-t_1\Lc} f}_{L^p(\beta)}+ 2^m C_1\norm*{\int_0^{ t_1} s^{-\lambda} (T_{b,s,\dL}W^{(m)}_s  f)(x)\,\dd \mi_{t_1}(s) }_{L^p_x(\beta)}\\
			&\meg C_{16} \norm{\ee^{-t_1\Lc} f}_{L^p(\beta)}+ 2^m C_1 t_1^{\lambda-\lambda}\norm{s^{\alpha}}_{L^{q'}_s(\mi_{t_1})}\big\|\norm{ s^{-(\alpha+\lambda)} (T_{b,s,\dL}W^{(m)}_s  f)(x)}_{L^q_s(\mi_{t_1})} \big\|_{L^p_x(\beta)},
	\end{split}
	\]
	so that the assertion follows from Lemmas~\ref{lem:2} and~\ref{lem:26}.
\qed

\section{Proof of Proposition~\ref{min-basis:pro}}\label{A2:sec}
		\textsc{Step I} Define $\pi\colon \gf\to \gf_*$ so that $\pi(X)\coloneqq X+\gf_{\deg(X)^-}$ for every $X\in \gf$. Observe that, even though $\pi$ need not be a homomorphism of Lie algebras (or even a linear mapping), it still satisfies the following properties:
		\begin{enumerate}
			\item[(1)] if $X'_1,\dots, X'_k\in \gf$, $\deg(X'_1)=\cdots=\deg(X'_k)$, and $\alpha_1,\dots, \alpha_k\in \R$, then
			\[
			\pi(\alpha_1 X'_1+\cdots+\alpha_k X'_k)= \alpha_1 \pi(X'_1)+\cdots+\alpha_k \pi(X'_k),
			\]
			\emph{provided that $\alpha_1 \pi(X'_1)+\cdots+\alpha_k \pi(X'_k)\neq 0$} (which is equivalent to saying that $\deg(\alpha_1 X'_1+\cdots+\alpha_k X'_k)=\deg(X'_1)$);
			
			\item[(2)] if $X'_1,\dots, X'_{k+1}\in \gf$, then
			\[
			\pi(\ad(X'_1)\cdots \ad(X'_k) X'_{k+1})= \ad(\pi(X'_1))\cdots \ad(\pi(X'_k)) \pi(X'_{k+1})
			\]
			\emph{provided that $\ad(\pi(X'_1))\cdots \ad(\pi(X'_k)) \pi(X'_{k+1})\neq0$} (which is equivalent to saying that \[
			\deg(\ad(X'_1)\cdots \ad(X'_k) X'_{k+1})=\deg(X'_1)+\cdots+\deg(X'_{k+1})).
			\]
		\end{enumerate}
		The first assertion follows easily from the definition. The second assertion follows from the case $k=1$ by induction, while the case $k=1$ follows from the definitions, since $[\pi(X'_1),\pi(X'_2)]=[X_1',X_2']+\gf_{(\deg(X_1)+\deg(X_2))^-}$.

		\textsc{Step II} For every $\lambda \Meg 1$, define $\Jc_\lambda$ as the set of $(j_1,\dots, j_{k+1})\in J^{k+1}$ ($k\in \N$) such that $\deg(X_{j_1})+\cdots +\deg(X_{j_{k+1}})\meg \lambda$, and define $\gf'_\lambda$ as the vector space generated by the $\ad(X_{j_1})\cdots \ad(X_{j_k}) X_{j_{k+1}}$ as $(j_1,\dots, j_{k+1})$ runs through $\Jc_\lambda$. Then, clearly $\gf'_\lambda \subseteq \gf_\lambda$ and the second part of the statement amounts to saying that $\gf'_\lambda=\gf_\lambda$ for every $\lambda \Meg 1$. Since $\gf_\lambda=\gf$ for $\lambda$ sufficiently large, this implies that $(X_j)$ generates $\gf$ as a Lie algebra.
		Let us then prove that $\gf'_\lambda=\gf_\lambda$ for
                every $\lambda\Meg 1$. Define $d_j\coloneqq \deg(X_j)$
                and $Y_j\coloneqq \pi(X_j)$ for every $j\in J$, so
                that $(Y_j+[\gf_*,\gf_*])_{j\in J}$ is a homogeneous
                basis of $\gf_*/[\gf_*,\gf_*]$; in particular, since $\gf_*$ is homogeneous, $(Y_j)$ generates
                $\gf$ as a Lie algebra. Let $(Y_j)_{j\in J'}$ be a homogeneous basis of $[\gf_*,\gf_*]$, so that $(Y_j)_{j\in J\cup J'}$ is a homogeneous basis of $\gf_*$, and define $d_j\coloneqq \deg(Y_j)$ for every $j\in J'$. If we choose $X_j\in \gf$ so that $\pi(X_j)=Y_j$ for every $j\in J'$, then one may prove, by induction, that $(X_j)_{j\in J\cup J'}$ is a basis of $\gf$, necessarily adapted to the filtration $(\gf_\lambda)$  (cf.~Proposition \ref{da-aggiungere:prop}).  
		Now, take $j\in J'$ and observe that, since $Y_j\in [\gf_*,\gf_*]$ and since $(Y_j)_{j\in J}$ generates $\gf_*$ as a Lie algebra, we may find a subset $J_j$ of $\Jc_{d_j}$ and a family $(\alpha_{j_1,\dots, j_{k+1}})$ of elements of $\R$ such that
		\[
		Y_j= \sum_{(j_1,\dots, j_{k+1})\in J_j} \alpha_{j_1,\dots, j_{k+1}}\ad(Y_{j_1})\cdots \ad(Y_{j_{k}}) Y_{j_{k+1}}.
		\]
		We may also assume that $\alpha_{j_1,\dots, j_{k+1}}\ad(Y_{j_1})\cdots \ad(Y_{j_{k}}) Y_{j_{k+1}}\neq 0$ for every $(j_1,\dots, j_{k+1})\in J_j$, so that~\textsc{step I} shows that
		\[
		X_j- \sum_{(j_1,\dots, j_{k+1})\in J_j} \alpha_{j_1,\dots, j_{k+1}}\ad(X_{j_1})\cdots \ad(X_{j_{k}}) X_{j_{k+1}}\in \gf_{d_j^-},
		\]
		so that
		\[
		X_j\in \gf'_{d_j}+\gf_{d_j^-}.
		\]
		Since clearly $\gf_{\dd_0}=\gf'_{\dd_0}$, by induction we may assume that $\gf_{d_j^-}=\gf'_{d_j^-}\subseteq \gf'_{d_j}$. The assertion then follows easily.
	\qed

\end{document}